\documentclass[11pt,reqno]{amsart}
\usepackage[bookmarks=true,%
    colorlinks=true,%
    linkcolor=blue,%
    citecolor=blue,%
    filecolor=blue,%
    menucolor=blue,%
    urlcolor=blue]{hyperref}
\usepackage[hmarginratio=1:1]{geometry}
\usepackage{caption}
\usepackage{subcaption}
\usepackage{import}
\usepackage{amsmath,mathrsfs,bm}
\usepackage{amssymb,amsthm}
\usepackage{mathtools}
\usepackage{graphicx}
\usepackage{color,transparent}
\usepackage{xcolor}
\usepackage{float}
\usepackage{tikz-cd}
\usepackage{svg}
\allowdisplaybreaks
\usepackage[toc,page]{appendix}
\usepackage[colorlinks=true]{hyperref}
\hypersetup{
    colorlinks=true,
    linkcolor=purple,
    citecolor=cyan
    }

\usepackage[multiple]{footmisc}
\let\oldFootnote\footnote
\newcommand\nextToken\relax

\renewcommand\footnote[1]{%
    \oldFootnote{#1}\futurelet\nextToken\isFootnote}

\newcommand\isFootnote{%
    \ifx\footnote\nextToken\textsuperscript{,}\fi}
\let\epsilon\varepsilon

\newcommand{\abs}[1]{\left|#1\right|}
\newcommand{\R}{{\mathbb{R}}}
\newcommand{\Z}{\mathbb{Z}}
\newcommand{\C}{\mathbb{C}}
\newcommand{\N}{\mathbb{N}}

\newcommand{\T}{\mathsf{T}} % for tangent space

\newcommand{\identity}{\mathsf{id}}
\newcommand{\inp}[1]{\left\langle#1\right\rangle} %Inner product
\newcommand{\eps}{\varepsilon}

\newcommand{\GL}{\mathsf{GL}}

\newcommand{\st}{\,\colon\,} % Such that for sets

\newcommand{\brac}[1]{\left(#1\right)}

\newcommand{\D}{\mathsf{d}} %differential

\newcommand{\interior}[1]{\mathsf{int}\brac{#1}}
\newcommand{\trace}{\mathsf{tr}}
\newcommand{\Ad}{\mathsf{Ad}}

\newcommand{\Reps}[1]{\mathscr R(#1,G)}
\newcommand{\CharVar}[1]{\mathscr X(#1,G)}
\newcommand{\vv}{\mathbf{v}}
\newcommand{\yy}{\mathbf{y}}
\newcommand{\leaf}{\mathscr L}

\newcommand{\PSL}[1]{\mathsf{PSL}(#1,\R)}
\newcommand{\SL}[1]{\mathsf{SL}(#1,\R)}
\newcommand{\flags}{\mathscr F}

\newcommand{\Hom}{\mathsf{Hom}}

\newcommand{\Hit}[1]{\mathsf{Hit}(#1,S)}

\newcommand{\ub}[1]{\underline{#1}}
\newcommand{\B}{\mathfrak{B}}
\newcommand{\gfr}{\mathfrak{g}}
\newcommand{\afr}{\mathfrak{a}}
\newcommand{\gfrk}{\gfr^{\oplus k}}

\newcommand{\HGroupCohom}[1]{H^{#1}(\pi,\gfr_{\Ad\rho})}

\newcommand{\HGroup}[1]{H_{#1}(\pi,\gfr_{\Ad\rho})}
\newcommand{\HSurfaceCohom}[1]{H^{#1}(S,\gfr_{\Ad\rho})}
\newcommand{\Hm}{\mathsf{H}} % for Hamiltonian vector field

\newcommand{\HReal}[1]{H_{#1}(S,\R)}

\newcommand{\HCoReal}[1]{H^{#1}(S,\R)}

\newcommand{\coai}[2]{#1_{o(y_{#2})}} %%% For graph of groups decomposition connecting paths
\newcommand{\ctai}[2]{#1_{t(y_{#2})}}

\newcommand{\proxGTheta}{\mathscr P_\Theta}
\newcommand{\bdryMap}{\flags_\rho}
\newcommand{\ThetaAnosovReps}{\mathscr X_\Theta(\pi,G)}
\newcommand{\figeight}{\delta}
\newcommand{\LiftTeich}{\widetilde{\mathcal T}(S)}
\newcommand{\Teich}{\mathcal T(S)}
\newcommand{\pml}{\mathcal{PML}(S)}
\newcommand{\tw}{\mathsf{tw}} 

\newcommand{\Sp}[1]{\mathsf{Sp(2#1,\R)}}

\newtheorem*{theorem*}{Theorem}
\newtheorem{theorem}{Theorem}[section]
\newtheorem{corollary}[theorem]{Corollary}
\newtheorem{proposition}[theorem]{Proposition}
\newtheorem{lemma}[theorem]{Lemma}
\newtheorem{thmA}{Theorem}

\newtheorem{corA}[thmA]{Corollary}

\theoremstyle{definition}
\newtheorem{notation}[theorem]{Notation}
\newtheorem{definition}[theorem]{Definition}
\newtheorem{remark}[theorem]{Remark}
\newtheorem{example}[theorem]{Example}
\newtheorem{convention}[theorem]{Convention}

\renewenvironment{proof}{{\bfseries Proof }}{\qed\\}

\theoremstyle{remark}

\title[Invariant multi--functions and Hamiltonian flows]{Invariant multi-functions and Hamiltonian flows for surface group representations}

\author[F. Camacho--Cadena]{Fernando Camacho--Cadena}
\address{Max Planck Institute for Mathematics in the Sciences, Inselstr. 22, 04103 Leipzig, Germany\vspace{.2cm}}
\email{fernando.camacho@mis.mpg.de}
\author[J. Farre]{James Farre}
\email{james.farre@mis.mpg.de}
\author[A. Wienhard]{Anna Wienhard}
\email{anna.wienhard@mis.mpg.de}
\thanks{This project was funded by the Deutsche Forschungsgemeinschaft (DFG, German Research Foundation) – Project-ID 281071066 – TRR 191. A.W. and F.C.C. were supported by European Research Council under ERC-Advanced Grant 101018839, A.W. thanks the Klaus Tschira Stiftung and the Hector Fellow Academy for support. F.C.C acknowledges support from the U.S. National Science Foundation grants DMS 1107452, 1107263, 1107367 "RNMS: Geometric Structures and Representation Varieties" (the GEAR Network), and from HITS gGmbH (Isabel Rojas Travel Award 2021).
}

\begin{document}

\maketitle

\begin{abstract}
    Goldman defined a symplectic form on the smooth locus of the $G$-character variety of a closed, oriented surface $S$ for a Lie group $G$ satisfying very general hypotheses.  
    He then studied the Hamiltonian flows associated to $G$-invariant functions $G \to \mathbb R$ obtained by evaluation on a simple closed curve and proved that they are \emph{generalized twist flows.}
    In this article, we investigate the Hamiltonian flows on (subsets of the) $G$-character variety induced by evaluating a $G$-invariant multi-function $G^k \to \mathbb R$ on a tuple $ \ub \alpha \in \pi_1(S)^k$.
    We introduce the notion of a \emph{subsurface deformation} along a \emph{supporting subsurface} $S_0$ for $\ub \alpha$ and prove that the Hamiltonian flow of an induced invariant multi-function is of this type.
    We also give a formula for the Poisson bracket between two functions induced by invariant multi-functions and prove that they Poisson commute if their supporting subsurfaces are disjoint. 
   We give many examples of functions on character varieties that arise in this way and discuss applications, for example, to the flow associated to the trace function for non-simple closed curves on $S$.
\end{abstract}
\tableofcontents

\section{Introduction}
The character variety\footnote{By character variety, we will always mean the space of closed orbits of the action by conjugation of the Lie group on the space of representations.} of the fundamental group of a surface into a reductive Lie group (or at least its smooth part) carries a natural symplectic structure \cite{SymplecticNature_Goldman}. Thus, there is a natural Poisson structure on the space of smooth functions, and any smooth function on the character variety gives rise to a Hamiltonian vector field. The associated Hamiltonian flows provide important ways to move in the character variety or within subsets thereof.\\ 

Probably the best known example for this is the Fricke--Teichm\"uller space of an oriented closed surface $S$ of genus $g\geq 2$ , which can be realized as a connected component of the character variety $\Teich \subset \Hom(\pi_1(S), \mathsf{PSL}(2,\mathbb{R}))// \mathsf{PSL}(2,\mathbb{R})$. The Weil--Petersson form gives a symplectic structure on $\mathcal{T} (S)$ \cite{Wolpert_Symplectic}.\\

Important functions on $\Teich$ are hyperbolic length functions. The length function associated to an element $\gamma \in \pi_1(S)$ associates to a point in $\Teich$ the length of the closed geodesic in the free homotopy class of $\gamma$ with respect to the hyperbolic metric determined by this point. Wolpert proved that the Hamiltonian vector field associated to the length function of a simple closed curve $\alpha$ is the Fenchel-Nielsen twist vector field. He further gives an explicit formula for the Poisson pairing of length functions, and deduces that Fenchel-Nielsen length and twist parameters $(\ell_i, \theta_i)_{i= 1, \dots, 3g-3}$ associated to a pair of pants decomposition of $S$ provide Darboux coordinates, leading to the magic formula for the Weil-Petersson form \cite{Wolpert_FNtwist,Wolpert_Symplectic,Wolpert_TheFNDef}: 
$$
\omega = \sum_{i=1}^{3g-3} d\ell_i \wedge d\theta_i. 
$$
A further consequence of Wolpert's work is that the action of a Dehn twist along a simple closed curve $\alpha$ is realized as the autonomous Hamiltonian diffeomorphism (i.e. the time 1 map of a Hamiltonian flow) of the function $\frac{1}{2} \ell_\alpha^2$. It has been recently shown by Farre that not only the Dehn twist, but also pseudo-Anosov diffeomorphisms can be realized as autonomous Hamiltonian diffeomorphisms of functions involving length functions of measures and transverse H\"older distributions on laminations \cite{Farre:hamiltonian}. \\

Goldman broadly generalized this picture. In \cite{SymplecticNature_Goldman} he introduced the symplectic structure on character varieties of fundamental groups of surfaces into any reductive Lie group $G$. Building upon the relation between the length function $\ell_\gamma$ associated to $\gamma \in \pi_1(S)$ for a hyperbolic structure on $S$ and the trace of the element $\rho(\gamma) \in \mathsf{SL}(2,\mathbb{R})$, where $\rho$ is a lift of the holonomy representation of the hyperbolic structure, he investigates invariant functions and their Hamiltonian flows in \cite{InvFct_Goldman}. 
Any function $f:G \to \mathbb{R}$ which is invariant under the action of $G$ on itself by conjugation induces functions on the character variety $\CharVar{\pi_1(S)}\coloneqq\Hom(\pi_1(S),G)//G$. Namely fixing an element $\gamma \in \pi_1(S)$, the function $f_\gamma: \CharVar{\pi_1(S)} \rightarrow \mathbb{R}$ is defined by $f_\gamma([\rho]) = f(\rho(\gamma))$.\\

Goldman obtains an explicit formula for the Poisson bracket of two such invariant functions, which in fact descends to a Lie algebra structure on the set of free homotopy classes of closed curves on $S$ $-$ the Goldman algebra. For a simple closed curve $\alpha$, Goldman shows that the Hamiltonian flow of $f_\alpha$ is covered by a generalized twist flow, which we will describe in more detail later. For closed curves with self-intersection, no nice geometric description of the Hamiltonian flows has been known. We provide some geometric information about such Hamiltonian flows as an application of our more general considerations. \\

In this article we consider invariant multi-functions, their associated functions on character varieties and prove several results about the structure of the associated Hamiltonian flows. 
An \emph{invariant multi-function} is a (smooth) function $f: G^k \to \mathbb{R}$ (or more generally to a field $\mathbb{K}$) that is invariant by the diagonal action of $G$ on $G^k$ by conjugation. Given a invariant multi-function, and a $k$-tuple $\underline{\gamma} = (\gamma_1, \cdots , \gamma_k)$ of elements of $\pi_1(S)$ we obtain a function $f_{\underline {\gamma}}: \CharVar{\pi_1(S)} \rightarrow \mathbb{R}$ defined by the rule  $f_{\underline {\gamma}}([\rho]) = f(\rho(\gamma_1), \cdots, \rho(\gamma_k))$. We give explicit formulas for the Poisson bracket of two such invariant multi-functions and show that the Hamiltonian flows associated to these multi-functions correspond to \emph{subsurface deformations} associated to the subsurface that \emph{supports $\underline{\gamma}$}. In the case when $k=1$ and the curve is a simple closed curve, this subsurface is a cylinder, and we recover Goldman's result. When $k=1$, but the curve has self-intersection, the supporting subsurface will  be more complicated, and our result gives a qualitative geometric description of the Hamiltonian flow. \\

We now describe the results in more detail. 

\subsection{Invariant multi-functions}
Let $G$ be a reductive Lie group, and let $\B$ be the Killing form on the Lie algebra $\gfr$ of $G$. 
An invariant function on $G^k$ is a $C^1-$ smooth function $f\colon G^k\to \R$, that is invariant under the diagonal action of $G$ on $G^k$ by conjugation. Given a 
$k$-tuple $\underline{\gamma} = (\gamma_1, \dots , \gamma_k)$ of elements of $\pi_1(S)$, we obtain an induced function on the character variety by setting 
$f_{\underline {\gamma}}: \CharVar{\pi_1(S)} \rightarrow \mathbb{R}$, defined as  $f_{\underline {\gamma}}([\rho]) = f(\rho(\gamma_1), \dots, \rho(\gamma_k))$. \\

We want to stress that it is important that $\underline{\gamma}$ is indeed a $k$-tuple of elements of the fundamental group of $S$, so $\gamma_1, \dots , \gamma_k$ are based loops, all based at the same base point and {\bf not} free homotopy classes of loops. 
Since we consider functions on $G^k$ that are only invariant by the diagonal action of $G$, not by the action of $G^k$, the common base point matters. Only in the case when $k=1$, we can pass to free homotopy classes of loops. \\

Our aim in this article is to investigate the Hamiltonian vector fields and the Hamiltonian flows associated to these induced functions. 
Passing from just invariant functions on $G$ to invariant multi-functions on $G^k$ vastly enhances the class of functions on $\CharVar{\pi_1(S)}$ one can realize as induced invariant functions $f_{\underline {\gamma}}$. Let us give some examples; see Section~\ref{sec : examples}:

\begin{enumerate}
\item When $G<\GL(n,\mathbb K)$ any trace function $\trace_\gamma: \CharVar{\pi_1(S)}\to \mathbb K$ arises as induced invariant functions of the trace $\trace:G \to \mathbb{K}$. Those trace functions have been investigated in 
\cite{InvFct_Goldman}. It is immediate that any polynomial in trace functions can be realized as an induced invariant function of a multi-function. 
\item A similar idea can be used to write the length function of measured laminations not as a limit of length functions of simple closed curves, but as a function induced from an invariant multi-function.
\item 
Invariant functions induced from multi-functions can be useful to give a better understanding of invariant functions $f_\gamma: \CharVar{\pi_1(S)} \to \R$ associated to a non-simple closed curve $\gamma$. Writing $\gamma$ as a product of simple closed curves on $S$, e.g. $\gamma = \alpha \cdot\beta $ we can realize $f_\gamma$ as an invariant function $f_{(\alpha,\beta)}$ arising form an invariant multicurve. This can also be used to realize the trace function of a web as a function induced by an invariant multi-function. We make use of this in Section \ref{sec : self intersecting curves} to give a geometric interpretation of the Hamiltonian flows of trace functions associated to the self-intersecting figure eight curve on a pair of pants. 
\item It is of interest also to consider functions that are not defined on the entire character variety $\CharVar{\pi_1(S)}$, but only on subsets. We focus in particular on functions on Hitchin components and $\Theta-$Anosov representations and show that many of those functions, for example triple ratios and cross ratios  \cite{FockGoncharov}, coordinate functions defined by Bonahon-Dreyer \cite{BonahonDreyer_FiniteLaminations, BonahonDreyer_GeneralLaminations}, correlation functions \cite{Labourie_SwappingAlgebra,SimpleLengthRigidity_BCL,GhostPolygons}, and the Darboux coordinates of \cite{DarbouxOnHitchin_SunZhang, FlowsPGLVHitchin_SWZ} all arise as induced invariant functions. 
\end{enumerate}

In order to get a description of the Hamiltonian vector fields and flows of functions induced by invariant multi-functions, a key step is to determine the Poisson bracket between two such functions. \\

For this we associate to an invariant multi-function $f\colon G^k\to \R$ its variation function $F\colon G^k\to \gfrk$, and denote by $F_i\colon G^k\to \gfr$ its $i-$th component. The variation functions $F_i$ measure the variation of $f$ in the $i-$th component. 
It is implicitly defined by 
\[
\B(F_i(g_1, \dots, g_k),X)=\frac{d}{dt}\bigg|_{t=0}f(g_1,\dots,g_{i-1},\exp(tX)g_i,g_{i+1},\dots,g_k),
\]
for $X\in\gfr$. The Poisson bracket between two such induced invariant functions is then nicely described in terms of their variation functions. 

\begin{thmA}[The product formula (Theorem \ref{thm: product formula})]\label{thm:product}
   Let $f\colon G^k\to \R$ and $f'\colon G^n\to\R$ be invariant multi-functions, and let $F$ and $F'$ be their variation functions.  Let $\ub{\alpha}=(\alpha_1,\dots,\alpha_k)\in\pi_1(S)^k$ and $\ub{\beta}=(\beta_1,\dots,\beta_n)\in\pi_1(S)^n$. The Poisson bracket $\{f_{\ub \alpha},f'_{\ub \beta}\}$ is the function $\CharVar{\pi_1(S)} \to\R$ given by
   \begin{equation*}
   \{f_{\ub \alpha},f'_{\ub \beta}\}([\rho])\mapsto \sum_{i,j}\sum_{q\in\beta_j\#\alpha_i}\eps(q;\beta_j,\alpha_i)\B(F'_j(\rho(\ub\beta))(q),F_i(\rho(\ub\alpha))(q)),
   \end{equation*}
   where $\beta_j\#\alpha_i$ is the set of transverse intersections between $\beta_j$ and $\alpha_i$, and
   \begin{align*}
    F_i(\rho(\ub \alpha)) &\coloneqq F_i(\rho(\alpha_1),\dots,\rho(\alpha_k)),\\
     F'_j(\rho(\ub \beta)) &\coloneqq F'_j(\rho(\beta_1),\dots,\rho(\beta_n)).
\end{align*}
\end{thmA}

The statement looks exactly like Goldman's formula when $k=n=1$ \cite[Theorem 3.5]{InvFct_Goldman}, but some difficulty is hidden in the statement and we refer the reader to Theorem~\ref{thm: product formula} for a precise statement. In the formula, the $k$-tuple of curves $\ub{\alpha}$ and $\ub{\beta}$ have to be chosen with respect to different base points, and the Lie algebra elements $ F_i(\rho(\ub \alpha))$ and $ F'_j(\rho(\ub \beta)) $ are interpreted as sections of a flat bundle.

\subsection{Subsurface deformations}

In order to describe the Hamiltonian vector fields and flows of functions on the character variety that are associated to invariant multi-functions, we need to consider decompositions of surfaces and associated {\em subsurface deformations} in the representation variety. \\

The fundamental group of a surface can be decomposed in many ways. For example, given a separating simple closed curve $\gamma \in \pi_1(S)$ that separates $S$ into two surfaces $S_1$ (to the left of $\gamma$) and $S_2$ (to the right of $\gamma$), the fundamental group can be written as an amalgamated product of the fundamental groups of $S_1$ and $S_2$ over the cyclic group generated by $\gamma$
\[
\pi_1(S) = \pi_1(S_1)*_{\langle\gamma\rangle} \pi_1(S_2).
\]
If $\gamma$ is non-separating, there is a similar decomposition of $\pi_1(S)$ as an HNN-extension.\\ 

The fundamental group of a surface can be more generally decomposed by writing it as a graph of groups (see Section~\ref{sec : Constructing graph of groups}). Of particular importance to us is the graph of groups decomposition associated to a   decomposition of $S = \bigcup_{i = 0}^m S_i$  into subsurfaces with boundary. 
The decomposition into subsurfaces gives a realization of $\pi_1(S)$ as the fundamental group of a graph of groups, where the vertex groups are the fundamental groups $\pi_1(S_i)$ and the edge groups arise from the gluing pattern along the boundary curves of the $S_i$'s. The decomposition associated to a simple closed curve is a special case of this, namely the case when the subsurface $S_0$ is the cylinder arising as a tubular neighborhood of the simple closed curve $\gamma$.\\

Given such a decomposition of $S = \bigcup_{i = 0}^m S_i$, we introduce the notion of a {\em subsurface deformation} of a representation $\rho: \pi_1(S) \to G$ along $S_0$. For this we consider the subset ${\leaf}_{\rho}(S_0) \subset \CharVar{\pi_1(S)}$, which consists of all (equivalence classes of) representations $[\rho'] \subset \CharVar{\pi_1(S)}$ such that $[\rho'|_{\pi_1(S_i)}] = [\rho|_{\pi_1(S_i)}]$ for $i= 1, \cdots, m$, that is, the restricted representations lie in the same conjugacy classes.\\ 

A continuous path of representations $[\rho_t]$ is then said to be a subsurface deformation along $S_0$ if $[\rho_t]$ is contained in $\leaf_{\rho_0}(S_0)$ for all $t$. 
A vector field on $\CharVar{\pi_1(S)}$ is said to be an infinitesimal subsurface deformation along $S_0$ if it is tangent to $\leaf_{\rho_0}(S_0)$. \\

Our main result then says that the Hamiltonian vector field of any function $f_{\underline{\gamma}}$ associated to an invariant multi-function is an infinitesimal subsurface deformation along a supporting subsurface (Definition \ref{def : supportin subsurface}) of $\underline{\gamma}$, and under additional smoothness conditions, the corresponding Hamiltonian flow is a subsurface deformation. 

\begin{thmA}[Theorem \ref{thm : hamiltonian vector field is infinitesimal subsurface deformation} and Theorem \ref{thm : Hamiltonian flow is a subsurface deformation}]\label{Theorem_subsurface}
	Let $f\colon G^k\to \R$ be an invariant function, $\ub\alpha\in\pi_1(S)^k$, and $S_0$ a supporting subsurface for $\ub\alpha$. 
 
 \begin{enumerate} 
 \item \label{thm: intro infinitesimal subsurface}The Hamiltonian vector field $\Hm f_{\ub\alpha}$ of $f_{\ub\alpha}$ is an infinitesimal subsurface deformation along $S_0$, i.e.
	\[
	\Hm f_{\ub\alpha}\in\T_{[\rho]}\leaf_\rho(S_0)
	\]
	for every $[\rho]\in \CharVar{\pi_1(S)}$.
\item \label{thm : intro subsurface def}
Assume there is an open subset $\mathscr O\subset\CharVar{\pi_1(S)}$ such that the decomposition into leaves $\leaf^\mathscr O(S_{0})$ is a smooth foliation of $\mathscr O$. Then the Hamiltonian flow of $f_{\ub\alpha}$ is a subsurface deformation along $S_{0}$.
\end{enumerate}
\end{thmA}

Essentially, a subsurface deformation arises from a deformation in the relative character variety of $S_0$. In general, relative character varieties are quite complicated, and thus we cannot give more geometric information on the subsurface deformation. However, when $S_0$ is a separating $m$-holed sphere, we give a geometric realization of the Hamiltonian flow of $f_{\ub\alpha}$ similar to Goldman's generalized twist flows. This is possible since the representation of the fundamental group of an $m$-holed sphere is completely determined by the boundary curves. 

\begin{corA}[Corollary \ref{cor : hamiltonian flow for fully separating subsurface}]\label{cor:fullysep}
	Assume that the supporting surface $S_0$ is a fully separating $m$-holed sphere, such that $S_1, \cdots, S_m$ are the connected components of the complement, where $S_i$ and $S_0$ are glued along a boundary curve  $c_i$. Then there exists a family of paths $g^t_1,\dots,g^t_m\in G$ with $g_i^0 = e$ such that for every $\gamma\in\Gamma_i=\pi_1(S_i)$, the path $\Xi_\rho\colon I(\rho)\to\Hom(\pi_1(S),G)$ of representations given by
	\[
		\Xi_{\rho}(t)(\gamma) = g_i^t\rho_i(\gamma)(g_i^t)^{-1},
	\]
	for $i = 1,\dots,m$ covers the Hamiltonian flow of $f_{\ub\alpha}$ restricted to $\Gamma_i$. In particular, we have that
	\[
	\Xi_{\rho}(t)(c_i) = g_i^t\rho_0(c_i)(g_i^t)^{-1}.
	\]
\end{corA}
Let us point out that this is a generalization of \cite[Theorem 4.5]{InvFct_Goldman}, because any separating simple closed curve admits a supporting subsurface that is an cylinder, i.e. a separating 2-holed sphere.\\ 

It is in general difficult to write explicit paths of conjugating matrices $g_i^t$; even though we have computed a few explicit examples, it is an open question how to extract information on the $g_i^t$ out of the function $f$ and its variation function $F$. \\

For a general point, the smoothness condition in Theorem~\ref{Theorem_subsurface} might not be satisfied. There are however examples where we can ensure the smoothness condition for all points in a given connected component of the character variety. We prove that for the Hitchin components $\Hit d$ when $G = \SL d$ satisfy the smoothness condition and thus obtain the following consequence. 

\begin{corA}[Corollary \ref{cor : subsurface deformations in Hitchin}]\label{cor:Hitchin}
	Let $f\colon \SL d^k\to\R$ be an invariant function, let $\ub\alpha\in\pi^k$, and let $S_{0}$ be a supporting subsurface.  Then the Hamiltonian flow of $f_{\ub\alpha}$ restricted to $\Hit d$ is a subsurface deformation along the supporting subsurface $S_{0}$.
\end{corA}

\subsection{Hamiltonian flow associated to self-intersecting curves}
We discuss now applications to the Hamiltonian flows arising from an invariant function evaluated at a curve $\gamma$ with self-intersections. \\

As a first application, we get an extension of Goldman's explicit formula for the Hamiltonian flows associated to powers of simple closed curves, which is the case when the supporting subsurface is a cylinder.  Namely, they are generalized twist flows, \`a la Goldman (Theorem \ref{thm : Hamiltonian flow of power of simple closed curve}). 
\\

The simplest example of a \emph{primitive} self-intersecting curve is the \emph{figure-eight curve}  $\gamma$ supported on a (separating) pair of pants $S_0\subset S$.  That is, $\gamma$ is freely homotopic to $ab^{-1}\in \pi_1(S_0)= \langle a, b, c ~\vert ~abc = 1\rangle$; see Figure \ref{fig : figure 8 curve}.
The induced invariant function 
\[\trace_\gamma : \mathscr X(\pi_1(S),\SL d) \to \R\]
can be ``factored'' as an induced invariant multi-function $f_{(a,b)}$, where
\begin{align*}
    f\colon \SL d^2&\to\R\\
    (g,h)&\mapsto \trace(gh^{-1})
\end{align*}
By Theorem \ref{thm : intro subsurface def} and Corollary \ref{cor:fullysep}, we find that the Hamiltonian flow of $f_{(a,b)}$ is a subsurface deformation along the pair of pants $S_0$, determined entirely by paths of conjugations in the boundary curves.  We give an explicit formula for a  cocycle representing the corresponding Hamiltonian vector field, see Equation \eqref{eq : cocycle fully separating pair of pants figure 8 curve} in Section \ref{sec : self intersecting curves}. \\

In the case that $G = \SL2$, we can say much more.  
Indeed, for $\eta: \pi_1(S_0) \to G$,  the trace identity 
\[\trace(
\eta(ab^{-1})) = \trace(\eta(a))\trace(\eta(b)) - \trace(\eta(ab))\]
holds, where all three curves $a$, $b$ and $ab$ represent pairwise homotopically disjoint distinct simple closed separating curves.\footnote{In general, one can use the $\SL2$-trace identity to ``resolve'' the self-intersections of a non-simple curve to obtain an integer polynomial in traces of simple curves, but the simple curves appearing in this polynomial need not be disjoint.  }
We deduce that the Hamiltonian flow of $\trace_\gamma$ is defined for all time and is covered by a composition of pairwise commuting twist flows with speeds that depend on the initial point in $\mathscr X (\pi_1(S), \SL2)$; see Corollary \ref{cor : Hamiltonian flow of figure 8 in Teich}. \\

The function $f_\gamma$ and its flow descend to $\mathcal T(S)$, and from the formula we conclude that flow lines have well defined limiting behavior along orbits in forward and backward time to a point in the open $2$-simplex of positive weights on the curves $a$, $b$, and $c$  in Thurston's boundary (Corollary \ref{cor : limits in thurstons compactification}).  
As the limit depends on the initial hyperbolic structure, this behavior is reminiscent but distinct from the behavior of earthquake flow lines in simple, weighted multi-curves.

\subsection*{Structure of the paper}
In Section \ref{sec : Prelim} we recall some background on the character variety, the Atiyah-Bott-Goldman symplectic form and cohomology with local coefficients. Then in Section \ref{sec : invariant functions}, we introduce invariant functions and their associated variation functions. The content of Section \ref{sec : Poisson bracket} is the proof of the product formula Theorem \ref{thm:product} (Theorem \ref{thm: product formula}). The main tool is the computation of the Poincar\'e dual to the Hamiltonian vector field of an invariant function, see Proposition \ref{prop: explicit cycle for Poincare dual of Hamiltonian vector field}. In Section \ref{sec : examples}, we explain how certain functions on character varieties, or subsets thereof, can be realized as induced invariant functions. We note that Section \ref{sec : examples} will not be used elsewhere, as it discusses applications of our setting of invariant functions.\\

The rest of the paper is devoted to understanding the Hamiltonian vector field and the Hamiltonian flow of an invariant function. For this, we recall graphs of groups decompositions and define a decomposition of character varieties into leaves in Section \ref{sec : graph of groups}. In Section \ref{sec : subsurface deformations} we define (infinitesimal) subsurface deformations (Definition \ref{def : inf and subsurface deformation}) and prove that Hamiltonian flows of invariant functions are of this type (Theorem \ref{thm : hamiltonian vector field is infinitesimal subsurface deformation} and Theorem \ref{thm : Hamiltonian flow is a subsurface deformation}). The paper finishes with applying our framework to study Hamiltonian flows associated to self-intersecting curves in Section \ref{sec : self intersecting curves}. As a further application, we describe in Corollary \ref{cor : limits in thurstons compactification} the limit points in Thurston's compactification of the twist along the figure $8$ curve in Teichm\"uller space.

\subsection*{Acknowledgements}
We would like to thank Bill Goldman and Tengren Zhang for enlightening conversations.  A.W. thanks Yael Karshon for discussions on Goldman's symplectic form, which motivated parts of this project.

\section{Preliminaries}\label{sec : Prelim}
Here we collect some facts about the character variety and the Atiyah-Bott-Goldman symplectic form.
\subsection{The local structure of character varieties}
Fix once and for all a closed oriented surface $S$ of genus $g\geq 2$, as well as a universal covering $\tilde S\to S$ to which we associate the group of deck transformations $\pi$. We follow Sections 1.2 and 1.3 of \cite{SymplecticNature_Goldman}; see also \cite[Section 1]{JohnsonMillson}. Let $G$ be a real Lie group with center $Z(G)$ and whose Lie algebra $\gfr$ admits an \textit{orthogonal structure} $\B$, i.e., a non-degenerate symmetric bilinear form $\B\colon\gfr\times\gfr\to\R$ that is invariant under the adjoint action $\Ad_G$. \\

The group $\mathsf{Inn}(G) \cong G/Z(G)$ of inner automorphisms acts on $\Hom(\pi,G)$.  
We are interested in the smooth points in the quotient.
More precisely, denote by $\Reps \pi\subset\Hom(\pi,G)$ the subset of representations $\rho$ satisfying
\begin{itemize}
    %\item $\dim Z(\rho)/ Z(G) = 0$, where $Z(\rho)$ is the centralizer of $\rho$;
    \item $ Z(\rho) = Z(G)$, where $Z(\rho)$ is the centralizer of $\rho$;
    \item the image of $\rho$ does not lie in a proper parabolic subgroup of $G$.
\end{itemize}
Then 
\[
	\CharVar{\pi}\coloneqq \Reps\pi/G\coloneqq \Reps\pi/\mathsf{Inn}(G),
\]
is an analytic manifold. When $G$ carries a (real) algebraic structure, then $\CharVar{\pi}$ includes into the  smooth points of a certain \emph{$G$-character variety}, which is a real semi-algebraic set; see \cite{RichardsonSlodowy}.
In this paper, we are only interested in the smooth structure of $\CharVar{\pi}$.\\

The (Zariski) tangent space to an equivalence class of representations $[\rho]\in \CharVar{\pi}$ is represented by the cohomology groups 
\[
	H^1(\pi,\gfr_{\Ad\rho}) = \frac{Z^1(\pi,\gfr_{\Ad\rho})}{B^1(\pi,\gfr_{\Ad\rho})}.
\]
Here, $Z^1(\pi,\gfr_{\Ad\rho})$ is the set of \emph{cocycles}, i.e., functions $u\colon\pi\to\gfr$ satisfying the cocycle condition 
\[
u(\alpha\beta) = u(\alpha)+\Ad_{\rho(\alpha)}(u(\beta))
\]
for every $\alpha,\beta\in\pi$. 
The set $B^1(\pi,\gfr_{\Ad\rho})$ denotes the set of \emph{coboundaries}, i.e., cocycles $u\colon\pi\to\gfr$ for which there exists $u_0\in\gfr$ such that
\[
u(\alpha) = \Ad_{\rho(\alpha)}(u_0) - u_0
\]
for every $\alpha\in\pi$.
\subsection{The Atiyah-Bott-Goldman symplectic form}\label{sec : symplectic structure on character varieties}
We now recall the construction of the Atiyah-Bott-Goldman symplectic form \cite{AtiyahBott,SymplecticNature_Goldman} and follow the construction from \cite[Section 1.4]{SymplecticNature_Goldman}. At an equivalence class of representations $[\rho]\in \CharVar{\pi}$, Goldman defined the pairing $\omega_{[\rho]}\colon \HSurfaceCohom{1}\times\HSurfaceCohom{1}\to\R$ given by the following diagram

\[
\begin{tikzcd}
  \HSurfaceCohom{1}\times\HSurfaceCohom{1} \arrow[r, "\omega_{[\rho]}"] \arrow[d, "\smile" left]
    & \R \\
  H^2(S,\gfr_{\Ad\rho}\otimes\gfr_{\Ad\rho}) \arrow[r, "\B_*" below] &\arrow[u, "\frown {[S]}" right]
 H^2(S,\R) \end{tikzcd}
\]

where $\gfr_{\Ad\rho}$ is the $\pi$ module $\gfr$ through the action of $\pi$ by the adjoint representation, $\smile$ and $\frown$ are the usual cup and cap products in group cohomology (see \cite[Chapter V]{Brown_CohomologyOfGroups} for these definitions). The pairing is nondegenerate, skew-symmetric and bilinear. Goldman showed that this is a symplectic form on the smooth locus, that is, the differential $2$-form $\omega$ is closed.\\

The symplectic form gives a way to assign a vector field to $C^1$ functions as follows. For any smooth function $f\colon \CharVar\pi\to\R$, there is a unique vector field $\Hm f$ defined by 
\[
    \D f(\cdot)=\omega(\cdot,\Hm f)
\]
called the \textit{Hamiltonian vector field of $f$}, whose flow is called the \textit{Hamiltonian flow of }$f$. \\

Moreover, the form $\omega$ gives rise to a Poisson bracket $\{\cdot,\cdot\}$ on the space $C^\infty(\CharVar{\pi})$ of smooth functions given by \[
    \{f,f'\}=\D f(\Hm f') = \omega (\Hm f',\Hm f) = -\D f'(\Hm f)
\]
for any $f,f'\in C^\infty(\CharVar{\pi})$.

\subsection{(Co)homology with local coefficients}\label{sec : cohomology with local coefficients}
An alternative description of the local structure of the character variety is through cohomology with local coefficients on a flat bundle over the surface. This cohomology is isomorphic to the twisted group cohomology $H^*(\pi,\gfr_{\Ad\rho})$. The cohomology with local coefficients description will be useful when computing the Hamiltonian vector field of certain functions.\\

We recall the background following Section $2$ of \cite{InvFct_Goldman}. Let $V$ be a vector space, and let $\phi\colon\pi\to\GL(V)$ be a representation. Out of this data, we may build a flat vector bundle $\xi_\phi\coloneqq \tilde S\times_\phi V$, which is defined as the quotient of $\tilde S\times V$ by the action $\gamma\cdot (q,v) = (\gamma\cdot q,\phi(\gamma)(v))$ for $\gamma\in\pi$. This is a flat vector bundle over $S$ which comes equipped with a flat connection, namely the one induced by the trivial connection on $\tilde S\times V$. Let $C_*(S,\xi_\phi)$ be the complex of smooth singular chains on $S$ with values in $\xi_\phi$. A basis for the complex is given by elements of the form $\sigma\otimes s$, where $\sigma\colon \Delta^k\to S$ is a smooth map from the standard $k$-simplex $\Delta^k$, and $s$ is a flat section of the bundle $\sigma^*\xi_\phi$ over $\Delta^k$. Since simplices are simply connected, a flat section is given by an element in the fiber over a point (and then extending the section using parallel transport). Thus from now on, we will adopt the following notation.
\begin{notation}\label{notation : flat sections}
  We will write chains as $\sigma\otimes X_{\tilde r}$, where $X\in V$ and $\tilde{r}\in\tilde S$ is a lift of $r = \sigma(0)$, the first vertex of the simplex $\sigma$. We will often abuse notation and write $\sigma\otimes X$ for the flat section $\sigma\otimes X_{\tilde r}$ when there is no ambiguity. Moreover, we will often write $X(q)$ to refer to the section $\sigma\otimes X_{\tilde r}$ at the point $q$. 
\end{notation}
The chains come equipped with boundary maps defined as follows. For $\sigma\otimes s\in C_{k}(S,\xi_\phi)$,
\[
	\partial(\sigma\otimes s) = \sum_{i = 0}^k(-1)^{i}\partial_i\sigma\otimes s_i,
\]

where $\partial_i\sigma$ is the $i-$th face of $\sigma$, and $s_i$ is the restriction of $s$ to the $i-$th face.\\

The boundary map turns $C_*(S,\xi_\phi)$ into a chain complex. The cycles are denoted by $Z_1(S,\xi_\phi)$, the boundaries by $B_1(S,\xi_\phi)$, and the homology is  $H_k(S,\xi_\phi) = Z_k(S,\xi_\phi)/B_k(S,\xi_\phi)$.
Note that $H_k(S,\xi_\phi)$ is isomorphic to $H_k(\pi,V_\phi)$, where $V_\phi$ denotes the $\pi$ module $V$ with the action of $\pi$ induced by the representation $\phi$.\\

The cohomology of $S$ with coefficients in the bundle $\xi_\phi$ is defined similarly. A $\xi_\phi-$valued cochain on $S$ is a function which assigns to a $k$-simplex $\sigma$ a flat section of $\sigma^*\xi_\phi$ over $\Delta^k$. The space of cochains is denoted by $C^k(S,\xi_\phi)$. The complex $C^*(S,\xi_\phi)$ comes with coboundary maps turning the complex into a cochain complex. The cocycles are denoted by $Z^1(S,\xi_\phi)$, the coboundaries by $B^1(S,\xi_\phi)$, and the cohomology is $H^k(S,\xi_\phi) = Z^k(S,\xi_\phi)/B^k(S,\xi_\phi)$, which is isomorphic to $H^k(\pi,V_\phi)$.\\

These (co)homologies come equipped with cap and cup products. Let $\xi_1,\xi_2$ and $\xi_3$ be flat vector bundles over $S$ and $B\colon\xi_1\times\xi_2\to\xi_3$ be a bilinear pairing of flat bundles, that is, for each point $p\in S$, $B_{p}\colon (\xi_1)_p\times(\xi_2)_p\to(\xi_3)_p$ is a bilinear pairing between the fibers which varies continuously as $p$ varies over $S$. The pairing $B$ induces cup and cap products:
\begin{align*}
    B_*(\smile)\colon C^k(S,\xi_1)\times C^l(S,\xi_2)&\to C^{k+l}(S,\xi_3)\\
    B_*(\frown)\colon C^k(S,\xi_1)\times C_l(S,\xi_2)&\to C_{l-k}(S,\xi_3),
\end{align*}
which descend to maps on (co)homology.\\

The cap product induces the Poincar\'e duality isomorphism $H^k(S,\xi)\cong H_{2-k}(S,\xi)$ for any flat vector bundle $\xi$ as follows. Consider the trivial bundle $\xi_0\coloneqq S\times \R$, and let $B_0\colon\xi_0\times\xi\to \xi$ be given by fiberwise multiplication, i.e. $B_0(x_p,Y_p) = x_p\cdot Y_p$ for any $p\in S$, and where $x_p\in(\xi_0)_p=\R$ and $Y_p\in\xi_p$, which is isomorophic to a vector space. Let $[S]\in H_2(S,\xi_0)$ be a fundamental class (given by the orientation of $S$) for the homology. Then the cap product
\begin{align*}
\frown [S]\colon H^k(S,\xi)&\to H_{2-k}(S,\xi)\\
[C]&\mapsto (B_0)_*([C]\frown[S])
\end{align*}
is an isomorphism which is known as  \emph{Poincar\'e duality} (see \cite[Section 2]{InvFct_Goldman} or \cite{Cohen}).\\

Now consider a bilinear pairing $B\colon \xi_1\times \xi_2\to\xi_3$. Poincar\'e duality induces an \textit{intersection pairing} $\bullet_B\colon H_1(S,\xi_1)\times H_{1}(S,\xi_2)\to H_0(S,\xi_3)$ as follows. For $[A]\in H_1(S,\xi_1)$ and $[C]\in H_1(S,\xi_2)$, the intersection pairing is defined by
\[
[A]\bullet_B [C] = \left(B_*((\frown [S])^{-1}([A])\smile (\frown [S])^{-1}([C]))\right)\frown[S].
\]
We now describe how to compute the intersection pairing. For $[A]$ and $[C]$ as above, choose representatives $A = \sum_i\sigma_i\otimes a_i\in Z_1(S,\xi_1)$ and $C = \sum_j \tau_j\otimes c_j\in Z_{1}(S,\xi_2)$. Moreover, assume that the simplices $\sigma_i$ and $\tau_j$ \textit{intersect transversely at double points}, that is, the maps $\sigma_i,\tau_j\colon \textnormal{Int}([v_0,\dots,v_k])\to S$ are transverse, and for each $q\in S$, $\sigma_i^{-1}(q)\cup \tau_j^{-1}(q)$ has at most two elements; such representatives always exist. Denote by $\sigma_i\#\tau_j$ the set of transverse intersections at double points. Then as in Section $2$ of \cite{InvFct_Goldman} and Section 4 of \cite{JohnsonMillson}, the intersection pairing is computed as
\begin{equation}\label{eq : formula for intersection pairing}
[A]\bullet_B [C] = \sum_{i,j}\sum_{q\in\sigma_i\#\tau_j}\eps(q;\sigma_i,\tau_j)B(a_i(q),c_j(q))\in H_0(S,\xi_3),
\end{equation}
where $\eps(q;\sigma_i,\tau_j)\in\{\pm 1\}$ is the algebraic intersection number at $q$ between $\sigma_i$ and $\tau_j$.\\

The intersection pairing can be computed explicitly by going to the universal cover. Assume that $\xi_l = \widetilde S\times_{\phi_l}V_l$ for representations $\phi_l\colon\pi\to\GL(V_l)$ and for $l = 1,2,3$. Consider a lift of the pairing $B$ to a pairing $\widetilde B\colon V_1\times V_2\to V_3$, meaning that it is in particular invariant under the action by $\pi$, that is
\[
\widetilde B(\phi_1(\gamma)(v_1),\phi_2(\gamma)(v_2)) = \widetilde B(v_1,v_2)
\]
for every $v_1,\in V_1$, $v_2\in V_2$ and $\gamma\in\pi$. Now lift the sections $a_i$ and $c_j$ to sections $\tilde a_i,\tilde c_j$ in the trivial bundles $\widetilde S\times V_1$ and $\widetilde S\times V_2$ respectively. Then for any lift $\tilde q$ of a point $q\in\sigma_i\#\tau_j$, we have that
\begin{equation}\label{eq : intersection in universal cover}
B(a_i(q),c_j(q)) = \widetilde B(\tilde a_i(\tilde q),\tilde c_j(q)).
\end{equation}
Since $\widetilde B$ is invariant under the action by $\pi$, the above  quantity is invariant under the choice of lift of $q$.\\

The vector bundles we will be working with are the bundles $\xi_\rho \coloneqq \tilde S\times_{\Ad\rho}\gfr$ and $\xi_\rho^*=\tilde S\times_{\Ad^*\rho}\gfr^*$, where $\rho\colon\pi\to G$ is a representation, and $\Ad_\rho$ and $\Ad^*_\rho$ are the corresponding adjoint representations on $\gfr$ and $\gfr^*$ respectively.\\
 
  We will also consider two possible pairings. The first is the pairing $\B\colon\xi_\rho\times\xi_\rho\to \xi_0=S\times \R$ induced by the orthogonal structure $\B$. This gives another way to define the symplectic structure. Namely, for $[u],[v]\in H^1(S,\xi_\rho)$, we have that 
  \begin{equation}\label{eq : alternative definition of symplectic form}
  \omega_{[\rho]}([u],[v]) =([u] \frown [S])\bullet_\B([v]\frown [S]).
  \end{equation}
  
  The second pairing we will use, is the pairing $\langle\cdot,\cdot\rangle\colon\xi_\rho\times\xi_\rho^*\to \xi_0=S\times\R$ coming from pairing the vector bundle with its dual. This will be used in Section \ref{sec : subsurface deformations} through the following lemma.

\begin{lemma}\label{lem : Poincare duality}
	A cocycle $u\in Z^k(S,\xi_\rho)$ is Poincar\'e dual to a cycle $A\in Z_{2-k}(S,\xi_\rho)$, i.e. $[u]\frown [S] = [A]$, if and only if for every cycle $C\in Z_{2-k}(S,\xi_\rho^*)$
\[
\inp{[u]\frown [C]}_* = [A]\bullet_{\langle\cdot,\cdot\rangle} [C]\in H_0(S,\R)\cong \R.
\]
\end{lemma}

\section{Invariant functions}\label{sec : invariant functions}
In this section, we define invariant functions and describe how they induce functions on the character variety.

\subsection{Invariant functions on \texorpdfstring{$G^k$}{Gk}}
We begin with a definition.
\begin{definition}
Let $k\in\N$. A $C^1$ function $f\colon G^k\to \R$ which is invariant under the diagonal action of $G$ by inner automorphisms on $G^k$, is said to be an \textit{invariant function}. 
\end{definition}
We will adopt the following notation. An element in $G^k$ will be denoted by $\ub g = (g_1,\dots,g_k)$, and for $h\in G$, $h\ub g = (hg_1,\dots,hg_k)$. Similarly, for an element $(X_1,\dots,X_k)\in\gfrk$, we have the diagonal adjoint action $\Ad_h(X_1,\dots,X_k) = (\Ad_h(X_1),\dots,\Ad_h(X_k))$.\\

We now turn to the variation function of an invariant function. Let $f\colon G^k\to\R$ be an invariant function. The differential $\D f$ is a 1-form on $G^k$, which is given by
\begin{align*}
    \D_{\ub{g}}f\colon \T_{g_1}G\oplus\cdots\oplus\T_{g_k}G&\to \R\\
    (X_{1},\dots,X_{k})&\mapsto \sum_{j = 1}^k\brac{\frac{\partial f}{\partial g_j}}_{\ub{g}}(X_{j}),
\end{align*}
where the partial derivative $(\partial f/\partial g_j)_{\ub{g}}$ defines an element of  $\T^*_{g_j}G$.\\

Now identify $\gfr$ with the right-invariant vector fields on $G$ and $\gfr^*$ with the right-invariant 1-forms on $G$. There are similar identifications for $\gfrk$ and $(\gfrk)^*$. With these, we uniquely extend the covector $(\partial f/\partial g_j)_{\ub{g}}$ to a right-invariant 1-form on $G$, %in $\gfr^*$, 
which we denote by $(\partial_j f)_{\ub{g}}$. Similarly, extend the vectors $X_{j}\in \T_{g_j}G$ to right-invariant vector fields $X_j\in\gfr$. The differential can therefore be rewritten as
\begin{align}
    \D_{\ub{g}}f\colon \gfrk&\to\R\nonumber\\
    (X_1,\dots, X_k)&\mapsto\sum_{j = 1}^k (\partial_j f)_{\ub{g}}(X_j).\label{eq:differential of multi-function is sum of partial differentials}
\end{align}
Explicitly, the 1-forms applied to elements in the Lie algebra are given by
\begin{equation}\label{eq:explicit formula for left invariant 1 form applied to an element in the Lie algebra}
    (\partial_j f)_{\ub{g}}(X_j)=\frac{d}{dt}\bigg|_{t=0}f(g_1,\dots,g_{j-1},\exp(tX_{j})g_j,g_{j+1},\dots,g_k).
\end{equation}

Thus, for $\ub{g}\in G^k$, the derivative $\D_{\ub{g}}f$ is an element in the dual $(\gfrk)^*$. We define
\begin{align*}
    \widehat{F}\colon G^k&\to (\gfrk)^*\\
    \ub{g}&\mapsto \D_{\ub{g}}f.
\end{align*}
By invariance of $f$, we get that the map $\widehat{F}$ is equivariant with respect to the diagonal action. That is, for $h\in G$ and $\ub{g}\in G^k$, we have that 
\[
	\widehat{F}(h\ub{g}h^{-1})= \Ad^*_h( \widehat{F}(\ub{g})),
\]
where $\Ad^*_h$ is the adjoint action on the dual of $\gfr$.\\

The orthogonal structure $\B$ on $G$ induces a pairing $\B_k=\B\oplus\cdots\oplus\B\colon\gfr^{\oplus k}\to\R$, which in turn induces an isomorhpism $\widetilde{\B}_k\colon \gfrk \to(\gfrk)^*$. 

\begin{definition}
The \textit{variation function of }$f$ (relative to $\B$) is the function $F\colon G^k\to \gfrk$ defined by the composition $\widetilde{\B}_k^{-1}\circ \widehat{F}$.
\end{definition}

By the invariance of $\B_k$ under the adjoint action and the equivariance of $\widehat{F}$, the function $F$ itself is equivariant under the diagonal action. That is
\[
	F(h\ub{g}h^{-1})=\Ad_h(F(\ub{g})),
\]
for $\ub{g}\in G^k$ and $h\in G$. Using the orthogonal structure, it is possible to give an implicit definition of the variation function. To do this, denote by $F_1,\dots,F_k\colon G^k\to \gfr$ the component functions of $F$, that is $F(\ub{g})=(F_1(\ub{g}),\dots, F_k(\ub{g}))$. From the above equivariance of the variation function, we see that for each $i = 1,\dots,k$,
\[
    F_i(h\ub{g}h^{-1}) = \Ad_h(F_i(\ub{g}))
\]
for every $h\in G$ and $\ub{g}\in G^k$.

\begin{lemma}\label{lemma:variation function of multi-function  is given implicitly wrt to orthogonal structure}
For $\ub{g}\in G^k$, the variation function is implicitly defined by
\[
	\B(F_j(\ub{g}),X)=\frac{d}{dt}\bigg|_{t=0}f(g_1,\dots,g_{j-1},\exp(tX)g_j,g_{j+1},\dots,g_k),
\]
for any $X\in\gfr$, and for $j=1,\dots,k$.
\end{lemma}
\begin{proof}
By definition of $F$, we have that for every $\ub{g}\in G^k$ and $(X_1,\dots, X_k)\in \gfrk$,
\[
    \B_k(F(\ub{g}),(X_1,\dots X_k))=(\widehat{F}(\ub{g}))(X_1,\dots, X_k).
\]
Then let $X\in\gfr$, and set $X_j=(0,\dots,0, X, 0,\dots,0)\in\gfrk$, that is, $0$ in all positions except the $j-$th. We see that
\[
    \B_k(F(\ub{g}),X_j)=\B(F_j(\ub{g}),X),
\]
where we used the bilinearity of $\B$. Then from \eqref{eq:differential of multi-function is sum of partial differentials} and \eqref{eq:explicit formula for left invariant 1 form applied to an element in the Lie algebra}, it follows that
\begin{align*}
    (\widehat{F}(\ub{g}))(X_j) &= (\partial_jf)_{\ub{g}}(X)\\
    &=\frac{d}{dt}\bigg|_{t=0}f(g_1,\dots,g_{j-1},\exp(tX)g_j,g_{j+1},\dots,g_k).
\end{align*}
\end{proof}
\subsection{Invariant functions through word maps}\label{sec : invariant functions through word maps}
A function $f\colon G\to\R$ gives rise to functions $G^k\to\R$ in the following way. Let $w$ be a word in $k$ letters, and consider the induced word map $w\colon G^k\to G$ given by substitution. We then obtain another invariant function
\begin{align*}
    f^w\colon G^{k}&\to\R\\
    (g_1,\dots,g_{k})&\mapsto f(w(g_1,\dots,g_{k})).
\end{align*}
A special case of this is when the word $w$ induces the word map which multiplies the elements or their inverses in the order they appear. That is, words whose word map is of the form

\begin{align*}
    w\colon G^k&\to G\\
    (g_1,\dots,g_k)&\mapsto g_1^{\eps_1}g_2^{\eps_2}\cdots g_k^{\eps_k},
\end{align*}
for some $\eps_i\in\{\pm 1\}$. For such a word, we have that for every $\ub g=(g_1,\dots,g_k)\in G^k$, $f^w(\ub g) = f(g_1^{\eps_1}\cdots g_k^{\eps_k})$. In the case when $\eps_1 = \cdots = \eps_k=1$, we say that $f^w$ is a \textit{reducible function}. In the following lemma, we describe the variation function of functions satisfying the above condition.%\footnote{Here we do not discuss the variation function for functions coming from general words.}.

\begin{lemma}\label{lemma: variation function of reducible function}
    Assume that for an invariant function $f\colon G^k\to \R$, there exists another invariant function $f'\colon G\to\R$ such that for all $(g_1,\dots,g_k)\in G^k$, $f(g_1,\dots,g_k) = f'(g_1^{\eps_1}\cdots g_k^{\eps_k})$ for $\eps_i\in\{\pm 1\}$. Letting $F\colon G^k\to \gfrk$ and $F'\colon G\to \gfr$ be the corresponding variation functions, we have that
    \[
    F_i(g_1,\dots,g_k)=\begin{cases}
        F'(g_ig_{i+1}^{\eps_{i+1}}\cdots g_k^{\eps_k}g_1^{\eps_1}\cdots g_{i-1}^{\eps_{i-1}}),\quad&\textnormal{if }\eps_i = 1\\
        -F'(g_{i+1}^{\eps_{i+1}}\cdots g_k^{\eps_k}g_1^{\eps_1}\cdots g_i^{-1}),\quad&\textnormal{if }\eps_i = -1.
    \end{cases}
    \]
    In particular, if $\eps_1=\cdots=\eps_k=1$,
    \[
    F_i(g_1,\dots,g_k) = \Ad_{(g_1\cdots g_{i-1})^{-1}}F_1(g_1,\dots,g_k).
    \]
\end{lemma}

\begin{proof}
    Using Lemma \ref{lemma:variation function of multi-function  is given implicitly wrt to orthogonal structure}, we get that
    \begin{align*}
        \B(F_i(g_1,\dots,g_k),X) &= \frac{d}{dt}\bigg|_{t = 0}f(g_1,\dots,\exp(tX)g_i,\dots,g_k)\\
        &=\frac{d}{dt}\bigg|_{t = 0} f'(g_1^{\eps_1}\cdots (\exp(tX)g_i)^{\eps_i}\cdots g_k^{\eps_k}).
    \end{align*}
    In the case when $\eps_i = 1$, using the conjugacy invariance of $f'$, we have that the above expression equals
    \begin{align*}
        \frac{d}{dt}\bigg|_{t = 0} f'(\exp(tX)g_{i}g_{i+1^{\eps_k}}\cdots g_k^{\eps_k} g_1^{\eps_1}\cdots g_{i-1}^{\eps_{i-1}})=\B(F'(g_{i}g_{i+1}^{\eps_{i+1}}\cdots g_k^{\eps_k} g_1^{\eps_1}\cdots g_{i-1}^{\eps_{i-1}}),X).
    \end{align*}
    Since this holds for all $X\in \gfr$, we obtain the first part of the lemma. In the case when $\eps_i = -1$, we have that 
    \begin{align*}
        \B(F_i(g_1,\dots,g_k),X) &= \frac{d}{dt}\bigg|_{t = 0}f'(g_i^{-1}\exp(-tX)g_{i+1}^{\eps_{i+1}}\cdots g_k^{\eps_k}g_1^{\eps_1}\cdots g_{i-1}^{\eps_{i-1}})\\
        &=\frac{d}{dt}\bigg|_{t = 0} f'(\exp(-tX)g_{i+1}^{\eps_{i+1}}\cdots g_k^{\eps_k}g_1^{\eps_1}\cdots g_i^{-1})\\
        &= \B(-F'(g_{i+1}^{\eps_{i+1}}\cdots g_k^{\eps_k}g_1^{\eps_1}\cdots g_i^{-1}),X)
    \end{align*}
    giving the second case. The last claim follows by using the equivariance of $F'$.
\end{proof}

\subsection{Induced functions on the character variety}
From an invariant function $f\colon G^k\to\R$, we obtain a function on the character variety similarly to the procedure in \cite[Section 3.1]{InvFct_Goldman}. 

\begin{definition}
	Let $\ub \alpha = (\alpha_1,\dots,\alpha_k)\in \pi^k$ be a $k$-tuple of deck transformations. The \textit{invariant function induced by }$\ub\alpha$ (or simply the \textit{induced invariant function} when it is clear from the context) is the function
	\begin{align*}
		f_{\ub\alpha}\colon \CharVar{\pi}&\to\R\\
		[\rho]&\mapsto f(\rho(\alpha_1),\dots,\rho(\alpha_k)),
	\end{align*}
	which is well defined due to the invariance of $f$.
\end{definition}

\begin{convention}
	We will always assume that the $k-$tuple $\ub\alpha$ consists of non-trivial curves.
\end{convention}

Alternatively, we may define the invariant function induced by a tuple by picking loops on the surface as follows. Let $p\in S$ be a base point, and $\ub \alpha = (\alpha_1,\dots,\alpha_k)\in \pi_1(S,p)^k$ a $k$-tuple of based loops. Fix a lift $\tilde p\in\tilde S$ of $p$, and let $\phi\colon\pi_1(S,p)\to\pi$ be the induced isomorphism. For an invariant function $f\colon G^k\to\R$, define
\begin{align*}
    f_{\ub \alpha}\colon \CharVar\pi&\to\R\\
    \rho&\mapsto f(\rho(\phi(\alpha_1)),\dots,\rho(\phi(\alpha_k))),
\end{align*}
which is once again well defined by the invariance of $f$.\\

\begin{remark}
	When picking based loops, the invariance of $f$ implies that $f_{\ub\alpha}$ is independent of the choice of isomorphism between $\pi_1(S,p)$ and $\pi$ (i.e. $f_{\ub\alpha}$ is also independent of the conjugacy class of based loops $\ub\alpha$). A further consequence of the invariance of $f$ is that we may pick a different base point $q\neq p$, and an intermediate isomorphism between $\pi_1(S,p)$ and $\pi_1(S,q)$ to induce a function on the character variety.
\end{remark}

\begin{remark}
In Goldman's work, $\alpha$ is allowed to be a curve in a fixed free homotopy class. In our setting, the curves in a tuple $(\alpha_1,\dots,\alpha_k)$ must be based at a common point, and we cannot take free homotopy classes of the curves. This is because an invariant function is only invariant under the diagonal action by conjugation of $G$ on $G^k$ and not necessarily by the action by conjugation of $G^k$ on $G^k$.
\end{remark}

\begin{remark}\label{rem : defining invariant functions on subsets of G}
    We may also consider functions $f\colon E^k\to \R$ for an open subset $E\subset G$ which is invariant under conjugation. In this case, such functions define a function on a subset of the character variety. Precisely, let $\ub\alpha\in\pi^k$ and define the subset
    \[
    \mathscr R^{E,\ub\alpha}(\pi,G)\coloneqq \{\rho\in\Reps{\pi}\st \rho(\alpha_i)\in E\,\,\forall i = 1,\dots,k\},
    \]
   and by $\mathscr X^{E,\ub\alpha}(\pi,G)$ the quotient. Then $f$ induces the function $f_{\ub\alpha}\colon\mathscr X^{E,\ub\alpha}(\pi,G)\to\R$ in the way described above.
\end{remark}

%Poisson bracket
\section{Poisson bracket between induced invariant functions}\label{sec : Poisson bracket}
The symplectic structure on $\CharVar{\pi}$ induces a Poisson bracket on $C^\infty(\CharVar{\pi})$ given by 
\[
\{f,f'\}=\omega(\Hm f',\Hm f).
\]
The main goal of this section is to prove Theorem \ref{thm: product formula}, which computes Poisson pairing of two induced invariant functions as an intersection pairing in homology over a flat bundle. The main tool for the proof is Proposition \ref{prop: explicit cycle for Poincare dual of Hamiltonian vector field}, which identifies the Poincar\'e dual to the Hamiltonian vector field in terms of the variation function. 
The product formula allows us to deduce special cases of commuting functions based on the intersection data of the based curves inducing the functions. For this, we introduce the notion of a supporting subsurface in Section \ref{sec : commuting flows}, and apply it to examples in Section \ref{sec : examples}.\\

We recall Notation \ref{notation : flat sections} briefly, noting that flat sections over curves are written as $\alpha\otimes X$, for $X\in\gfr$, where we omit the reference to the base point in the universal cover. Moreover, $X(q)$ will refer to the section $\alpha\otimes X$ evaluated at the point $q$ in $S$.\\

\subsection{The product formula}
The main result of this section is the following product formula.

\begin{theorem}[The product formula]\label{thm: product formula}
   Let $f\colon G^k\to \R$ and $f'\colon G^n\to\R$ be invariant functions. Let $p\neq r\in S$ and choose lifts $\tilde p,\tilde r\in\tilde S$ respectively, inducing isomorphisms $\phi_p\colon\pi_1(S,p)\to\pi$ and $\phi_r\colon\pi_1(S,r)\to\pi$. Let $\ub{\alpha}=(\alpha_1,\dots,\alpha_k)\in\pi_1(S,p)^k$ and $\ub{\beta}=(\beta_1,\dots,\beta_n)\in\pi_1(S,r)^n$. The Poisson bracket $\{f_{\ub \alpha},f'_{\ub \beta}\}$ is the function $\CharVar\pi\to\R$ given by
   \begin{equation}\label{eq : Poisson bracket}
   [\rho]\mapsto \sum_{i,j}\sum_{q\in\beta_j\#\alpha_i}\eps(q;\beta_j,\alpha_i)\B(F'_j(\rho(\ub\beta))(q),F_i(\rho(\ub\alpha))(q)),
   \end{equation}
   where
   \begin{align*}
    F_i(\rho(\ub\alpha)) &= F_i(\rho(\phi_p(\alpha_1)),\dots,\rho(\phi_p(\alpha_k)))\\
    F'_j(\rho(\ub\beta)) &= F'_j(\rho(\phi_r(\beta_1)),\dots,\rho(\phi_r(\beta_n))).
   \end{align*}
\end{theorem}

The main step of the proof of this theorem, is to find explicitly the Poincar\'e dual of the Hamiltonian vector field at an equivalence class of representations. The Poincar\'e dual is an explicit homology class, which we also use to compute the Hamiltonian vector field of an induced invariant function in Section \ref{sec : subsurface deformations}.

\begin{proposition}\label{prop: explicit cycle for Poincare dual of Hamiltonian vector field}
Let $f\colon G^k\to\R$, $\ub{\alpha}=(\alpha_1,\dots,\alpha_k)\in\pi_1(S,p)^k$ and $\phi_p\colon\pi_1(S,p)\to\pi$ the isomorphism arising from choosing a lift $\tilde p\in\tilde S$. Let $\rho\colon\pi\to G$ be a representation. Then the Poincar\'e
 dual to the Hamiltonian vector field at $[\rho]$, $\Hm f_{\ub{\alpha}}([\rho])\in H^1(S,\xi_\rho)$, has representative 
 \[
    \sum_{i = 1}^k\alpha_i\otimes F_i(\rho(\ub\alpha))_{\tilde p}\in Z_1(S,\xi_\rho).
 \]
\end{proposition}
We will prove Theorem \ref{thm: product formula} and then prove the proposition afterwards.\\

\begin{proof}[of Theorem \ref{thm: product formula} assuming \ref{prop: explicit cycle for Poincare dual of Hamiltonian vector field}]
Adopting Notation \ref{notation : flat sections}, by Proposition \ref{prop: explicit cycle for Poincare dual of Hamiltonian vector field}, $\Hm f_{\ub \alpha}([\rho])\frown [S] = \left[\sum_i \alpha_i\otimes F_i(\rho(\ub \alpha))\right]$ and $\Hm f'_{\ub \beta}([\rho])\frown[S] = [\sum_j \beta_j\otimes F'_j(\rho(\ub\beta))]$. By the definition of the symplectic structure (Equation \eqref{eq : alternative definition of symplectic form}), and the duality between the cup product and the intersection pairing from Section \ref{sec : cohomology with local coefficients},
\begin{align*}
    \{f_{\ub \alpha},f'_{\ub \beta}\}([\rho]) &= \omega_{[\rho]}(\Hm f'_{\ub \beta}([\rho]),\Hm f_{\ub \alpha}([\rho])) = \B_*(\Hm f'_{\ub \beta}([\rho])\smile\Hm f_{\ub \alpha}([\rho]))\frown [S]\\
    &=\left(\sum_j \beta_j\otimes F'_j(\rho(\ub\beta))\right)\bullet_\B \left(\sum_i \alpha_i\otimes F_i(\rho(\ub \alpha))\right)\\
    &= \sum_{i,j}\sum_{q\in\beta_j\#\alpha_i}\eps(q;\beta_j,\alpha_i)\B(F'_j(\rho(\ub\beta))(q),F_i(\rho(\ub\alpha))(q)),
\end{align*}
where the last line follows by the formula for the intersection pairing in Equation \eqref{eq : formula for intersection pairing}.
\end{proof}

We now prove the intermediate proposition.\\

\begin{proof}[of Proposition \ref{prop: explicit cycle for Poincare dual of Hamiltonian vector field}]
We follow the proof of Proposition 3.7 in \cite{InvFct_Goldman}, and we reproduce it here with the necessary alterations. We begin with some definitions. Denote by $ ^t\widetilde{\B}\colon H^1(\pi;\gfr_{\Ad\rho}^*)^*\to \HGroupCohom{1}^*$ the dual of the isomorphism $\HGroupCohom{1}\to H^1(\pi;\gfr_{\Ad\rho}^*)$ induced by the isomorphism $\widetilde{\B}\colon \gfr\to\gfr^*$ (which itself is induced by the nondegenerate bilinear form $\B$). Let $\eta\colon \HGroup{1}\to H^1(\pi;\gfr_{\Ad\rho}^*)^*$ be the map induced by the cap product $H^1(\pi;\gfr_{\Ad\rho}^*)\times\HGroup{1}\to H_0(\pi;\R)\cong\R$. Finally, let $\theta\colon \HGroupCohom{1}\to H^1(\pi,\gfr_{\Ad\rho}^*)^*$ be the map induced by the pairing
\[
    \HGroupCohom{1}\times H^1(\pi;\gfr_{\Ad\rho}^*)\xrightarrow{\inp{\cdot\smile\cdot}} H^2(\pi,\R)\cong\HCoReal{2}\xrightarrow{\frown [S]}\HReal{0}\cong\R.
\]
Lemma 3.8 in \cite{InvFct_Goldman} states that the diagram
\[
\begin{tikzcd}
  \HGroupCohom{1} \arrow[r, "\frown {[S]}"] \arrow[d, "\widetilde{\omega}" left] \arrow[dr,"\theta"]
    & \HGroup{1} \arrow[d,"\eta"]\\
  \HGroupCohom{1}^*  &
H^1(\pi,\gfr_{\Ad\rho}^*)^*\arrow[l, " ^t\widetilde{\B}" below] \end{tikzcd}
\]
commutes. Thus, to find the Poincar\'e dual to the Hamiltonian vector field, we follow the lower part of the diagram. Let $[u]\in\HGroupCohom{1}$, with $u$ a representative cocycle. Then $\widetilde{\omega}(\Hm f_{\ub{\alpha}}([\rho]))([u])$ is given by the differential
\begin{align*}
    \widetilde{\omega}(\Hm f_{\ub{\alpha}}([\rho]))([u]) &= \D_{[\rho]}f_{\ub{\alpha}}([u])\\
    &= \hat{F}(\rho(\phi_p(\alpha_1)),\dots,\rho(\phi_p(\alpha_k)))(u(\phi_p(\alpha_1)),\dots, u(\phi_p(\alpha_k)))\\
    &= \B_k(F(\rho(\phi_p(\alpha_1)),\dots,\rho(\phi_p(\alpha_k))),u(\phi_p(\alpha_1)),\dots, u(\phi_p(\alpha_k)))\\
    &= \sum_{i = 1}^k\B(F_i(\rho(\phi_p(\alpha_1)),\dots,\rho(\phi_p(\alpha_k))),u(\phi_p(\alpha_i))).
\end{align*}
The rest of the proof is now identical to Proposition 3.7 in \cite{InvFct_Goldman}, which we reproduce here. It follows from the above computation that 
\[
    ( ^t\widetilde{\B})^{-1}(\D_{[\rho]}f_{\ub{\alpha}})([u]) = \eta \left(\sum_{i = 1}^k\phi_p(\alpha_i)\otimes F_i(\rho(\ub\alpha))\right).
\]
Therefore,
\begin{align*}
    \sum_{i = 1}^k\phi_p(\alpha_i)\otimes F_i(\rho(\ub\alpha)) &= \eta^{-1}\circ (\D_{[\rho]}f_{\ub{\alpha}}([u]))\\
    &= (\widetilde{\omega}^{-1}(\D_{[\rho]}f_{\ub{\alpha}}([u])))\frown [S]\\
    &= \Hm f_{\ub{\alpha}}([\rho])\frown [S].
\end{align*}
Using Lemma \ref{lem : Poincare duality}, we conclude that $\sum_{i = 1}^k\phi_p(\alpha_i)\otimes F_i(\rho(\ub\alpha))\in Z_1(\pi,\gfr_{\Ad\rho})$ is the Poincar\'e dual cocycle to $\Hm f_{\ub \alpha}([\rho])$. Using the canonical isomorphism between group cohomology with local coefficients and cohomology of flat sections given by picking the lift $\tilde p$ of $p$, we have that the Poincare dual to $\Hm f_{\ub\alpha}([\rho])$ is 
\[
\sum_{i = 1}^k\alpha_i\otimes F_i(\rho(\ub\alpha))_{\tilde p}\in Z_1(S,\xi_\rho).
\]
\end{proof}

\begin{remark}\label{prop : Variation function for reducible functions}
In the special case when $f\colon G^k\to \R$ is a reducible function as in Section \ref{sec : invariant functions through word maps}, the Poincar\'e dual to the Hamiltonian vector field has two equivalent descriptions. Indeed, if $f$ is reducible, there exists an invariant function $f'\colon G\to \R$ such that for every $(g_1,\dots,g_k)\in G^k$, $f(g_1,\dots,g_k) = f'(g_1\cdots g_k)$. Let $(\alpha_1,\dots,\alpha_k)\in\pi_1(S,p)^k$. Then by Proposition \ref{prop: explicit cycle for Poincare dual of Hamiltonian vector field}, the Poincar\'e dual to $\Hm f_{\ub\alpha}([\rho])$ has representative
\[
\sum_{i = 1}^k\alpha_i\otimes F_i(\rho(\ub\alpha)).
\]
Moreover, since $f_{(\alpha_1,\dots,\alpha_k)} = f'_{(\alpha_1\cdots\alpha_k)}$, Lemma \ref{lemma: variation function of reducible function} implies that the Poincar\'e dual to $\Hm f_{\ub\alpha}([\rho])$ also has representative
\[
\alpha_1\cdots\alpha_k\otimes F_1(\rho(\ub\alpha)).
\]
These two cycles are homologous and differ by the boundary 
\[
\partial\left(\sum_{i = 1}^{k - 1}\sigma_i\otimes F_1(\rho(\ub\alpha))\right),
\]
where $\sigma_i$ is the $2-$simplex given by $\partial_0\sigma_i = \alpha_{i+1}$, $\partial_1\sigma_i = \alpha_1\cdots\alpha_{i+1}$, and $\partial_2\sigma_i = \alpha_1\cdots\alpha_i$.
\end{remark}

\subsection{Commuting flows}\label{sec : commuting flows}
With the formula for the Poisson pairing \eqref{eq : Poisson bracket}, it immediately follows that 
\[
\{f_{\ub\alpha},f'_{\ub\beta}\}\equiv 0
\]
if there are representatives of the curves $\alpha_i$ and $\beta_j$ such that $\alpha_i\#\beta_j = \emptyset$ for every $i,j$. In particular, the Hamiltonian flows of $f_{\ub\alpha}$ and $f_{\ub\beta}'$ commute (see for example \cite[Corollary 9 pp. 218]{MathMethodsOfClassicalMechanics_Arnold}).\\ %To apply this to the examples in Section \ref{sec : examples}, we make the following definition, which we will also use in Section \ref{sec : subsurface deformations}.

Say that a subsurface $S_0\subset S$ is \emph{essential} if every loop in $S_0$ that bounds a disk in $S$ also bounds a disk in $S_0$.
The isotopy classes of essential subsurfaces of $S$ are partially ordered by inclusion of representatives.\\

An element $\underline \alpha \in \pi_1(S,p)$ defines, up to homotopy preserving the basepoint, a continuous map 
\[ (\vee_{i = 1}^k S^1, \star) \to (S,p)\]
called a \emph{bouquet of curves}.
\begin{definition}\label{def : supportin subsurface}
    For $\underline \alpha \in \pi_1(S,p)^k$, say that an essential subsurface $S_0 \subset S$ is a \emph{supporting subsurface of $\underline \alpha$} if the corresponding bouquet of curves can be (freely) homotoped into $S_0$, and the isotopy class of $S_0$ is the smallest essential subsurface of $S$ with this property.
\end{definition}

\begin{remark}
    Any supporting surface $S_0$  for $\underline \alpha \in \pi_1(S,p)^k$ containing $p$  satisfies the following property: There are $\gamma_1, ..., \gamma_k \in \pi_1(S_0, p)$ and $\gamma \in \pi_1(S, p)$ such that that $\gamma(\gamma_1, ..., \gamma_k)\gamma^{-1} = \underline \alpha \in \pi_1(S,p)^k$. 
\end{remark}

Using this setup, we obtain a condition on when two induced invariant functions Poisson--commute. We will use this in Section \ref{sec : examples} to describe when certain Hamiltonian flows commute.

\begin{corollary}\label{cor : if supporting subsurfaces are disjoint, flows commute}
    Let $f\colon G^k\to \R$ and $f'\colon G^n\to\R$ be invariant functions. Let $\ub\alpha\in\pi_1(S,p)^k$ and $\ub\beta\in\pi_1(S,r)^n$. If there are supporting subsurfaces $S_1$ and $S_2$ of $\ub\alpha$ and $\ub\beta$ respectively such that $S_1\cap S_2 = \emptyset$, then
    \[
    \{f_{\ub\alpha},f'_{\ub\beta}\}\equiv 0.
    \]
\end{corollary}
\begin{proof}
    The condition on the disjointness of supporting subsurfaces implies that there are representatives of the curves $\alpha_i$ and $\beta_j$ such that $\alpha_i\#\beta_j = \emptyset$ for every $i,j$. The result then follows by Theorem \ref{thm: product formula}. 
\end{proof}

\section{Examples}\label{sec : examples}
We now describe a variety of examples of functions on the character variety, which can be realized as induced invariant functions.
\subsection{Polynomials in trace functions}\label{polynomials in trace functions}
In this subsection, we work with a reductive subgroup $G<\GL(d,\mathbb K)$, where $\mathbb K$ is either $\R$ or $\C$. Notice that in the case that $\mathbb K = \C$, we may consider a $\C$-valued orthogonal structure on $\gfr$ and with that a $\C-$valued symplectic form, following the same definition as in Section \ref{sec : symplectic structure on character varieties}. There is a natural invariant function to consider on these groups, namely the trace $\trace\colon G\to\mathbb K$. Such functions, their variations and their Poisson brackets are studied extensively in \cite[Sections 1 and 3]{InvFct_Goldman}.\\

For a curve $\gamma\in\pi$, we call $\trace_\gamma\colon\CharVar\pi\to\R$ the \textit{trace function of $\gamma$}. A \textit{polynomial in trace functions} is an element of $\mathbb K[\trace_\gamma\st\gamma\in\pi]$. Any polynomial in trace functions is realized as an induced invariant function as follows. Let $T\in \mathbb K[\trace_\gamma\st\gamma\in\pi]$ and define the invariant function 
\begin{align*}
    \tau\colon G^k&\to\mathbb K \\
    (g_1,\dots,g_k)&\mapsto T(\trace(g_1),\dots,\trace(g_k)).
\end{align*}
Let $\gamma_1,\dots,\gamma_k\in\pi$ be the curves appearing in the polynomial $T$. We then obtain the function $\tau_{(\gamma_1,\dots,\gamma_k)}$ on $\CharVar{\pi}$ which is equal to $T$.
\begin{remark}
    \sloppy The function $\tau$ in this case is not only invariant under the diagonal action of $G$ on $G^k$, but also invariant under the action by conjugation of $G^k$ on $G^k$. That is, for any $(h_1,\dots,h_k),(g_1,\dots,g_k)\in G^k$,
    \[
    \tau(h_1g_1h_1^{-1},\dots,h_kg_kh_k^{-1}) = \tau(g_1,\dots,g_k).
    \]
    In particular, this implies that the variation function of $\tau$ is also equivariant under the action of $G^k$.
\end{remark}

In general, the curves $\gamma_1,\dots,\gamma_k$ may be complicated and have self intersections. One can resolve these curves by writing them as words in simple curves.

\begin{lemma}\label{lemma : polynomial in traces is invariant in simple}
	Let $G<\mathsf{GL}(d,\mathbb K)$ and $T$ a polynomial in trace functions. Then there is an integer $n$, an invariant function $\tau\colon G^n\to\mathbb K$, and an $n-$tuple $\ub\alpha\in\pi^n$ consisting of simple closed curves such that $T = \tau_{\ub\alpha}$.
\end{lemma}
\begin{proof}
    As above, let $T\in\mathbb K[\trace_\gamma\st\gamma\in\pi]$ and $\gamma_1,\dots,\gamma_k\in\pi$ the curves appearing in the polynomial $T$. Now let $\alpha_1^i,\dots,\alpha^i_{m(i)}\in\pi$ be deck transformations associated to simple closed curves such that $\gamma_i = \alpha^i_1\cdots\alpha^i_{m(i)}$ for each $i=1,\dots,k$. Then let
    \begin{align*}
	\tau\colon G^{m(1)}\times\cdots\times G^{m(k)}&\to\mathbb K\\
	(\ub {g_1},\dots,\ub{g_k})&\mapsto T(\trace(\gamma_1(\ub{g_1})),\dots,\trace(\gamma_k(\ub{g_k}))),
\end{align*}
where $\gamma_i$ is the word map on $G^{m(i)}$ to $G$ given by substitution. Thus, $\tau_{\ub\alpha} = T\colon\CharVar{\pi}\to\mathbb K$, where $\ub{\alpha} = (\alpha_1^1,\dots,\alpha^1_{m(1)},\dots,\alpha_1^k,\dots,\alpha^k_{m(k)})$.
\end{proof}

\begin{remark}
    Another way to do this, is to fix a generating set for $\pi$ given by simple closed curves and having $2g$ elements. Then every $\gamma_i$ defines a word in the generators, which in turn defines a function $G^{2g}\to G$ given by substitution. Following the last part of the proof of Lemma \ref{lemma : polynomial in traces is invariant in simple}, we see that any polynomial in trace functions comes from an invariant function $G^{2g}\to\mathbb K$ associated to simple closed curves. However, the words in this case may be very complicated, and the generators not well adapted to the curves $\gamma_i$.
\end{remark}
\begin{remark}
    The function $\tau$ in Lemma \ref{lemma : polynomial in traces is invariant in simple} associated to a polynomial in trace functions is only invariant under the diagonal conjugation action by $G$, and not necessarily under the conjugation action by $G^k$.
\end{remark}
\begin{example}
Consider a pair of pants inside $S$ with fundamental group $\langle a,b,c\,|\,abc=1\rangle$ with $a,b,c$ corresponding to the three boundary curves. We work with the curve $ab^{-1}$, which is a self intersecting curve, see Figure \ref{fig : figure 8 curve}. The trace function $\trace\colon G\to\mathbb K$ gives rise to the function $\trace_{ab^{-1}}\colon\CharVar\pi\to\mathbb K$. The function on the character variety also arises in another way. Namely, let $\tau\colon G^2\to\mathbb K$ given by $(g_1,g_2)\mapsto \trace(g_1g_2^{-1})$. Then we see that $\tau_{(a,b)} = \trace_{ab^{-1}}$, with the advantage that in this case the invariant function $\tau_{(a,b)}$ is induced by simple closed curves.
\end{example}

\begin{figure}[ht]
	\centering
	\includesvg[scale=0.5]{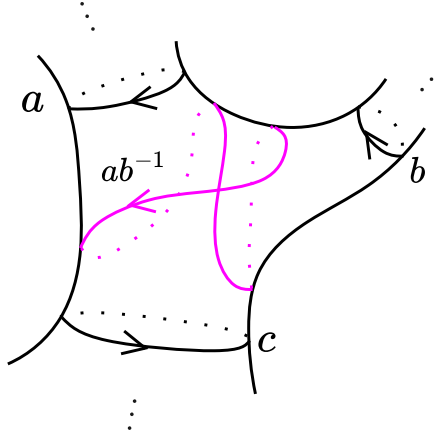}
	\caption{Self intersecting curve in a pair of pants.}
	\label{fig : figure 8 curve}
\end{figure}
\begin{corollary}
    Let $\mathbb K = \C$, and $G$ be one of $\mathsf{GL}(d,\C),\mathsf{SL}(d,\C)$ with $d\geq 2$, or $\mathsf{O}(d,\C)^\circ, \mathsf{SO}(d,\C)$ with $d\geq 3$, or $\mathsf{Sp}(2d,\C)$. Then for any regular invariant function $T$ on $\CharVar{\pi}$, there is an integer $n$, an invariant function $\tau\colon G^n\to\mathbb K$, and an $n-$tuple $\ub\alpha\in\pi^n$ consisting of simple curves such that $T = \tau_{\ub\alpha}$.
\end{corollary}
\begin{proof}
    By \cite{Procesi}, any regular invariant function in one of the above cases is a polynomial in trace functions. The corollary then follows by Lemma \ref{lemma : polynomial in traces is invariant in simple}.
\end{proof}

\subsection{Trace of a web}\label{sec : trace of a web}
Let $V$ be a $d-$dimensional vector space over a field $\mathbb K$ and let $G \le \GL(V)$ be a reductive subgroup.
A \emph{$d-$web} $\mathcal G$ is an embedded $d-$regular bipartite graph on an oriented surface $S$.
We consider isotopic $d-$webs as identical.
The vertices are colored black and white, and we consider the edges to be oriented such that all white vertices are sinks and all black vertices are sources.
\\

Place identical copies of $V$ at the vertices of an embedded $d-$web $\mathcal G$.
A \emph{$G$-connection} on $\mathcal G$ is, for every edge $e$ of $\mathcal G$ with black vertex $b_e$ and white vertex $w_e$, an isomorphism $A_e \in G$ mapping the copy of $V$ over $b_e$ to the copy of $V$ over $w_e$.
Say that a $G$-connection is \emph{flat} if whenever a cycle in $\mathcal G$ bounds a disk in $S$, then the holonomy  around that cycle, the composition of linear maps along the edges, is trivial.
A flat connection defines a representation of the fundamental group of the subsurface $S_0\subset S$ filled by $\mathcal G$, i.e., the homotopy class of the smallest  homotopically essential  embedded subsurface such that $\mathcal G$ is freely homotopic into $S_0$ (compare with Definition \ref{def : supportin subsurface}).\\

Two connections $\{A_e\}$ and $\{A_e'\}$ are \emph{equivalent} if for each vertex $v$  there is $B_v \in G$ such that $ A_e' = B_{w(e)}^{-1} A_e B_{b(e)}$.   The conjugacy class of the holonomy representation $\pi_1(\mathcal G,v) \to G$ is clearly an invariant of the equivalence class of the connection, and similarly for the induced representation $\pi_1(S_0) \to G$, if the connection is flat.
Conversely, the conjugacy class of a representation $\pi \to G$ induces an equivalence class of flat connections on any $d$-web $\mathcal G$ embedded in $S$.\\

The \emph{trace} of a flat connection $\{A_e\}$ on $\mathcal G$ can be given in terms of tensor networks (see \cite[Section 4]{Sikora:graphs} or \cite{DKS:webs}).
Briefly, one contracts certain very symmetric tensors (the co-determinants and dual co-determinants) along the edges of $\mathcal G$ using the flat connection.  Summing then produces a function to $\mathbb K$ that is an invariant of the equivalence class of the flat connection. Thus every $d-$web $\mathcal G$ defines a function $\trace_{\mathcal G}$ on the space of $G-$conjugacy classes of representations $\pi \to G$.\\
\begin{figure}
    \centering
    \includesvg[width=0.5\linewidth]{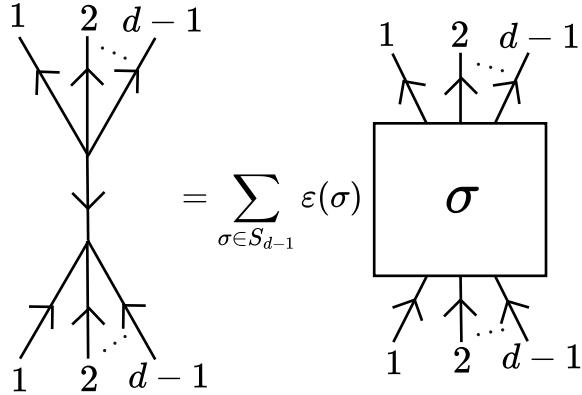}
    \caption{Notationally, $\epsilon(\sigma) \in \{\pm 1\}$ is the parity of the permutation $\sigma$ in the symmetric group $S_{d-1}$, and the box containing $\sigma$ on the right hand side of the equality represents the braid where the $i-$th strand on the bottom is joined to the $\sigma(i)-$th strand on the top of the box.  We do not keep track of crossings.}
    \label{fig : skein}
\end{figure}

Sikora proved \cite[Theorem 3.6]{Sikora:graphs} that the trace of a $d-$web is invariant under \emph{skein operations}.
The skein operations on $\mathcal G$ are depicted in Figure \ref{fig : skein}, and in the \emph{skein algebra}, $\mathcal G$ is equivalent to a formal $\Z-$linear combination of \emph{immersed} $d-$webs and closed, oriented loops \cite[\S4]{Sikora:graphs}.
So by applying a sequence of skein operations, the trace $\trace_{\mathcal G}$ of each $d-$web $\mathcal G$ can be expressed as a polynomial with integer coefficients in the usual trace functions in closed curves in $S_0$; see Figure \ref{fig : theta web}.
This polynomial expression is generally not unique, but rather depends on the sequence of skein relations applied to reduce the $d-$web through immersed $d-$webs to curves.
However, the corresponding functions on $\CharVar{\pi}$ are all equal.

\begin{figure}
    \centering
    \includesvg[width=0.5\linewidth]{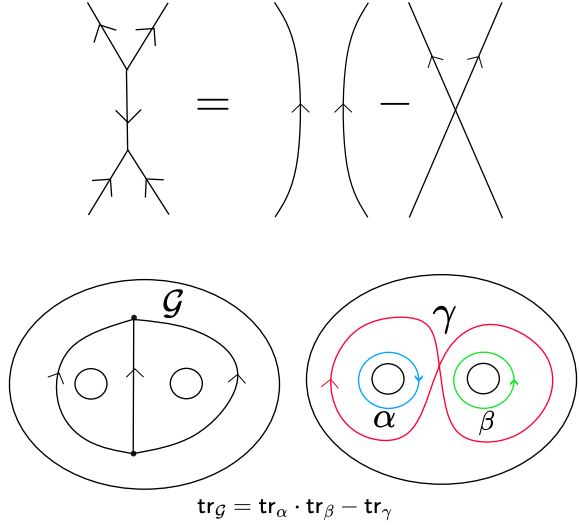}
    \caption{The skein relation for $\dim V = 3$ is pictured on top.  The simplest non-trivial $3-$web (the $\Theta-$graph $\mathcal G$), supported on a pair of pants, is depicted beneath.  We have applied the skein relation to see the trace function for $\mathcal G$ as a polynomial in the traces of oriented closed curves.}
    \label{fig : theta web}
\end{figure}

\begin{corollary}
    The trace of a $d-$web $\mathcal G$ in $S$ can be realized as an induced invariant function on $\CharVar{\pi}$.
    If two webs $\mathcal G_1$ and $\mathcal G_2$ fill homotopically disjoint essential subsurfaces $S_1$ and $S_2$, then the trace functions $\mathsf {tr}_{\mathcal G_1}$ and $\mathsf{tr}_{\mathcal G_2}$ Poisson-commute.  
\end{corollary}

\subsection{Invariants of \texorpdfstring{$\Theta-$}{Theta}Anosov representations}\label{sec : Invariants of Anosov representations}
In this section, we describe how a variety of functions associated to $\Theta-$Anosov representations in $\SL d$ can be realized as induced invariant functions. In particular, we describe correlation functions as defined in \cite{Labourie_SwappingAlgebra,SimpleLengthRigidity_BCL,GhostPolygons}, as well as cross ratios and triple ratios as defined for example in \cite{FockGoncharov,BonahonDreyer_GeneralLaminations,CollarTheta}.\\

Fix $G = \SL d$ and a subset $\Theta\subseteq \{1,\dots,d-1\}$ which is symmetric, i.e. if $k\in\Theta$, then $d-k\in\Theta$. The flag variety $\flags_\Theta$ associated to $\Theta=\{i_1,\dots,i_r\}$ is the set of filtered linear subspaces $\{0\}=V_0\subsetneq V_{i_1}\subsetneq \cdots\subsetneq V_{i_{r}}\subsetneq V_{d}= \R^d$, where $\dim V_{i_j}=i_j$. For a flag $F\in\flags_\Theta$, we denote by $F^{(k)}$ the $k-$th dimensional part of the flag $F$. We say that two flags $F,E\in\flags_\Theta$ are \textit{transverse} if $\R^d = F^{(i_j)}\oplus E^{(d-i_j)}$ is a direct sum for every $i_j\in\Theta$. With this, we define for $k\in\N$, the space $\flags_\Theta^{[k]}$ of ordered $k-$tuples of pairwise transverse flags.\\

Taking the definition from \cite{GGKW_17}, we say that an element $g\in G$ is \textit{proximal in $\flags_\Theta$} if $g$ has an attracting fixed point in $\flags_\Theta$, and denote the subset of such elements by $\proxGTheta$. Elements in $\proxGTheta$ have unique attracting fixed points \cite[Proposition 3.3 (c)]{GGKW_17}, which allows us to define the map $g\mapsto g^+$ sending an element $g\in \proxGTheta$ to its attracting flag in $\flags_\Theta$. We also get a map $g\mapsto g^-$ sending $g$ to the attracting fixed flag of $g^{-1}$. The flags $g^+$ and $g^{-}$ are transverse. Denote by $\proxGTheta^{[k]}\subset\proxGTheta^k$ the tuples $(g_1,\dots,g_k)$ such that $(g_1^+,\dots,g_k^+)\in\flags_\Theta^{[k]}$ is a tuple of pairwise transverse flags. The set $\proxGTheta^{[k]}$ is an open subset of $G^k$ which is invariant under conjugation.\\

Anosov representations were defined originally in \cite{Labourie_AnosovReps}, and generalized in \cite{AnosovReps_GuichardWienhard}. We will not give a definition of $\Theta-$Anosov representations, but we will give their main properties given in \cite[Theorem 1.7]{GGKW_17}. If $\rho\colon\pi\to G$ is a $\Theta-$Anosov representation, then there exists a continuous $\rho-$equivariant boundary map $\bdryMap\colon\partial_\infty\pi\to\flags_\Theta$ which is
\begin{enumerate}
    \item \textit{transversal}, that is, for every $x\neq y\in\partial_\infty\pi$, the flags $\flags_\rho(x)$ and $\flags_\rho(y)$ are transverse, and
    \item \textit{dynamics preserving}, that is, for every non-trivial $\gamma\in\pi$ with attracting fixed point $\gamma^+\in\partial_\infty\pi$, the flag $\flags_\rho(\gamma^+)$ is the attracting fixed point of $\rho(\gamma)$, in particular $\rho(\gamma)$ is proximal in $\flags_\Theta$.
\end{enumerate}

We will denote by $\Hom^\Theta(\pi,G)$ the set of $\Theta-$Anosov representations, and by $\ThetaAnosovReps$ the set of conjugacy classes of $\Theta-$Anosov representations.\\

A class of representations that will reappear throughout the article is Hitchin representations, which we now describe. A representation is \textit{Hitchin} if it is in the same connected component of $\Hom(\pi,G)$ as representations which factor through an irreducible representation $\SL 2\to\SL d$. The quotient of the space of Hitchin representations will be denoted by $\Hit d$. It was shown in \cite{Labourie_AnosovReps} that every Hitchin representation is $\Delta = \{1,\dots,d-1\}-$Anosov.

\subsubsection{Length functions}\label{sec : length functions}
An important function on Teichm\"uller space is the length function. Namely, for a closed curve $\alpha$ in $S$, let $\ell_\alpha\colon\Teich\to\R$ be the function measuring the length of the geodesic representative of $\alpha$ in a given hyperbolic metric. The Hamiltonian flow of this function is the well known twist flow along $\alpha$ \cite{Wolpert_FNtwist}. Another way to realize the length function is through Goldman's setting of invariant functions \cite{InvFct_Goldman}. Let 
\begin{align*}
    \ell^1\colon \PSL 2&\to\R\\
    g&\mapsto \log\frac{\lambda_1(\widetilde g)}{\lambda_2(\widetilde g)},
\end{align*}
where $\widetilde g$ is a lift of $g$ to a matrix in $\SL 2$, and $\lambda_i(\widetilde g)$ is the modulus of the eigenvalue of $\widetilde g$ in non-decreasing order. Then the induced function $\ell^1_\alpha$ on the character variety $\mathscr X(\pi,\PSL 2)$ restricted to the Teichm\"uller component is equal to the function $\ell_\alpha$.\\

Length functions generalize to the higher rank setting when $G = \SL d$ and $\Theta-$Anosov representations in the context of Goldman's invariant functions \cite{InvFct_Goldman}. For an element $g\in\SL d$, let $\lambda_1(g)\geq\lambda_2(g)\geq\cdots\geq\lambda_d(g)$ be the modulus of the eigenvalues of $g$ counted with multiplicity. The for each $j\in\Theta$, the \textit{$j-$th simple root length} is defined to be
\begin{align*}
    \ell^j\colon\proxGTheta &\to\R\\
    g&\mapsto \log\frac{\lambda_j(g)}{\lambda_{j+1}(g)},
\end{align*}
which by \cite[Definition 2.25]{GGKW_17} is always positive for elements that are proximal in $\flags_\Theta$. Hence, picking a curve $\alpha\in\pi$, we obtain a function
\begin{align*}
    \ell^j_\alpha\colon \ThetaAnosovReps & \to\R\\
    [\rho]&\mapsto \ell^j(\rho(\alpha))
\end{align*}
which is positive. The Hamiltonian flow of these functions $\ell^j_\alpha$ was already described by \cite{InvFct_Goldman}.\\

In the case of Hitchin representations, Bonahon-Dreyer extended these length functions to oriented geodesic currents in \cite[Section 7]{BonahonDreyer_GeneralLaminations} by approximating currents with weighted curves. For an oriented geodesic current $\mu$ and $j\in\{1,\dots,d-1\}$, denote by $\ell^j_\mu$ the $j-$th simple root length of the geodesic current. In our setting with invariant multi-functions, it is possible to realize the length function of an oriented geodesic current as an induced invariant function. 
A supporting subsurface $S_0 \subset S$ for a geodesic current $\mu$ is defined in a similar fashion as in Definition \ref{def : supportin subsurface}.

\begin{corollary}\label{cor : Bonahon Dreyer length functions are invariant functions}
    Let $j\in\{1,\dots,d-1\}$, $\mu$ an oriented geodesic current, and $S_0$ the supporting subsurface for the support of $\mu$. Then the length function $\ell^j_\mu\colon\Hit d\to\R$ can be realized as an invariant function induced by generators of the fundamental group of the supporting subsurface $S_0$ of $\mu$.
\end{corollary}
\begin{proof}
    The space of oriented geodesic currents can be described as a closure of the space of free homotopy classes of curves with real positive weights \cite[Section 1, Proposition 2]{TeichviaCurrents_Bonahon}. Let $\mu$ be an oriented geodesic current, and let $(x_n\gamma_n)_{n\in\N}$, with $x_n\in\R_+$ and $\gamma_n\in\pi$ be a sequence convering to $\mu$. We may moreover require that the curves $\gamma_n$ be contained in the supporting subsurface $S_0$. Now pick generators $\{\alpha_1,,\dots,\alpha_k\}$ of the fundamental group of $S_0$, and denote also by $\gamma_{n}$ the word maps sending the generators to the curve $\gamma_n$. Then define
    \begin{align*}
        L^j\colon\proxGTheta^{k}&\to\R\\
        (h_1,\dots,h_{k})&\mapsto \lim_{n\to\infty}x_n\cdot \ell^j(\gamma_n(h_1,\dots,h_{k})).
    \end{align*}
    This is an invariant function and by Theorem 7.1 in \cite{BonahonDreyer_GeneralLaminations}, it follows that 
    \[
        \ell^j_\mu = L^j_{(\alpha_1,\dots,\alpha_k)}
    \]
    when restricted to the Hitchin component of $\SL d$.
\end{proof}

We can further use this corollary combined with Corollary \ref{cor : if supporting subsurfaces are disjoint, flows commute} to deduce when the Hamiltonian flows of length functions of geodesic currents commute.

\begin{corollary}
    Let $j_1,j_2\in\{1,\dots, d-1\}$, and $\mu_1,\mu_2$ be two oriented geodesic currents. If $\mu_1$ and $\mu_2$ admit supporting subsurfaces $S_1$ and $S_2$ such that $S_1\cap S_2 = \emptyset$, then
    \[
    \left\{\ell^{j_1}_{\mu_1},\ell^{j_2}_{\mu_2}\right\}\equiv 0
    \]
    on $\Hit d$, and in particular their Hamiltonian flows commute.
\end{corollary}

\subsubsection{Correlation functions}\label{sec : correlation functions}
Correlation functions were first introduced in \cite{Labourie_SwappingAlgebra} to study the swapping algebra defined in the same paper, and in \cite{SimpleLengthRigidity_BCL} in order to establish simple marked length rigidity for Hitchin representations. In \cite{GhostPolygons}, Bridgeman and Labourie generalize correlation functions in particular to $\Theta-$Anosov representations and establish a formula for the Poisson bracket between such functions. We now explain what correlation functions are and how they can be realized as induced invariant functions.\\

Begin by associating to an element $g\in\proxGTheta$ a splitting of $\R^d$. Let $L_j(g) = (g^+)^{(i_j)}\cap (g^-)^{(d-i_{j-1})}$ for $j = 1,\dots,r+1$, where we use the convention that $i_0 = 0$, $i_{r+1}=d$ and $F^{(d)}= \R^d$ for any flag $F$. Since the flags $g^+$ and $g^-$ are transverse, we have that $\dim L_j(g) = i_j-i_{j-1}$ and moreover,
\[
\R^d = L_1(g)+L_2(g)+\cdots+L_{r+1}(g),
\]
    is a direct sum. The \textit{projector} $p_j(g)\colon\R^d\to\R^d$ is defined to be the projection onto $L_j(g)$ with kernel $\bigoplus_{i\neq j}L_i(g)$ for any $j \in\{1,\dots,r+1\}$.\\

Let $\rho\colon\pi\to G$ be a $\Theta-$Anosov representation. In particular, we have that for every non-trivial $\gamma\in\pi$, $\rho(\gamma)\in\proxGTheta$, which allows us to define a projector $p_{j}(\gamma)$ for any $j\in\Theta$. Consider a $k-$tuple $\ub\alpha = (\alpha_1,\dots,\alpha_k)\in\pi^k$ as a well as a $k-$tuple $I = (i_j)_{j = 1}^k$ with $i_j\in \Theta$. Then the \textit{correlation function associated to $\ub\alpha$ and $I$} is the function
\begin{align*}
\T^I(\ub\alpha)\colon\ThetaAnosovReps&\to\R\\
[\rho]&\mapsto \trace\left(p_{i_k}(\rho(\alpha_k))p_{i_{k-1}}(\rho(\alpha_{k-1}))\cdots p_{i_1}(\rho(\alpha_1))\right).
\end{align*}

Since this function can be realized as an induced invariant function through the function

\begin{align*}
	C^I\colon \proxGTheta^k&\to\R\\
	(g_1,\dots,g_k)&\mapsto \trace\left(\prod_{j = 1}^kp_{i_j}(g_j)\right),
\end{align*}
we immediately obtain the following corollary.

\begin{corollary}
    The correlation functions $T^I(\ub\alpha)$ on $\ThetaAnosovReps$ are realized as invariant functions induced by the tuple $\ub\alpha$.
\end{corollary}

\subsubsection{Functions associated to boundary maps}\label{sec : functions associated to boundary maps}
Here we define functions depending on distinct points on the boundary $\partial_\infty\pi$ by heavily using the existence and dynamical properties of the boundary map attached to $\Theta-$Anosov representations. The main examples of such functions will be cross and triple ratios.\\

Denote by $(\partial_\infty\pi)^{[k]}$ the set of pairwise distinct ordered $k-$tuples in $(\partial_\infty\pi)^k$. Let $\phi\colon\flags_\Theta^{[k]}\to\R$ be a smooth function invariant under the diagonal action of $G$ on $\flags_\Theta^{[k]}$. Given a tuple $\ub x = (x_1,\dots,x_k)\in(\partial_\infty\pi)^{[k]}$, we obtain a function 
\begin{align*}
\phi_{\ub x}\colon\ThetaAnosovReps&\to\R\\
[\rho]&\mapsto \phi(\flags_\rho(x_1),\dots,\bdryMap(x_k)),
\end{align*}
which is well defined by the $\rho-$equivariance of the limit map, the invariance of $\phi$ and the transversality of the boundary map. In the following lemma, we show that the function $\phi_{\ub x}$ is a limit of smooth induced invariant functions.

\begin{lemma}\label{lem : functions on boundary are limits of invariant functions}
    Let $\phi\colon\flags_\Theta^{[k]}\to\R$ be a smooth function invariant under the diagonal action of $G$. The function
    \begin{align*}
        f^\phi\colon \proxGTheta^{[k]}&\to\R\\
        (g_1,\dots,g_k)&\mapsto \phi(g_1^+,\dots,g_k^+),
    \end{align*}
    is a smooth $G-$invariant function. Moreover, let $\ub x = (x_1,\dots,x_k)\in(\partial_\infty\pi)^{[k]}$ and let $(\ub\gamma^n)_{n\in\N}=(\gamma_1^n,\dots,\gamma_k^n)_{n\in\N}\subset \pi^k$ be a sequence such that $((\gamma_1^n)^+,\dots,(\gamma_k^n)^+)\in (\partial_\infty\pi)^{[k]}$ for every $n\in\N$ and $(\gamma_i^n)^+\xrightarrow[]{n\to\infty}x_i$ for $i = 1,\dots,k$. Then 
    \[
    \lim_{n\to\infty}f^\phi_{\ub\gamma^n}([\rho]) = \phi_{\ub x}([\rho])
    \]
    for every $[\rho]\in\ThetaAnosovReps$.
\end{lemma}

\begin{proof}
    The function $f^\phi$ is smooth since the map $\proxGTheta^{[k]}\to\flags_\Theta^{[k]}$ sending elements to their attracting flags and $\phi$ are smooth. To see that $f$ is an invariant function, we compute that
    \[
    f^\phi((hg_1h^{-1}),\dots,hg_kh^{-1}) = \phi(h\cdot g_1^+,\dots,h\cdot g_k^+) = \phi(g_1^+,\dots,g_k^+) = f^\phi(g_1,\dots,g_k),
    \]
    where the second equality follows from the $G-$invariance of $\phi$. For the last part, let $[\rho]\in\ThetaAnosovReps$. Since for any $\Theta-$Anosov representation $\rho$ the boundary map $\bdryMap\colon\partial_\infty\pi\to \flags_\Theta$ is dynamics preserving, we have that $\rho(\gamma)^+ = \bdryMap(\gamma^+)$ for any $\gamma\in\pi$. Hence,
    \[
    f^\phi_{\ub\gamma^n}([\rho]) = \phi(\flags_\rho((\gamma_1^n)^+),\dots,\flags_\rho((\gamma_k^n)^+)).
    \]
    Then
    \begin{align*}
        \lim_{n\to\infty}f^\phi_{\ub\gamma^n}([\rho]) &= \lim_{n\to\infty} \phi(\flags_\rho((\gamma_1^n)^+),\dots,\flags_\rho((\gamma_k^n)^+))\\
        &= \phi(\flags_\rho(x_1),\dots,\flags_\rho(x_k)) = \phi_{\ub x}([\rho]),
    \end{align*}
    where the second equality follows by the smoothness of $\phi$, and the continuity of $\flags_\rho$.
\end{proof}

With this construction and Corollary \ref{cor : if supporting subsurfaces are disjoint, flows commute}, we deduce when Hamiltonian flows of functions on tuples in $\partial_\infty\pi$ commute.
For $\ub x \in (\partial_\infty \pi)^{k}$, say that (the isotopy class of) an essential subsurface $S_{\ub x}$ supports $\ub x$ if the projection of set of all geodesics joining points in $\ub x$ is contained in $S_{\ub x}$, for some hyperbolic metric on $S$.

\begin{corollary}
    Let $k_1,k_2\in\N$ and $\phi^i\colon \flags_\Theta^{[k_i]}\to\R$ be a smooth invariant function for $i = 1,2$. Let $\ub x\in(\partial_\infty\pi)^{[k_1]}$ and $\ub y\in(\partial_\infty\pi)^{[k_2]}$. If there are supporting subsurfaces $S_{\ub x}$ and $S_{\ub y}$  for $\ub x$ and for $\ub y$
    respectively such that $S_{\ub x}\cap S_{\ub y}=\emptyset$, then 
    \[
    \{\phi^1_{\ub x},\phi^2_{\ub y}\} \equiv 0
    \]
    on $\ThetaAnosovReps$.
\end{corollary}
\begin{proof}
    Since the subsurfaces $S_{\ub x}$ and $S_{\ub y}$ are disjoint, we may pick sequences of tuples $(\ub\alpha_n)_{n\in\N}$ in $\pi_1(S_{\ub x})^{k_1}$ and $(\ub\beta_n)_{n\in\N}$ in $\pi_1(S_{\ub y})^{k_2}$ whose attracting endpoints approximate $\ub x$ and $\ub y$ respectively. Then by Corollary \ref{cor : if supporting subsurfaces are disjoint, flows commute}
    \[
    \{f^{\phi^1}_{\ub\alpha_n},f^{\phi^2}_{\ub\beta_n}\}\equiv 0
    \]
    for every $n\in\N$ and using the notation from Lemma \ref{lem : functions on boundary are limits of invariant functions}. We conclude by using Lemma \ref{lem : functions on boundary are limits of invariant functions} again to get that $f^{\phi^1}_{\ub\alpha_n}\xrightarrow[]{n\to\infty}\phi^1_{\ub x}$ and $f^{\phi^2}_{\ub\beta_n}\xrightarrow[]{n\to\infty}\phi^2_{\ub y}$ point-wise.
\end{proof}

There is a variety of functions defined on $\flags_\Theta^{[k]}$. When $k = 4$, there are different notions of cross ratios (also called quadruple or double ratios), see for example \cite[Section 9.4]{FockGoncharov}, \cite[Definitions 3.2, 4.2]{CrossRatios_Labourie}, \cite[Definition 3.5]{CollarLemmaHyperconvex_BeyrerPozzetti}, \cite[Section 1.4]{BonahonDreyer_GeneralLaminations}, \cite{HartnickStrubel_CrossMaximal},\cite{CollarTheta}. In the case when $k = 3$, there is the triple ratio, defined in different ways in \cite[Section 9.4]{FockGoncharov}, \cite[Remark 4, pp. 149]{CrossRatios_Labourie}, \cite[Section 1.2]{BonahonDreyer_GeneralLaminations}. By Lemma \ref{lem : functions on boundary are limits of invariant functions}, and the proof of Corollary \ref{cor : Bonahon Dreyer length functions are invariant functions} we obtain the following.

\begin{corollary}\label{cor : Labourie's cross ratios are limits of induced invariant functions}
	Cross and triple ratio functions on $\ThetaAnosovReps$ can be realized as induced invariant functions.
\end{corollary}

Bonahon-Dreyer use the cross and triple ratios defined in \cite{BonahonDreyer_GeneralLaminations} to give coordinates on the Hitchin component.

\begin{corollary}\label{cor : Bonahon-Dreyer coordinates are limits of induced invariant functions}
    Each Bonahon-Dreyer coordinate of $\Hit d$ can be realized as an induced invariant function.
\end{corollary}

In \cite{FlowsPGLVHitchin_SWZ}, Sun-Wienhard-Zhang define eruption and hexagon flows on $\Hit d$. The Hamiltonian functions of these flows are given in a reparametrization of Bonahon-Dreyer coordinates adapted to a maximal lamination with finitely many leaves, together with length functions on closed curves \cite[Theorem 8.22]{FlowsPGLVHitchin_SWZ}.

\begin{corollary}
    The Hamiltonian functions of eruption and hexagon flows on $\Hit d$ can be realized as induced invariant functions.
\end{corollary}

\section{Graph of groups decomposition, its representations and leaves}\label{sec : graph of groups}
We begin this section by recalling some background about graphs of groups and their fundamental groups. We then describe the representations of fundamental groups of graphs of groups into $G$, as well as the Zariski tangent spaces of the spaces of representations. Using this language, we introduce the main notion of this section, namely leaves supported on vertices, which are particular subsets of the character variety associated to vertices in the graph of groups. This definition is used to define (infinitesimal) subsurface deformations in Definition \ref{def : inf and subsurface deformation}. In the following subsection, we construct an explicit graph of groups decomposition associated to a decomposition of the surface, and describe its fundamental group in Proposition \ref{prop: fundamental group of graph of groups form of gluing loops}. This description is used in the proof of Theorem \ref{thm : hamiltonian vector field is infinitesimal subsurface deformation}, which states that the Hamiltonian vector field of an invariant function is an infinitesimal subsurface deformation. We finish by describing the Zariski tangent spaces to the leaves, which will also be used to prove Theorem \ref{thm : hamiltonian vector field is infinitesimal subsurface deformation} and Lemma \ref{lem : leaves are smooth for hitchin}, which says that the foliation by leaves in the Hitchin component is smooth.

\subsection{The graph of groups and representations}
We begin by recalling some definitions and notation from \cite{Serre_Trees}. Given an oriented graph $Y$ with vertices $V$ and edges $E$, we denote by $o(y)$ and $t(y)$ the origin and respectively the endpoint of the edge $y\in E$. Moreover, we denote by $\overline{y}\in E$ the edge $y$ with the opposite orientation.

\begin{definition}\cite[Section 4.4, Definition 8; and Section 5.1]{Serre_Trees}\label{def:graph of groups and fundamental group} 
 \begin{enumerate}
     \item A \textit{graph of groups} $(\Gamma,Y)$ consists of an oriented graph $Y=(V,E)$, a group $\Gamma_v$ for each vertex $v\in V$, and a group $\Gamma_y$ for each edge $y\in E$ together with monomorphisms $\Gamma_y\to\Gamma_{t(y)}$ (denoted by $a\mapsto a^y$); with the additional condition that $\Gamma_y=\Gamma_{\overline{y}}$.
     \item Given a graph of groups $(\Gamma,Y)$, choose a maximal tree $T$ in $Y$. Then the \textit{fundamental group of $(\Gamma,Y)$ at $T$}, denoted by $\pi_1(\Gamma,Y,T)$, is the group generated by $\{\Gamma_v,g_y\st v\in V, y\in E\}$, where $g_y$ is an added letter we call a \textit{gluing loop}, subject to the relations

\begin{align}
	g_ya^yg_y^{-1}=a^{\overline{y}}  \textnormal{ for every }y\in E\textnormal{ and } a\in \Gamma_y,\label{relation:graph of groups relation 1} \\
	g_{\overline{y}}=g_y^{-1} \textnormal{ for every }y\in E\label{relation:graph of groups relation 1 prime}\textnormal{ and}\\
	 g_y = 1 \textnormal{ if } y\in T.\label{relation:graph of groups relation 2}
\end{align}
     \end{enumerate}
\end{definition}

As stated in I.5 Proposition 20 in \cite{Serre_Trees}, different choices of a maximal tree in $Y$ define isomorphic fundamental groups.\\

We now describe the space of representations from the fundamental group of a graph of groups to a Lie group $G$. Given a graph of groups $(\Gamma,Y)$ together with a choice of maximal tree $T$ in $Y$, we consider the space of representations $\Hom(\pi_1(\Gamma,Y,T),G)$.\\

A representation $\rho\in\Hom(\pi_1(\Gamma,Y,T),G)$ gives rise to representations $\rho_v\colon\Gamma_v\to G$ for each vertex $v\in V$ as well as $\rho_y\colon \langle g_y\rangle\to G$ for every edge $y\in E$. The relations \eqref{relation:graph of groups relation 1},\eqref{relation:graph of groups relation 1 prime} and \eqref{relation:graph of groups relation 2} in $\pi_1(\Gamma,Y,T)$ imply the following relations on the induced representations:
\begin{align}
	\rho_y(g_y)\rho_{t(y)}(a^y)\rho_y(g_y)^{-1} = \rho_{o(y)}(a^{\overline{y}}) \textnormal{ for every } y\in E\textnormal{ and every } a\in \Gamma_y,\label{condition: representation on graph of groups non tree edges}\\
	\rho_{\overline{y}}(g_{\overline{y}}) = \rho_y(g_y)^{-1} \textnormal{ for every } y\in E,\textnormal{ and}\label{condition: representation on graph of groups gy and gy^-1}\\
	\rho_y(g_y) = e\textnormal{ if }y\in T.\label{condition: representation on graph of groups tree edges}
\end{align}

Conversely, given a family of representations $\{\rho_v\colon \Gamma_v\to G,\rho_y\colon\langle g_y\rangle \to G \st v\in V, y\in E\}$ satisfying conditions \eqref{condition: representation on graph of groups non tree edges},\eqref{condition: representation on graph of groups gy and gy^-1}, and \eqref{condition: representation on graph of groups tree edges}, we obtain a unique representation $\rho\colon \pi_1(\Gamma,Y,T)\to G$ (by definition of the fundamental group $\pi_1(\Gamma,Y,T)$). By condition \eqref{condition: representation on graph of groups gy and gy^-1} and \eqref{condition: representation on graph of groups tree edges}, we have that $\rho_y\equiv e$ whenever $y\in T\cup \overline{T}$; and we denote $T\cup \overline{T}$ by $T_u$ (the unoriented tree) and its complement by $T_u^c$.  Dropping the representations $\rho_y$ for $y\in T_u$, we get the additional condition
\begin{equation}\label{condition : conjugation along tree edges is nothing}
	\rho_{t(y)}(a^y) = \rho_{o(y)}(a^{\overline y}) \quad\textnormal{ for every } y\in T_u\textnormal{ and }a\in \Gamma_y.
\end{equation}

Thus, there is an identification
\begin{gather*}
\Hom(\pi_1(\Gamma,Y,T),G)\cong \\
\{\rho_v\colon \Gamma_v\to G,\rho_y\colon \langle g_y\rangle \to G \st v\in V, y\in T_u^c, \textnormal{ satisfying } \eqref{condition: representation on graph of groups non tree edges},\eqref{condition: representation on graph of groups gy and gy^-1},\eqref{condition : conjugation along tree edges is nothing}\}.
\end{gather*}

Through this identification, we fix the following notation.
\begin{notation}
	A representation $\rho\in\Hom(\pi_1(\Gamma,Y,T),G)$ will be written as a tuple $(\rho_\vv,\rho_\yy)$, where $\rho_\vv = \{\rho_v\colon \Gamma_v\to G\,|\, v\in V\}$ and $\rho_\yy = \{\rho_y\colon \langle g_y\rangle\to G\,|\, y\in T_u^c\}$ are the components of the representation $\rho$.
\end{notation}

\subsection{Zariski tangent spaces}
In this section, we describe the Zarisiki tangent spaces of $\Hom(\pi_1(\Gamma,Y,T),G)$.\\

\begin{lemma}\label{lemma : conditions for tangent space graph of groups representations}
	\sloppy Let $\rho = (\rho_\vv,\rho_\yy)\in\Hom(\pi_1(\Gamma,Y,T),G)$. The Zariski tangent space $\T_\rho\Hom(\pi_1(\Gamma,Y,T),G)$ is identified with the subspace of pairs of cocycles $(u_\vv,u_\yy)$, where $u_\vv = \{u_v\in Z^1(\Gamma_v,\gfr_{\Ad\rho_v})\st v\in V\}$ and $u_\yy = \{u_y\in Z^1(\langle g_y\rangle,\gfr_{\Ad\rho_y})\st y\in T_u^c\}$ which satisfy the following conditions
		\begin{align}
		u_y(g_y)+\Ad_{\rho_y(g_y)}(u_{t(y)}(a^y))-\Ad_{\rho_{o(y)}(a^{\overline{y}})}(u_y(g_y)) = u_{o(y)}(a^{\overline y}),\,\forall y\in T_u^c,\forall a\in \Gamma_y\label{tangent space condition graph of groups conjugation part},\\
		u_{\overline y}(g_{\overline y}) = -\Ad_{\rho_{\overline y}(g_{\overline y})}(u_y(g_y)),
		\quad\forall y\in T_u^c, \label{tangent space condition graph of groups orientation flip part}\textnormal{ and}\\
		u_{o(y)}(a^{\overline y}) = u_{t(y)}(a^y),\quad\forall y\in T_u \label{tangent space condition graph of groups tree edges}.
	\end{align}
\end{lemma}

\sloppy Before the proof of the lemma, we fix some notation. For a representation $(\rho_\vv,\rho_\yy)\in\Hom(\pi_1(\Gamma,Y,T),G)$, we denote by $Z^1(\pi_1(\Gamma,Y,T),\gfr_{\Ad(\rho_\vv,\rho_\yy)})\cong\T_{(\rho_\vv,\rho_\yy)}\Hom(\pi_1(\Gamma,Y,T),G)$ the set of pairs of cocycles $(u_\vv,u_\yy)$ satisfing conditions \eqref{tangent space condition graph of groups conjugation part}, \eqref{tangent space condition graph of groups orientation flip part} and \eqref{tangent space condition graph of groups tree edges}.\\

\begin{proof}
	Consider a small path of representation $(\rho_\vv^s,\rho_\yy^s)$ for $s\in(-\eps,\eps)$ with $(\rho^0_\vv,\rho^0_\yy) = (\rho_\vv,\rho_\yy)$. As in \cite{SymplecticNature_Goldman}, the paths of representations are infinitesimally described by cocycles $(u_\vv,u_\yy)$ so that
\begin{equation}\label{eq: variation of graph of groups representations with cocycles}
	\rho^s_v(\gamma_v) = \exp(su_v(\gamma_v)+\mathcal O(s^2))\rho_v(\gamma_v),\quad \rho^s_y(g_y) = \exp(su_y(g_y)+\mathcal O(s^2))\rho_y(g_y),
\end{equation}
for every $v\in V$, every $\gamma_v\in \Gamma_v$ and any $y\in T_u^c$. Arguing as in \cite[Section 1.2]{SymplecticNature_Goldman}, we see that $u_v\in Z^1(\Gamma_v,\gfr_{\Ad\rho_v})$ and $u_y\in Z^1(\langle g_y\rangle,\gfr_{\Ad\rho_y})$.\\

Differentiating condition \eqref{condition: representation on graph of groups non tree edges} and evaluating at $s = 0$, we get that
	\begin{align*}
		u_y(g_y)\rho_y(g_y)\rho_{t(y)}(a^y)\rho_y(g_y)^{-1}+\rho_y(g_y)u_{t(y)}(a^y)\rho_{t(y)}(a^y)\rho_y(g_y)^{-1}\\
		+\rho_y(g_y)\rho_{t(y)}(a^y)\rho_y(g_y)^{-1}(-u_y(g_y)) = u_{o(y)}(a^{\overline y})\rho_{o(y)}(a^{\overline y}),
	\end{align*}
	where we abuse notation for the derivative of left and right multiplication by an element of $G$ by simply multiplication by that element. This computation together with using conditions \eqref{condition: representation on graph of groups non tree edges} and \eqref{condition: representation on graph of groups gy and gy^-1} yields condition \eqref{tangent space condition graph of groups conjugation part}. Differentiating condition \eqref{condition: representation on graph of groups gy and gy^-1}, we get that 
	\[
		u_{\overline y}(g_{\overline y})\rho_{\overline y}(g_{\overline y}) = -\rho_{y}(g_y)^{-1}u_y(g_y).
	\]
	Reusing condition \eqref{condition: representation on graph of groups gy and gy^-1}, we get \eqref{tangent space condition graph of groups orientation flip part}. Differentiating condition \eqref{condition : conjugation along tree edges is nothing} immediately gives \eqref{tangent space condition graph of groups tree edges}.
\end{proof}

\subsection{Leaves supported on vertices}\label{sec : leaves supported on vertices}
Associated to a graph of groups, we define the projection

\begin{align*}
	\widehat p\colon \Hom(\pi_1(\Gamma,Y,T),G)&\to\prod_{v\in V}\Hom(\Gamma_v,G)/G\\
	(\rho_\vv,\rho_\yy)&\mapsto([\rho_v])_{v\in V}
\end{align*}

\begin{remark}
	The projection of a smooth point does not necessarily consist of smooth points.
\end{remark}

Moreover, for a vertex $w\in V$, we define the \textit{projection onto the complement of $w$}
\begin{align*}
	\widehat p_w\colon \Hom(\pi_1(\Gamma,Y,T),G)&\to\prod_{v\in V\smallsetminus \{w\}}\Hom(\Gamma_v,G)/G\\
	(\rho_\vv,\rho_\yy)&\mapsto ([\rho_v])_{v\in V\smallsetminus \{w\}}
\end{align*}

Both of the above projection maps descend to the quotient by conjugation, and we obtain a pair of projections
\begin{align}
	p\colon \Hom(\pi_1(\Gamma,Y,T),G)/G&\to \prod_{v\in V}\Hom(\Gamma_v,G)/G\label{eq: definition projection map to character varieties of vertex groups}\\
	[\rho_{\vv},\rho_\yy]&\mapsto ([\rho_v])_{v\in V}, \nonumber
\end{align}
and
\begin{align}
	p_w\colon \Hom(\pi_1(\Gamma,Y,T),G)/G &\to \prod_{v\in V\smallsetminus \{w\}}\Hom(\Gamma_v,G)/G\\
	[\rho_\vv,\rho_\yy]&\mapsto ([\rho_v])_{v\in V\smallsetminus \{w\}}.\nonumber
\end{align}
Denote by $\widehat{\mathbf H}_w\coloneqq \prod_{v\in V\smallsetminus \{w\}}\Hom(\Gamma_v,G)$ and by $\mathbf H_w\coloneqq \prod_{v\in V\smallsetminus\{w\}}\Hom(\Gamma_v,G)/G$, the set of \textit{representations in the complement of $w$}. With these projections at hand, we make the following definition.

\begin{definition}\label{def : leaves supported on vertices}
	Let $w$ be a vertex in $V$, and $\rho$ a representation in $\Hom(\pi_1(\Gamma,Y,T),G)$. The \textit{leaf supported on $w$ at $\rho$}, denoted by $\widehat \leaf_{w,\rho}$, is the preimage
	
	\[
		\widehat \leaf_{w,\rho}\coloneqq \widehat p_w^{-1}(\widehat p_w(\rho)).
	\]
	Similarly, we define 
	\[
		\leaf_{w,[\rho]}\coloneqq p_w^{-1}(p_w([\rho])), 
	\]
	and also call it the leaf supported on $w$ at $[\rho]$. In the case that the vertex $w$ is clear from the context, we will drop this subscript in the notation.
\end{definition}

\begin{remark}\label{rmk : alternative definition for leaves}
	By definition of the projection, we see that the leaves can also be described as
	\[
		\widehat\leaf_{w,\rho}\coloneqq \left\{\varphi = (\varphi_\vv,\varphi_\yy)\in\Hom(\pi_1(\Gamma,Y,T),G)\st [\varphi_v]=[\rho_v],\,\,\forall v\in V\smallsetminus \{w\}\right\},
	\]
	and similarly
	\[
		\leaf_{w,[\rho]}\coloneqq \left\{[\varphi] = [(\varphi_\vv,\varphi_\yy)]\in\Hom(\pi_1(\Gamma,Y,T),G)/G\st [\varphi_v]=[\rho_v],\,\,\forall v\in V\smallsetminus \{w\}\right\}.
	\]
	This allows us to make some abuse of notation and denote $\leaf_{w,[\rho]}$ by $\leaf_{w,\rho}$.
\end{remark}

\begin{remark}
	In some cases, instead of defining the leaves over a representation, we define it over a conjugacy class of a tuple of conjugacy classes. Namely, for a tuple $\mathcal R = (\mathcal R_v)_{v\in V\smallsetminus \{w\}}\in \widehat{\mathbf H}_w$, we denote by  $[\mathcal R] = ([\mathcal R_v])_{v\in V\smallsetminus \{w\}}\in\mathbf H_w$, and abuse notation by writing
	\[
		\widehat \leaf_{w,\mathcal R}\coloneqq \widehat p_w^{-1}([\mathcal R]).
	\]
	Similarly, we define $\leaf_{w,\mathcal R}=\leaf_{w,[\mathcal R]}$ as the preimage $p_w^{-1}([\mathcal R])$, which by Remark \ref{rmk : alternative definition for leaves} is well defined. It may be the case however, that $\leaf_{w,\mathcal R}$ is empty if $[\mathcal R]$ does not lie in the image of the projection $\widehat p_w$.
\end{remark}

Since the leaves are defined as the fibers of a map, we immediately obtain the following lemma; from which the name for leaf becomes apparent.

\begin{lemma}\label{lem: disjoint union "relative representation varieties"}
	 The representation variety $\Hom(\pi_1(\Gamma,Y,T),G)$ decomposes as a disjoint union
	\[
	\Hom(\pi_1(\Gamma,Y,T),G) = \bigsqcup_{[\mathcal R]\in \mathbf H_w} \widehat\leaf_{w,\mathcal R},
	\]
	which descends to a decomposition 
	\[
	\Hom(\pi_1(\Gamma,Y,T),G)/G = \bigsqcup_{[\mathcal R]\in \mathbf H_w} \leaf_{w,\mathcal R}.
	\]
\end{lemma}

We will thus denote by
\begin{equation}\label{eq : def of decomposition into leaves associated to w}
\widehat\leaf(w)\coloneqq \{\widehat\leaf_{w,\mathcal R}\st[\mathcal R]\in \mathbf H_w\}
\end{equation}
the \textit{decomposition into leaves associated to $w$}. Similarly, we obtain a decomposition
\[
\leaf(w)\coloneqq \{\leaf_{w,\mathcal R}\st[\mathcal R]\in \mathbf H_w\}.
\]

The leaves carry an additional structure, depending on the Lie group $G$ and the representation it is supported at.
\begin{lemma}\label{lem: "relative representation varieties" are analytic algebraic or semialgebraic}
	Let $(\Gamma,Y)$ be a graph of groups such that $\Gamma_v$ is finitely presented for every $v\in V$, and let $T$ be a maximal tree in $Y$. Fix a vertex $w\in V$ and a tuple of representations $\mathcal R = (\mathcal R_v)_{v\in V\smallsetminus\{w\}}\in\widehat{\mathbf H}_w$ in the complement of $w$. 
 
 \begin{enumerate}
     \item Assume $G$ is a Lie group equipped with an analytic atlas. The leaf $\widehat \leaf_{w,\mathcal R}$ supported on $w$ at $\mathcal R$ is naturally an analytic subvariety of $G^k$ for some $k\in\N$.\label{lemma : if analytic Lie group}
     \item If $G$ is a complex algebraic group, then $\widehat \leaf_{w,\mathcal R}$ is an algebraic subvariety.\label{lemma : if complex}
     \item If $G$ is a real algebraic group, then $\widehat \leaf_{w,\mathcal R}$ is in general a semialgebraic subvariety.\label{lemma : if real}
     \begin{enumerate}
         \item In the case when the conjugacy classes $\left[\mathcal R_v\right]$ are all algebraic, the leaf $\widehat \leaf_{w,\mathcal R}$ is algebraic.\label{lemma : if conjugacy classes are algebraic then algebraic}
     \end{enumerate}
 \end{enumerate}
\end{lemma}
\begin{proof}
	For part \ref{lem: "relative representation varieties" are analytic algebraic or semialgebraic} \eqref{lemma : if analytic Lie group} note that a conjugacy class of an element $g\in G$ is a smooth submanifold isomorphic to $G/Z(g)$, giving the conjugacy class the structure of an analytic manifold. Assume now that for every $v\in V$, the vertex group $\Gamma_v$ is generated by $n_v$ elements and has presentation $\langle \gamma_v^1,\dots,\gamma_v^{n_v} |Q_v\rangle$. Consider the injection
	\begin{align*}
		\Hom(\pi_1(\Gamma,Y,T),G)&\to \left(\prod_{v\in V\smallsetminus\{w\}}G^{n_v}\right)\times G^{n_w}\times G^{|T_u^c|}\\
		(\varphi_\vv,\varphi_\yy)&\mapsto \left((\varphi_v(\gamma_v^i)_{i = 1}^{n_v})_{v\in V\smallsetminus \{w\}},(\varphi_w(\gamma_w^i))_{i = 1}^{n_w},(\varphi_y(g_y))_{y\in T_u^c}\right).
	\end{align*}
If $\leaf_{w,\mathcal R}$ is empty, there is nothing to prove. So we assume that it is nonempty, and denote by $\rho_v\colon\Gamma_v\to G$ the representation $\mathcal R_v$ for every $v\in V\smallsetminus \{w\}$. Let $C_v^i\coloneqq [\rho_{v}(\gamma_v^i)]$ be the conjugacy class of the generators for $i = 1,\dots,n_v$ and for every $v\in V\smallsetminus \{w\}$. Then the direct product of conjugacy classes $C_v\coloneqq C_v^1\times\cdots\times C_v^{n_v}\subset G^{n_v}$ is an analytic submanifold. By the alternative definition of the leaves in Remark \ref{rmk : alternative definition for leaves}, the above map identifies the leaf $\widehat\leaf_{w,\mathcal R}$ with the subset of 
\[
	\left(\prod_{v\in V\smallsetminus \{w\}} C_v^i\right)\times G^{n_{w}}\times G^{|T_u^c|}
\]
cut out by the relations $Q_v$, all of which are analytic maps $G^k\to G$. Thus, $\widehat\leaf_{w,\mathcal R}$ carries the structure of an anlytic variety.\\

The other cases follow similarly by noticing that when $G$ is a complex algebraic group, conjugacy classes are algebraic subvarieties. On the other hand, in a real algebraic group, conjugacy classes are only semialgebraic varieties in general. In the special case when all the conjugacy classes are algebraic varieties, the leaves become algebraic varieties as well.
\end{proof}

\subsection{Graph of groups associated to a surface decomposition}\label{sec : Constructing graph of groups}
In this section, we describe how a decomposition of a surface into subsurfaces gives rise to a graph of groups, and describe its fundamental group in terms of this decomposition with explicit gluing loops in Proposition \ref{prop: fundamental group of graph of groups form of gluing loops}. We use the explicit description in the proof of Theorem \ref{thm : hamiltonian vector field is infinitesimal subsurface deformation} to show that the Hamiltonian vector field of induced invariant functions is an infinitesimal subsurface deformation. \\

 For this section let $S = \bigcup_{i = 1}^kS_i$ be a decomposition of $S$ into compact connected surfaces with boundary. We assume for the rest of the article, that each surface $S_i$ is essential and homotopically non-trivial. Assume that the intersections of the subsurfaces are either empty, or boundary components.

\subsubsection{Construction of the graph of groups}\label{sec : construction for graph of groups}
Choose base points $p_1,\dots,p_k$ in the interior of each $S_1,\dots,S_k$ respectively. Up to relabelling we may assume that $p = p_1$. Then we specify (arbitrarily) an orientation on each curve in the boundary of the subsurfaces. Let $a_1,\dots,a_n$ be the collection of boundary curves.\\

\begin{figure}[ht]
\centering
\includesvg[scale=0.5]{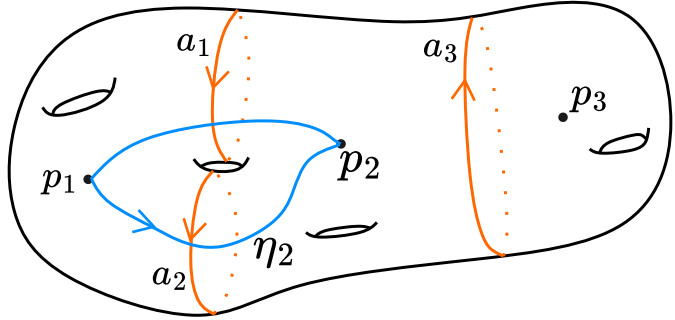}	
\caption{Example of a graph of groups decomposition into $3$ subsurfaces. The gluing loop is in blue according the maximal tree as in Figure \ref{fig : graph for graph of groups example}.}
\label{fig : Graph of groups example decomposition}
\end{figure}

With the chosen orientations, the graph of groups $(\Gamma,Y)$ is given as follows. Let the set of vertices be $V = \{v_1,\dots,v_k\}$ associated to each subsurface, and let $\pi_1(S_i,p_i)$ be the vertex group of $v_i$ for each $i = 1,\dots,k$. Then construct the following edges:

\begin{itemize}
    \item For $j = 1,\dots,n$, let $y_{j}$ be an edge from the subsurface to the left of $a_j$ to the subsurface to the right of $a_j$;
    \item For each of the above edges, also add the edge with the opposite orientation, i.e. add $\overline{y}_{1},\dots,\overline{y}_{n}$.
\end{itemize}

The graph of groups associated to the decomposition in Figure \ref{fig : Graph of groups example decomposition} is shown in Figure  \ref{fig : graph for graph of groups example}. \\

\begin{figure}[ht]
	\centering
	\includesvg[scale=0.5]{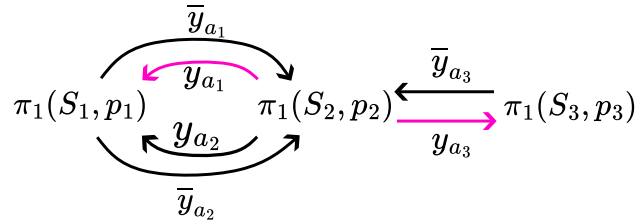}
	\caption{Graph of groups associated to the decomposition in Figure \ref{fig : Graph of groups example decomposition} with a choice of maximal tree in pink.}
	\label{fig : graph for graph of groups example}
\end{figure}

For ease of notation, we will often confuse a vertex $v_i$ with either the vertex group $\pi_1(S_i,p_i)$ or with the index $i$ itself, which will be clear from the context.\\

Now assign to each edge the group $\Z$ and we define the monomorphisms from the edge groups to the vertex groups. The monomorphsisms are specific homotopy classes of loops associated to the boundary components. For each $j\in\{1,\dots,n\}$ and $i\in\{1,\dots,k\}$ such that $a_j$ is a boundary component of $S_i$, choose a path $\nu^j_i$ from the base point $p_i$ to the curve $a_j$, with the additional conditions that $\coai{\nu}{j}^j = \nu_{t(\overline{y}_{j})}^j$ and that the endpoints of $\coai{\nu}{j}^j$ and $\ctai{\nu}{j}^j$ are the same. These paths allow us to define loops in the fundamental groups of the subsurfaces. Namely, let
\begin{equation}\label{eq: definition boundary curves based at different subsurfaces}
        \alpha^j_i\coloneqq[\nu^j_i*a_j*(\nu_i^j)^{-1}]\in\pi_1(S_i,p_i),
\end{equation}
where once again, $i$ is such that $a_j$ is a boundary curve of $S_i$. With these paths, we obtain the monomorphisms
\begin{align}
\Z&\to\pi_1(\coai{S}{j},\coai{p}{j})\label{eq:GraphOfGroups:Monomorphsisms a_i}\\
1&\mapsto 1^{y_{j}}\coloneqq \coai{\alpha}{j}^j,\nonumber
\end{align}
and
\begin{align}
\Z&\to\pi_1(S_{o(\overline{y}_{j})},p_{o(\overline{y}_{j})})\label{eq:GraphOfGroups:Monomorphsisms a_i reversed}\\
1&\mapsto 1^{\overline{y}_{j}}\coloneqq \alpha^j_{o(\overline{y}_{j})}=\ctai{\alpha}{j}^j,\nonumber
\end{align}

\begin{figure}[ht]
	\centering
	\includesvg[scale=0.5]{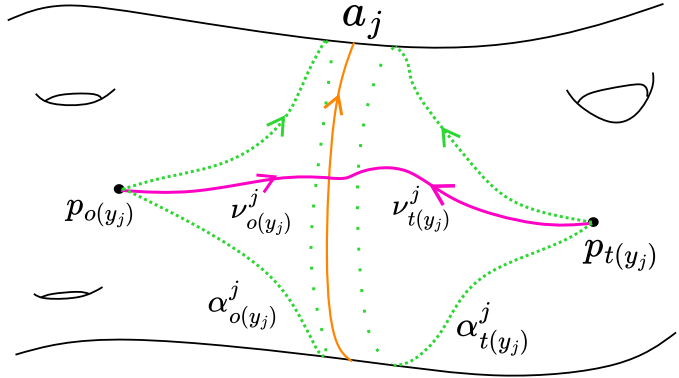}
	\caption{Identifying boundary curves ($a_j$ in solid orange) with elements $\alpha^j_{o(y_j)}$ (and $\alpha^j_{t(y_j)}$) in dotted green in the different fundamental groups through paths $\nu^i_j$ in solid pink.}
	\label{fig : connecting paths}
\end{figure}
see for example Figure \ref{fig : connecting paths}.  This finishes the construction of the graph of groups.

\subsubsection{The fundamental group of the graph of groups}
Begin by choosing a maximal tree $T$ in $Y$, as for example in pink in Figure \ref{fig : graph for graph of groups example}. Before describing the additional letters from Definition \ref{def:graph of groups and fundamental group}, we define homomorphisms from the vertex groups to the fundamental group $\pi_1(S,p)$.\\

Let $v_j\in V$ be a vertex in the graph of groups and recall that $p = p_1$. Since $T$ is a tree, there exists a unique geodesic $(y_{i_1},\dots,y_{i_m})$ in $T$ from $v_1$ to $v_j$. This geodesic defines a path from $p_1$ to $p_j$ as follows. Let
\[
    z_j\coloneqq (\coai{\nu}{i_1}^{i_1}*(\ctai{\nu}{i_1}^{i_1})^{-1})*(\coai{\nu}{i_2}^{i_2}*(\ctai{\nu}{i_2}^{i_2})^{-1})*\cdots * (\coai{\nu}{i_m}^{i_m}*(\ctai{\nu}{i_m}^{i_m})^{-1}).
\]
The arcs in the above path are paired to see that each term is a path from $\coai{p}{i_l}$ to $\ctai{p}{i_l}$, for $l = 1,\dots,m$. This allows us to define the isomorphisms for $j = 1,\dots,k$
\begin{align}
    \psi_j\colon\pi_1(S,p_j)&\to\pi_1(S,p_1)
    \label{eq:definition of inclusion of vertex groups in graph of groups}\\
    [\gamma]&\mapsto[z_j*\gamma*z_j^{-1}].\nonumber
\end{align}
In particular, the isomorphisms induce homomorphisms $\pi_1(S_j,p_j)\to\pi_1(S,p)$.
\begin{convention}\label{convention : Gammaj is image of psij}
    For the rest of the article, we will write
    \begin{equation}\label{eq : Gamma_j is image of psi_j}
        \Gamma_j\coloneqq \psi_j\left(\pi_1(S_j,p_j)\right).
    \end{equation}
    We will now replace the vertex groups $\pi_1(S_j,p_j)$ with the groups $\Gamma_j$ in the graph of groups $(\Gamma,Y)$ (which depend on the choice of maximal tree $T$).
\end{convention}
\begin{remark}\label{rmk : Gamma_j not necessarily isomorphic to fundamental group of subsurface}
    Note that $\Gamma_j$ is only isomorphic to $\pi_1(S_j,p_j)$ if the subsurface $S_j$ is incompressible. However, the fact that $\psi_j$ is not necessarily injective when restricted to $\pi_1(S_j,p_j)$ will not be important in the rest of the article.
\end{remark}

We now give explicit descriptions for what the gluing loops from Definition \ref{def:graph of groups and fundamental group} will be. Let  $\{y_{i_1},\dots,y_{i_N}\}\subset T^c_u$ be a subset such that if $y\in \{y_{i_1},\dots,y_{i_N}\}$, then $\overline{y} \notin\{y_{i_1},\dots,y_{a_{i_N}}\}$. For $l = 1,\dots,N$, define
\begin{equation}\label{eq:definition gluing loops eta}
    \eta_l\coloneqq [\ctai{z}{i_l}*\ctai{\nu}{i_l}^{i_l}*(\coai{\nu}{i_l}^{i_l})^{-1}*\coai{z}{i_l}^{-1}]\in\pi_1(S,p),
\end{equation}
see for example Figure \ref{fig : graph for graph of groups example} (where $i_l = 2$).\\

With this setup, we can now describe the fundamental group of the surface $S$ using the vertex groups and the above gluing loops.

\begin{proposition}\label{prop: fundamental group of graph of groups form of gluing loops}
Consider the graph of groups $(\Gamma,Y)$ constructed in Section \ref{sec : Constructing graph of groups}, choose a maximal tree $T$, and consider the corresponding construction of the isomorphisms $\psi_j$ in \eqref{eq:definition of inclusion of vertex groups in graph of groups} mapping the vertex groups to subgroups of $\pi_1(S,p)$. The group generated by 
\[
	\Gamma_1 = \psi_1(\pi_1(S_1,p_1)),\Gamma_2 = \psi_2(\pi_1(S_2,p_2)),\dots,\Gamma_k = \psi_k(\pi_1(S_k,p_k)),\eta_1,\dots,\eta_N,
\]
subject to the graph of group relations in Definition \ref{def:graph of groups and fundamental group}, where the vertex groups are $\Gamma_j=\psi_j(\pi_1(S_j,p_j))$ for $j = 1,\dots,k$ and the gluing loops $\eta_1,\dots,\eta_N$ defined in \eqref{eq:definition gluing loops eta}, is the fundamental group $\pi_1(S,p)$.

\end{proposition}

\begin{proof}
By \cite[Section I.5]{Serre_Trees} and the corollary of Lemma 5.3 in \cite{JohnsonMillson}, the fundamental group $\pi_1(\Gamma,Y,T)$ corresponding to the graph of groups and tree we chose is isomorphic to $\pi_1(S,p)$. We thus need to compare the generators in the proposition to the generators from Definition \ref{def:graph of groups and fundamental group}, as well as the relations. All the vertex groups are in the statement of the proposition. By relations \eqref{relation:graph of groups relation 1 prime} and \eqref{relation:graph of groups relation 2}, we only need to consider the generators corresponding to edges in $T_u^c$. The generators $\eta_1,\dots,\eta_N$ correspond to the gluing loops in Definition \ref{def:graph of groups and fundamental group}. Thus, we only need to check that relation \eqref{relation:graph of groups relation 1} holds.\\

We require that for $l = 1\dots,N$,
\[
\eta_l\cdot \psi_{o(y_{i_l})}(1^{y_{i_l}})\cdot\eta_l^{-1} = \psi_{t(y_{i_l})}(1^{\overline{y}_{i_l}}).
\]
Starting from the left hand side, 
\begin{gather*}
    \eta_l\cdot 1^{y_{i_l}}\cdot\eta_l^{-1} = \eta_l\cdot\coai{\psi}{i_l}(\coai{\alpha}{i_l}^{i_l})\cdot \eta_l^{-1} = \\
    (\ctai{z}{i_l}*\ctai{\nu}{i_l}^{i_l}*(\coai{\nu}{i_l}^{i_l})^{-1}*\coai{z}{i_l}^{-1})*(\coai{z}{i_l}*(\coai{\nu}{i_l}^{i_l}*a_{i_l}*(\coai{\nu}{i_l}^{i_l})^{-1})*\coai{z}{i_l}^{-1})*\\
    *(\coai{z}{i_l}*\coai{\nu}{i_l}^{i_l}*(\ctai{\nu}{i_l}^{i_l})^{-1}*\ctai{z}{i_l}^{-1})=\\
    =\ctai{z}{i_l}*\ctai{\nu}{i_l}^{i_l}*a_{i_l}*(\ctai{\nu}{i_l}^{i_l})^{-1}*\ctai{z}{i_l}^{-1}.
\end{gather*}
For the right hand side, we have that
\begin{gather*}
    \psi_{t(y_{i_l})}(1^{\overline{y}_{i_l}}) = \ctai{\psi}{i_l}(\alpha_{o(\overline{y}_{i_l})}^{i_l})=\ctai{\psi}{i_l}(\ctai{\alpha}{i_l}^{i_l}) =\\
   =\ctai{\psi}{i_l}(\ctai{\nu}{i_l}^{i_l}*a_{i_l}*(\ctai{\nu}{i_l}^{i_l})^{-1})
    =\ctai{z}{i_l}*\ctai{\nu}{i_l}^{i_l}*a_{i_l}*(\ctai{\nu}{i_l}^{i_l})^{-1}*\ctai{z}{i_l}^{-1}.
\end{gather*}
This finishes the proof.
\end{proof}

\subsection{Graph of groups decomposition associated to a subsurface, and the Zariski tangent spaces to the corresponding leaves}\label{sec : "Relative character varieties"}
In this section we focus on the particular case when the surface $S$ is decomposed by cutting out a connected essential subsurface. We describe the corresponding graph of groups decomposition, the Zariski tangent spaces to the leaves in this case, as well as the dimension of the Zariski tangent spaces. The lemmata here will be used mainly in the proof of Theorem \ref{thm : hamiltonian vector field is infinitesimal subsurface deformation} where we show that the Hamiltonian vector field of smooth invariant functions is an infinitesimal subsurface deformation.\\

For this subsection, fix a connected essential subsurface $S_0\subseteq S$ which is not in the same homotopy class of $S$. Otherwise, the associated graph of groups decomposition has a single vertex with $\pi_1(S,p)$ as the vertex group. Let $S_0$ have genus $g_0\geq 0$ and $n_0> 0$ boundary components. Orient the boundary components of $S_0$ such that $S_0$ lies to the left of the boundary curves. Choose $p\in S_0$ and generators of the fundamental group of $S_0$ so that it has the presentation
\begin{equation}\label{eq : presentation of fundamental group of S'}
\pi_1(S_0,p) = \left\langle a_1,b_1,\dots, a_{g_0},b_{g_0},c_1,\dots,c_{n_0}\bigg|\prod_{i = 1}^{g_0}[a_i,b_i]\prod_{j = 1}^{n_0}c_j=1\right\rangle,
\end{equation}
where the curves $c_j$ correspond to the boundary curves of $S_0$.\\

Now we move on to the decomposition of the fundamental group of $S$ associated to $S_0$. Cutting out $S_0$ from $S$, we obtain a decomposition of $S$ into subsurfaces $S_0,S_1,\dots,S_m$ for some $n_0\geq m\geq 1$. Choose points $p_1,\dots,p_m$ in the interiors of $S_1,\dots,S_m$ respectively, as well as paths from the base points to the boundary curves. This allows us to construct the graph of groups $(\Gamma,Y)$ with vertices $\{v_0,v_1,\dots,v_m\}$ as in Section \ref{sec : construction for graph of groups}, where the vertex groups are $\pi_1(S_j,p_j)$ for $j =1,\dots,m$ and $\pi_1(S_0,p)$ and the edge groups are $\Gamma_{y_j}=\langle c_j\rangle\cong \Z$. In this case, the monomorphisms associated to the edges $\overline y_j$ are taken to be the identity. We will make some abuse of notation and write $c_j^{\overline y_j} = c_j$.\\

 After choosing a maximal tree $T$ in $Y$ and paths between the base points, we obtain $N\coloneqq n_0-m\geq 0$ gluing loops as defined in Equation \eqref{eq:definition gluing loops eta}, as well as homomorphisms $\psi_j\colon \pi_1(S_j,p_j)\to\pi_1(S,p)$. Using Proposition \ref{prop: fundamental group of graph of groups form of gluing loops}, and using Convention \ref{convention : Gammaj is image of psij}, the fundamental group $\pi_1(S,p)$ is generated by
\[
	\Gamma_0 = \psi_0(\pi_1(S_0,p)),\Gamma_1 = \psi_1(\pi_1(S_1,p_1)),\dots,\Gamma_m=\psi_m(\pi_1(S_m,p_m)),\eta_1,\dots,\eta_N;
\]
and we call these the \textit{generators associated to the graph of groups decomposition arising from $S_0$}. The generators also depend on the choice of maximal tree $T$ and the chosen paths between base points and boundary curves, but we supress this data from the definition. Using this notation, we will denote representations $\rho \in\Hom(\pi_1(S,p),G)$ by tuples $(\rho_\vv,\rho_\yy)$, where $\rho_\vv=\{\rho_i\colon\Gamma_i\to G \,|\, i = 0,\dots,m\}$ and $\rho_\yy = \{\rho_{\eta_i}\colon \langle\eta_i \rangle\to G\,|\,i = 1,\dots,N \}$. Similarly, we denote by $(u_\vv,u_\yy)\in Z^1(\pi_1(S,p),\gfr_{\Ad(\rho_\vv,\rho_\yy)})$ a pair where $u_\vv = \{u_i\in Z^1(\Gamma_i,\gfr_{\Ad\rho_i})\st i = 0,\dots,m\}$ and $u_\yy = \{u_{\eta_i}\in Z^1(\langle\eta_i\rangle,\gfr_{\rho_{\eta_i}})\st i=1,\dots,N\}$.\\

Associated to the subsurface $S_0$, we obtain a decomposition into leaves $\widehat\leaf(S_0) = \{\widehat\leaf_{v_0,\mathcal R}\st[\mathcal R]\in\mathbf H_{v_0}\}$ associated to $S_0$ as in \eqref{eq : def of decomposition into leaves associated to w}. For the rest of the section, we will abuse notation and write $\widehat\leaf_{S_0,\mathcal R}$ or $\widehat\leaf_{\mathcal R}$ for $\widehat\leaf_{v_0,\mathcal R}$. We obtain respectively the decompositions $\leaf(S_0)$ into leaves $\leaf_\mathcal R$ of the character variety.

\begin{remark}\label{rmk : notation on base points and deck transformations}
	In Section \ref{sec : subsurface deformations}, we will be working with different base points, as well as the group of deck transformations $\pi$ simultaneously. Let $r\in S$ be another base point. Choosing a path between $p$ and $r$ to obtain an isomorphism between fundamental groups. We will denote by $\Gamma_j^r$ and $\eta_1^r,\dots,\eta_N^r$ the images under this map of the generators associated to the graph of groups decomposition arising from $S_0$ and $T$. Moreover, after choosing an isomorphism between $\pi_1(S,p)$ and the group of deck transformation $\pi$, we will denote by $\Gamma_i^\pi$ and $\eta_1^\pi,\dots,\eta_N^\pi$ the generators. We will often make some abuse of notation when it is clear, and denote also by $\Gamma_i$ the image of the vertex groups, and by $\eta_1,\dots,\eta_N$ the image of the gluing loops. Moreover, the isomorphisms also induce corresponding projection maps $\widehat p_w$ and $p_w$ on the corresponding representation varieties for $\pi_1(S,r)$ and $\pi$ as well as their corresponding quotients by $G$. Hence, given a subsurface $S_0$, we also obtain decompositions into leaves $\widehat\leaf^r(S_0)$ with leaves $\widehat\leaf^r_{\mathcal R}$ (and similarly $\leaf^r(S_0)$ with leaves $\leaf^r_\mathcal R$), and respectively $\widehat\leaf^\pi(S_0)$ and $\leaf^\pi_{\mathcal R}$. In the case of the group of deck transformations $\pi$, we will often abuse notation and drop the superscript $\pi$ in the notation for the leaves and decompositions when it is clear from the context.
\end{remark}

In the following lemma, we describe the Zariski tangent space to the leaves in $\widehat \leaf(S_0)$. This will be one of the main ingredients in the proof of Theorem \ref{thm : hamiltonian vector field is infinitesimal subsurface deformation}, where we show that the Hamiltonian flow of an induced invariant function is an infinitesimal subsurface deformation (Definition \ref{def : inf and subsurface deformation}).\\

\begin{lemma}\label{lem: tangent space "relative representation variety" in larger variety}
	Let $\rho\in\Hom(\pi_1(S,p),G)$ and $(\varphi_\vv,\varphi_\yy)\in \widehat\leaf_{\rho}$. The Zariski tangent space is identified with
	\begin{gather}
		\T_{(\varphi_\vv,\varphi_\yy)}\widehat\leaf_{\rho}\cong\nonumber\\
		\big\{(u_\vv,u_\yy)\in Z^1(\pi_1(S,p),\gfr_{\Ad(\varphi_\vv,\varphi_\yy)})\st u_{i}\in B^1(\Gamma_{i},\gfr_{\Ad\varphi_{i}}) \textnormal{ for }i\neq 0\big\}.
	\end{gather}
	In particular, for any $(u_\vv,u_\yy)\in \T_{(\varphi_\vv,\varphi_\yy)}\widehat\leaf_{\rho}$, we have that $u_{0}$ is a \textit{parabolic cocycle} in the sense of  \cite[Section 3]{ParabolicCohomology_GHJW}, i.e. there exist $X_1,\dots,X_{n_0}\in\gfr$ such that
	\[
		u_{0}(c_j) = \Ad_{\varphi_{0}(c_j)}(X_j)-X_j
	\]
	for every $j = 1,\dots,n_0$.
\end{lemma}
\begin{proof}
	Consider a variation $(\varphi^s_\vv,\varphi^s_\yy)\in \widehat \leaf_{\rho}$ given as in \eqref{eq: variation of graph of groups representations with cocycles}. In order to remain in $\widehat\leaf_{\rho}$, we need variations of the form
	\[
		\varphi^s_{i}(\gamma_{i}) = (g^s_i)^{-1}\varphi_{i}(\gamma_{i})g_i^s
	\]
	for some paths in $G$ which up to first order are given by $g_i^s = \exp(s Y_i+\mathcal O(s^2))$ for some $Y_i\in \gfr$, for any $\gamma_{i}\in\Gamma_i$ and for $i = 1,\dots,m$. Differentiating the above expression, we see that 
	\begin{equation}\label{eq: outer vertex and edge groups look like coboundaries}
		u_{i}(\gamma_{i}) = \frac{d}{ds}\bigg|_{s=0}\varphi_i^s(\gamma_i)\varphi^0_i(\gamma_i)^{-1} =  \Ad_{\varphi_{i}(\gamma_{i})}(Y_i) - Y_i.
	\end{equation}
	which is precisely the coboundary condition, namely $u_{i}\in B^1(\Gamma_{i},\gfr_{\Ad\varphi_{i}})$.\\
	
	To get the second part of the lemma, note that the cocycles must also satisfy conditions \eqref{tangent space condition graph of groups conjugation part},\eqref{tangent space condition graph of groups orientation flip part},\eqref{tangent space condition graph of groups tree edges} in Lemma \ref{lemma : conditions for tangent space graph of groups representations}. We now have two cases depending on whether an edge is in the (unoriented) tree $T_u$ or not. Consider first an edge $y\in T_u^c$ (that is, it is not in the unoriented tree) such that\footnote{There is no loss of generality here since the graph of groups in the case of cutting along a single subsurface, there are no edges between subsurface $S_i$ and $S_j$ for $i\neq j\in\{1,\dots,m\}$, see for example Figure \ref{fig : graph of groups associated to S prime}.} $t(y) = v_0$ and let $a\in \Gamma_y$. Then using \eqref{eq: outer vertex and edge groups look like coboundaries}, condition \eqref{tangent space condition graph of groups conjugation part} can be rearranged to read
	\begin{gather*}
		u_{0}(a^y) = \Ad_{\varphi_y(g_y)^{-1}}\left(u_{o(y)}(a^{\overline y})-u_y(g_y)+\Ad_{\varphi_{o(y)}(a^{\overline y})}(u_y(g_y))\right)\\
		=\Ad_{\varphi_y(g_y)^{-1}}\left(\Ad_{\varphi_{o(y)}(a^{\overline y})}(Y_{o(y)})-Y_{o(y)}-u_y(g_y)+\Ad_{\varphi_{o(y)}(a^{\overline y})}(u_y(g_y))\right)\\
		= \Ad_{\varphi_y(g_y)^{-1}\varphi_{o(y)}(a^{\overline y})}(Y_{o(y)}) - \Ad_{\varphi_y(g_y)^{-1}}(Y_{o(y)}) -\Ad_{\varphi_y(g_y)^{-1}}(u_y(g_y))+\\
		+ \Ad_{\varphi_y(g_y)^{-1}\varphi_{o(y)}(a^{\overline y})}(u_y(g_y))\\
		= \Ad_{\varphi_{0}(a^y)}(\Ad_{\varphi_y(g_y)^{-1}}(Y_{o(y)}+u_y(g_y))-(\Ad_{\varphi_y(g_y)^{-1}}(Y_{o(y)}+u_y(g_y))),
	\end{gather*}
	where we used condition \eqref{condition: representation on graph of groups non tree edges} in the last line. Thus setting 
	\[
		X_j= \Ad_{\varphi_y(g_y)^{-1}}(Y_{o(y)}+u_y(g_y)),
	\]
	where $o(y) = v_j$ and noting that in this case $\Gamma_{y} = \langle c_j\rangle$, we have that $c_j^y = c_j$, settles the first case.\\
	
	The second case is when $y\in T_u$. We may assume that $t(y) = v_0$. Then condition \eqref{tangent space condition graph of groups tree edges} together with \eqref{eq: outer vertex and edge groups look like coboundaries} implies that 
	\begin{align*}
		u_{0}(a^y) &= u_{o(y)}(a^{\overline y}) = \Ad_{\varphi_{o(y)}(a^{\overline y})}(Y_{o(y)}) - Y_{o(y)}\\
		&=\Ad_{\varphi_{0}(a^y)}(Y_{o(y)}) - Y_{o(y)}.
	\end{align*}
	Setting 
	\[
		X_j =Y_{o(y)}
	\]
	where again $o(y) = v_j$, $\Gamma_y = \langle c_j\rangle$ and $c_j^y = c_j$, we get the result.
\end{proof}

Using this lemma, we can compute dimensions of Zariski tangent spaces. The formula is used in Lemma \ref{lem : leaves are smooth for hitchin} to show that the decomposition $\leaf(S_0)$ associated to a connected essential subsurface $S_0$ is a smooth foliation of the space of Hitchin representation $\Hit d$.
We will be primarily interested in the case that $G= \SL d$, which is a real algebraic group.  Thus $\Hom(\pi_1(S,p),G)$ has the structure of a real affine algebraic variety.

\begin{lemma}\label{lem: dimension of "relative representation variety"}
    Let $G$ be a real algebraic group.
	Let $\rho=(\rho_\vv,\rho_\yy)\in\Hom(\pi_1(S,p),G)$ and let $(\varphi_\vv,\varphi_\yy)\in \widehat\leaf_{\rho}$, with $\varphi_\vv=\{\varphi_i\colon\Gamma_i\to G \,|\, i = 0,\dots,m\}$ and $\varphi_\yy = \{\varphi_{\eta_i}\colon \langle\eta_i \rangle\to G\,|\,i = 1,\dots,N \}$. The dimension of the Zariski tangent space of $\widehat\leaf_{\rho}$ at $(\varphi_\vv,\varphi_\yy)$ is
	\begin{align*}
		n_0\dim G - \sum_{i = 1}^m\dim Z(\rho_i)
		 + (2g_0-1)\dim G + \sum_{j = 1}^{n_0}\dim C_j + \dim Z(\varphi_{0}),
	\end{align*}
	where $C_j$ is the conjugacy class of $\rho_{0}(c_j)$ in $G$. In particular, the dimension of the Zariski tangent space changes only in the $\dim Z(\varphi_{0})$ summmand and the dimension is minimized when $\dim Z(\varphi_{0})$ is minimized.
\end{lemma}

\begin{remark}\label{rmk : smooth point when dimension of centralizer is small }
    In particular, in the case when $\widehat \leaf_\rho$ is an algebraic variety, it is a smooth variety when the dimension of the Zariski tangent spaces $ \T_{(\varphi_\vv,\varphi_\yy)}\widehat\leaf_\rho$ is constant on $\widehat\leaf_\rho$.
\end{remark}

\begin{proof}
	By Lemma \ref{lem: tangent space "relative representation variety" in larger variety}, the $u_{i}$ component of a tangent cocycle $(u_{\vv},u_\yy)\in \T_{(\varphi_\vv,\varphi_\yy)}\widehat\leaf_{\rho}$ is a coboundary for $i\neq 0$. As in \cite{SymplecticNature_Goldman}, it is easily seen that the space of coboundaries $B^1(\Gamma_i,\gfr_{\Ad\varphi_{i}})$ is isomorphic to $\gfr/\mathfrak{Z}(\varphi_{i})$, where $\mathfrak{Z}(\varphi_{i})$ is the Lie algebra of the centralizer $Z(\varphi_{i})$. Moreover, the action by conjugation does not change the dimension of the centralizer, that is $\dim Z(\varphi_{i}) = \dim Z(\rho_{i})$ for each $i = 1,\dots,m$. Thus 
	\begin{equation}\label{eq: dimension of coboundaries at v_i}
		\dim B^1(\Gamma_i,\gfr_{\Ad\varphi_{i}}) = \dim G - \dim Z(\varphi_{i}) = \dim G - \dim Z(\rho_{i}).
	\end{equation}
	Next we deal with the dimension of $Z^1(\langle \eta_i\rangle,\gfr_{\Ad\varphi_{\eta_i}})$. By the cocycle condition and the fact that $\langle \eta_i\rangle$ is cyclic, a cocycle is fully determined by a single element of $\gfr$. Moreover, any such element gives rise to a cocycle. Hence
	\begin{equation}\label{eq: dim of cocycles for gy cyclic part}
		\dim Z^1(\langle \eta_i\rangle,\gfr_{\Ad\varphi_{\eta_i}}) = \dim G.
	\end{equation}
	By the second part of Lemma \ref{lem: tangent space "relative representation variety" in larger variety}, we have that $u_{0}$ is a parabolic cocycle. Proposition 2.4.9 in \cite{ArnaudThesis} states that the dimension of the space of such cocycles is equal to 
	\begin{equation}\label{eq: dimension of parabolic cocycles}
(2g_0-1)\dim G + \sum_{j = 1}^{n_0}\dim C_j+\dim Z(\varphi_{0}),
	\end{equation}
	where $C_j$ is the conjugacy class of $\varphi_{0}(c_j)$. By Remark \ref{rmk : alternative definition for leaves}, the conjugacy class of $\varphi_0(c_j)$ is equal to the conjugacy class of $\rho_0(c_j)$ for every $j = 1,\dots,n_0$.\\
	
	Now we turn to the linear relations \eqref{tangent space condition graph of groups conjugation part},\eqref{tangent space condition graph of groups orientation flip part},\eqref{tangent space condition graph of groups tree edges}. Equations \eqref{tangent space condition graph of groups orientation flip part} and \eqref{tangent space condition graph of groups tree edges} say that $u_y$ is determined fully by $u_{\overline y}$ for every $y\in E$. Thus, by \eqref{eq: dim of cocycles for gy cyclic part},  the dimension is cut down by $\dim G|T_u^c|/2$.\\
	
	As seen in the proof of Lemma \ref{lem: tangent space "relative representation variety" in larger variety}, Equations \eqref{tangent space condition graph of groups conjugation part} and \eqref{tangent space condition graph of groups tree edges} imply that $u_{0}$ is a parabolic cocycle. Combining \eqref{eq: dimension of coboundaries at v_i}, \eqref{eq: dim of cocycles for gy cyclic part}, \eqref{eq: dimension of parabolic cocycles} and the above paragraph, we see that the dimension of the Zariski tangent space is
	\begin{gather*}
		\sum_{i = 1}^m (\dim G - \dim Z(\rho_{i})) + |T_u^c|\dim G + \\
		(2g_0-1)\dim G + \sum_{j = 1}^{n_0}\dim C_j+\dim Z(\varphi_{0}) - \frac{|T_u^c|}{2}\dim G\\
		=\left(m+ |T_u^c| - \frac{|T_u^c|}{2}\right)\dim G - \sum_{i = 1}^m\dim Z(\rho_{i})\\
		+(2g_0-1)\dim G + \sum_{j = 1}^{n_0}\dim C_j+\dim Z(\varphi_{0})=\\
		=n_0\dim G - \sum_{i = 1}^m\dim Z(\rho_{i})
		 + (2g_0-1)\dim G + \sum_{j = 1}^{n_0}\dim C_j + \dim Z(\varphi_{0}),
			\end{gather*}
			where in the second to last line we use that $|T_u^c| = 2(n_0-m)$, which in turn follows from the fact that $|E|=2n_0$ is twice the number of boundary components and $|T_u|=2m$ is twice the number of subsurfaces.
\end{proof}

\section{Subsurface deformations}\label{sec : subsurface deformations}

We begin by introducing the main definition of this section.

\begin{definition}\label{def : inf and subsurface deformation}
    Let $S_0\subseteq S$ be a connected essential subsurface of $S$, and consider the corresponding decomposition into leaves $\leaf(S_0)$ of $\CharVar{\pi}$. 
    \begin{enumerate}
        \item \label{def : infinitesimal subsurface deformation}A vector field $\mathsf V$ on $\CharVar{\pi}$ is said to be an \textit{infinitesimal subsurface deformation along }$S_0$ if 
        \[
        \mathsf{V}([\rho])\in \T_{[\rho]}\leaf_\rho(S_0)
        \]
        for every $[\rho]\in\CharVar{\pi}$.
        \item\label{def : subsurface deformation} A $C^1$ path of conjugacy classes of representations $[\rho_\cdot]\colon I\subset \R\to\CharVar{\pi}$ is said to be a \textit{subsurface deformation along $S_0$}  if $[\rho_t]\in \leaf_{\rho_0}(S_0)$ for every $t\in I$. Equivalently, the projection $p_{S_0}([\rho_t])$ is constant in $t$.
    \end{enumerate}
    
\end{definition}

The following result describes the structure of Hamiltonian flows corresponding to induced invariant functions supported on a proper subsurface of $S$.\\

\begin{theorem}\label{thm : hamiltonian vector field is infinitesimal subsurface deformation}
	Let $f\colon G^k\to \R$ be an invariant function, $\ub\alpha\in\pi^k$, and $S_0$ a supporting subsurface for $\ub\alpha$. Then the Hamiltonian vector field $\Hm f_{\ub\alpha}$ is an infinitesimal subsurface deformation along $S_0$, i.e.
	\[
	\Hm f_{\ub\alpha}\in\T_{[\rho]}\leaf_\rho(S_0)
	\]
	for every $[\rho]\in \CharVar{\pi}$.
\end{theorem}
We prove this by explicitly computing the cocycle $\Hm f_{\ub\alpha}([\rho])$ in Proposition \ref{prop : Hamiltonian vector field for general subsurface} and show that the cocycle restricted to each vertex group $\Gamma_i$ (see Section \ref{sec : "Relative character varieties"}) is a coboundary.  Theorem \ref{thm : hamiltonian vector field is infinitesimal subsurface deformation} then follows by applying Lemma \ref{lem: tangent space "relative representation variety" in larger variety}. 
The computation is quite technical and is deferred to Section \ref{Appendix: proof of prop about infinitesimally subsurface deformation}.\\

The case when the supporting subsurface is a cylinder is much simpler since it reduces to the setting of Goldman's invariant functions in \cite{InvFct_Goldman}. From his results, we compute the Hamiltonian flow in Theorem \ref{thm : Hamiltonian flow when supporting subsurface is a cylinder}.\\

For the case when $S_0$ has negative Euler characteristic, if the decomposition into leaves is sufficiently smooth, the Hamiltonian flow of $f_{\ub\alpha}$ is a subsurface deformation.  This is recorded as Theorem \ref{thm : Hamiltonian flow is a subsurface deformation}, whose hypotheses are satisfied on the space of Hitchin representations $\Hit d$; see Section \ref{sec : Hamiltonian flow is a subsurface deformation for Hitchin representations}.\\ 

\begin{remark}
    As explained in Section \ref{sec : examples}, all regular functions (when $G$ is a complex Lie group), as well as coordinates for Hitchin representations can be realized as induced invariant functions.  This means that since induced invariant functions form a large class of functions on the character variety, it is in general hard to say anything geometric about the Hamiltonian flow of such a function.
\end{remark}

For the remainder of the section, we make the following choices. Fix a point $p\in S$, a lift $\tilde p\in\tilde S$ of $p$, and the corresponding isomorphism from $\pi_1(S,p)$ to $\pi$. Recall that we will only be working with the smooth points of the character variety $\CharVar{\pi}$ to work with the Goldman symplectic form.\\

\subsection{The Hamiltonian flow of \texorpdfstring{$f_{\protect\ub\alpha}$}{f alpha} when the supporting subsurface is a cylinder}\label{sec : Hamiltonian flow when supporting subsurface is a cylinder}
We first deal with the case when the supporting subsurface $S_0$ of a tuple $\ub\alpha = (\alpha_1,\dots,\alpha_k)\in\pi^k$ is a cylinder. Without loss of generality, we assume that the cylinder $S_0$ is incompressible. Otherwise, all the curves in $\ub\alpha$ are homotopically trivial in this case, and hence the induced function $f_{\ub\alpha}$ on the character variety is constant. In turn, this means that the Hamiltonian flow of $f_{\ub\alpha}$ is trivial.\\

The case when $S_0$ is an incompressible cylinder then reduces to Goldman's setting in \cite[Section 4]{InvFct_Goldman} as we now explain. Let $a\in\pi_1(S,p)$ be a generator for the fundamental group of $S_0$. Let $\alpha_i\colon G\to G$ be the word map arising from presenting the curve $\alpha_i$ in terms of the generator $a$ for $i = 1,\dots,k$. Letting
\begin{align}
    f'\colon G&\to\R\nonumber\\
    h&\mapsto f(\alpha_1(h),\dots,\alpha_k(h)),\label{eq : definition f' for when supporting subsurface is a cylinder}
\end{align}
we have that $f_{\ub\alpha} = f'_a$. Since the curve $a$ is simple, Goldman already described the full Hamiltonian flow in \cite[Section 4]{InvFct_Goldman}, which we will restate here. To do this, we have the case when $S_0$ is separating, and when it is non-separating. In the separating case, let $S_1$ be the subsurface of $S$ to the left of $a$, and $S_2$ the subsurface to the right. Picking an identification of $\pi_1(S,p)$ with $\pi$, and letting $\Gamma_1\cong\pi_1(S_1,p)$, $\Gamma_2\cong\pi_1(S_2,p)$ under this identification, the group of deck transformations $\pi$ is then an amalgamated product
\begin{equation}\label{eq : amalgamated product}
\pi = \Gamma_1*_{\langle a\rangle}\Gamma_2,
\end{equation}
where we abuse notation and also write $a$ for its image in $\pi$. In the case when $S_0$ is non$-$separating, the fundamental group can be written as an HNN extension. Let $\eta$ be a loop intersecting $a$ transversely with positive intersection number and let $S_1$ be the subsurface which is the complement of $S_0$. Again pick an identification of $\pi_1(S,p)$ with $\pi$ and let $\Gamma_1$ be the image of $\pi_1(S_1,p)$ and abusing notation also let $\eta$ be the image of $\eta$. Then
\begin{equation}\label{eq : HNN extensions expression}
\pi = \left\langle\Gamma_1,\eta\,|\, \eta a_+\eta^{-1}=a_-\right\rangle,
\end{equation}
where $a_-$ and $a_+$ are the boundary curves of $S_1$ which are freely homotopic to $a$ in $S$. We can now state the result, which follows from Theorem 4.5 and Theorem 4.7 in \cite{InvFct_Goldman}.
\begin{theorem}\label{thm : Hamiltonian flow when supporting subsurface is a cylinder}
    Let $f\colon G^k\to\R$ be an invariant function, $f'\colon G\to\R$ as in Equation \eqref{eq : definition f' for when supporting subsurface is a cylinder}, and let $\ub\alpha\in\pi^k$ such that the supporting subsurface $S_0$ is an incompressible cylinder.
    \begin{enumerate}
        \item If $S_0$ is separating, the Hamiltonian flow of $f_{\ub\alpha}$ is covered by the flow of representations $\Xi_\rho\colon\R\to\Reps\pi$ given by
        \[
       \Xi_\rho(t)(\gamma)\coloneqq \begin{cases}
            \rho(\gamma),&\gamma\in\Gamma_1\cong\pi_1(S_1,p)\\
            \exp\left(t F'(\rho(a))\right)\rho(\gamma)\exp\left(-t F'(\rho(a))\right), & \gamma\in\Gamma_2\cong\pi_1(S_2,p).
        \end{cases}
        \]
        \item If $S_0$ is non$-$separating, the Hamiltonian flow of $f_{\ub\alpha}$ is covered by the flow of representations $\Xi_\rho\colon\R\to\Reps\pi$ given by
        \[
        \Xi_\rho(t)(\gamma) = \begin{cases}
            \rho(\gamma),&\gamma\in\Gamma_1\cong\pi_1(S_1,p),\\
            \rho(\eta)\exp\left(t F'(\rho(a))\right), &\gamma = \eta.
        \end{cases}
        \]
    \end{enumerate}
    In particular, the Hamiltonian flow of $f_{\ub\alpha}$ is a subsurface deformation along $S_0$ and is defined for all time $t\in\R$.
\end{theorem}

\subsection{The Hamiltonian flow of \texorpdfstring{$f_{\protect\ub\alpha}$}{f alpha} when the supporting subsurface has negative Euler characteristic} 
We now deal with the case when the supporting subsurface $S_0$ has negative Euler characteristic. We say that a foliation of a smooth manifold is \textit{smooth} if every leaf of the foliation is a smooth submanifold. Moreover, given any connected essential subsurface $S_0$ and a subset $\mathscr O\subset \CharVar\pi$, we let $\leaf^{\mathscr O}(S_0)$ be the following decomposition into leaves:
\[
\leaf^{\mathscr O}(S_0) = \{\leaf^{\mathscr O}_{\mathcal R}\coloneqq \leaf_{\mathcal R}\cap\mathscr O\st [\mathcal R]\in\mathbf H_{S_0}\}.
\]

We also fix some notation for the Hamiltonian flow. For $[\rho]\in\CharVar{\pi}$, let $I([\rho])\subset \R$ be the interval of definition (including $0$) of the Hamiltonian flow of $f_{\ub\alpha}$ starting at $[\rho]$. The Hamiltonian flow is a curve $\Phi^{f_{\ub\alpha}}_{[\rho]}\colon I([\rho])\to\CharVar\pi$ satisfying
\[
	\frac{d}{dt}\bigg|_{t = s}\Phi^{f_{\ub\alpha}}_{[\rho]}(t) = \Hm f(\Phi^{f_{\ub\alpha}}_{[\rho]}(s))
\]
for every $s\in I([\rho])$.
\begin{theorem}\label{thm : Hamiltonian flow is a subsurface deformation}
	Let $f\colon G^k\to\R$ be an invariant function, let $\ub\alpha\in\pi^k$, and let $S_{0}$ be a supporting subsurface for $\ub\alpha$. Consider an open subset $\mathscr O\subset \CharVar{\pi}$ and assume that the decomposition into leaves $\leaf^\mathscr O(S_{0})$ is a smooth foliation of $\mathscr O$. Then the Hamiltonian flow of $f_{\ub\alpha}$ is a subsurface deformation along $S_{0}$. That is, for every $[\rho]\in\mathscr O$, the Hamiltonian flow $\Phi^{f_{\ub\alpha}}_{[\rho]}\colon I([\rho])\to\mathscr O$ is a subsurface deformation along $S_0$.
\end{theorem}

\begin{proof}
	Since the foliation $\leaf^{\mathscr O}(S_{0})$ is smooth, any vector field which is tangent to the foliation has flow lines which remain in the leaves (this follows for example from Proposition 19.19 in \cite{Lee_SmoothManifolds} combined with Frobenius' Theorem \cite[Theorem 19.12]{Lee_SmoothManifolds}). Using Theorem \ref{thm : hamiltonian vector field is infinitesimal subsurface deformation}, we have $\Hm f_{\ub\alpha}\in \T_{[\rho]}\leaf_\rho^{\mathscr O}(S_0)$ for every $[\rho]$. The theorem follows. 
\end{proof}

In the special case when the supporting surface $S_0$ is an $m-$holed sphere that is \textit{fully separating}, that is, $S\smallsetminus S_0$ has $m$ components, we have the following corollary.

\begin{corollary}\label{cor : hamiltonian flow for fully separating subsurface}
	Let $f\colon G^k\to\R$ be an invariant function, and suppose $\ub\alpha\in\pi^k$ admits a supporting subsurface $S_0$ which is a fully separating $m-$holed sphere. Let $\mathscr O\subset\CharVar\pi$ be an open subset such that $\leaf^{\mathscr O}(S_0)$ is a smooth foliation. Let $\Phi^{f_{\ub\alpha}}_{[\rho]}\colon I(\rho)\subseteq\R\to \mathscr O$ be the Hamiltonian flow of $f_{\ub\alpha}$ starting at $[\rho] = [\rho_\vv,\rho_\yy]$, where $\rho_\vv=\{\rho_i\colon\Gamma_i\to G \,|\, i = 0,\dots,m\}$. Then there exists a family of paths $g^t_1,\dots,g^t_m\in G$ with $g_i^0 = e$ such that for every $\gamma\in\Gamma_i=\pi_1(S_i,p)$, the flow $\Xi_\rho\colon I(\rho)\to\Reps\pi$ given by
	\[
		\Xi_{\rho}(t)(\gamma) = g_i^t\rho_i(\gamma)(g_i^t)^{-1},
	\]
	for $i = 1,\dots,m$ covers the Hamiltonian flow $\Phi^{f_{\ub\alpha}}_{[\rho]}$ restricted to $\Gamma_i$. In particular, calling the boundary curves $c_1,\dots,c_m$, with $c_i$ being a boundary of $S_i$, we have that
	\[
	\Xi_{\rho}(t)(c_i) = g_i^t\rho_0(c_i)(g_i^t)^{-1}.
	\]
\end{corollary}

\begin{proof}
	In the case when $S_0$ is fully separating, the corresponding graph of groups decomposition is made into a tree by choosing the outgoing edges from $S_0$ as in Figure \ref{fig : graph of groups for fully separating m holed sphere} where the generators associated to the graph of groups decomposition are given by 
	\[
	\Gamma_0=\pi_1(S_0,p),\Gamma_1=\pi_1(S_1,p),\dots,\Gamma_m=\pi_1(S_m,p).
	\]
	Recall here from Section \ref{sec : "Relative character varieties"} that $\Gamma_i = \psi_i(\pi_1(S_i,p_i))$. Since $S_0$ is a fully separating $m-$holed sphere, the surfaces $S_i$ are incompressible and hence $\Gamma_i= \pi_1(S_i,p)$.   For $[\rho]\in\mathscr O$, Theorem \ref{thm : Hamiltonian flow is a subsurface deformation} gives that the Hamiltonian flow $\Phi_{[\rho]}^{f_{\ub\alpha}}$ is completely contained in $\leaf^\mathscr O_\rho$. Thus by Remark \ref{rmk : alternative definition for leaves} and the fact that the fundamental group is generated by the vertex groups omitting $\Gamma_0$, the Hamiltonian flow is covered by a conjugation of the vertex groups not associated to $S_0$. This gives the first part of the corollary.\\
	
	For the second part of the corollary, note that using the graph of groups relation \eqref{relation:graph of groups relation 1} and the fact that all edges of the graph are in $T_u$, the unoriented tree, the image of the edge monomorphisms $\langle c_i\rangle\to\Gamma_i$ and $\identity \colon\langle c_i\rangle\to \Gamma_0$ are identified. The first part of the corollary then implies the second using this fact.
	\begin{figure}[ht]
		\centering
		\includesvg[scale=0.4]{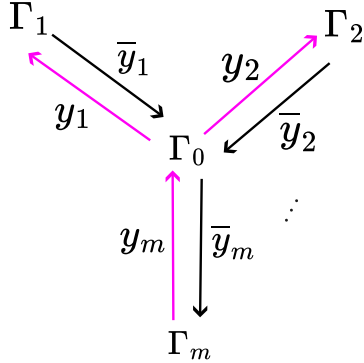}
		\caption{Graph of groups arising from a fully separating $m-$holed sphere.}
		\label{fig : graph of groups for fully separating m holed sphere}
	\end{figure}
\end{proof}

\subsection{Proof of Theorem \ref{thm : hamiltonian vector field is infinitesimal subsurface deformation}}\label{Appendix: proof of prop about infinitesimally subsurface deformation}

We first give a setup to have an explicit expression for the Hamiltonian vector field. Throughout, let $S_0$ be the supporting subsurface for the tuple $\ub\alpha\in\pi^k$ and assume that it is not a cylinder.\\

Assume that $S_0$ is not in the same homotopy class as $S$. Let the genus of $S_0$ be $g_0\geq 0$ and $n_0\geq 0$ the number of boundary components. Let $S = \bigcup_{i=0}^m S_i$ be the corresponding decomposition of the surface $S$, and $\Gamma_0,\Gamma_1,\dots,\Gamma_m<\pi$ the associated vertex groups (recall Convention \ref{convention : Gammaj is image of psij}). For each $i = 1,\dots,m$, let $K_i$ be the number of boundary components of the surface $S_i$, so that $K_1+\cdots+K_m = n_0$. Denote by $\zeta^i_1,\dots,\zeta^i_{K_i}$ the boundary curves of $S_i$ oriented so that $S_0$ lies to the left of each $\zeta^i_l$, where $i = 1,\dots,m$ and $l=1,\dots,K_i$. Let $p$ be a point in the interior of $S_0$ and choose a presentation of $\Gamma_0$ 

\begin{equation}\label{presentation of fundamental group of S' for subsurface deformations}
	\pi_1(S_0,p) = \left\langle a_1,b_1,\dots, a_{g_0},b_{g_0},c_1^1,\dots,c^{m}_{K_m}\bigg|\prod_{j = 1}^{g_0}[a_j,b_j]\prod_{i = 1}^m\prod_{l = 1}^{K_i}c_l^i=1\right\rangle,
\end{equation}
where under the isomorphism with $\pi_1(S,p)$, the curves $c_l^i$ are freely homotopic to $\zeta_l^i$ as shown for example in Figure \ref{fig : setup for fundamental group of S'} and Figure \ref{fig : zoom in Si}. The generators $a_j,b_j$ are shown in Figure \ref{fig : generators for genus}. We will abuse notation and also denote with the same letters the generators in $\pi$.\\

\begin{figure}[ht]
	\centering
	\includesvg[scale=0.8]{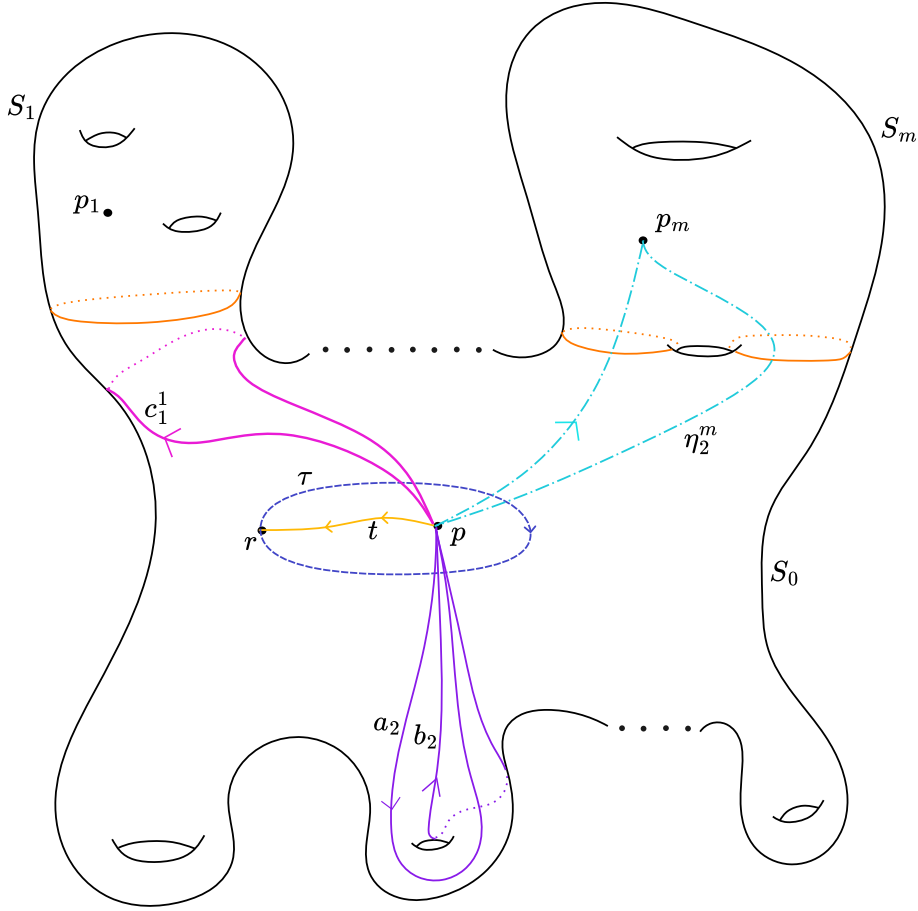}
	\caption{The solid pink and purple lines are some of the generators $c^i_l,a_j,b_j$ for the fundamental group $\Gamma_0$. The dashed blue loop $\tau$ will be used to adapt the curves in the proof of Proposition \ref{prop : Hamiltonian vector field for general subsurface}. The dotted/dashed light blue loop $\eta^m_2$ is an example of a gluing loop arising from the graph of groups decomposition.}
	\label{fig : setup for fundamental group of S'}
\end{figure}

\begin{figure}[ht]
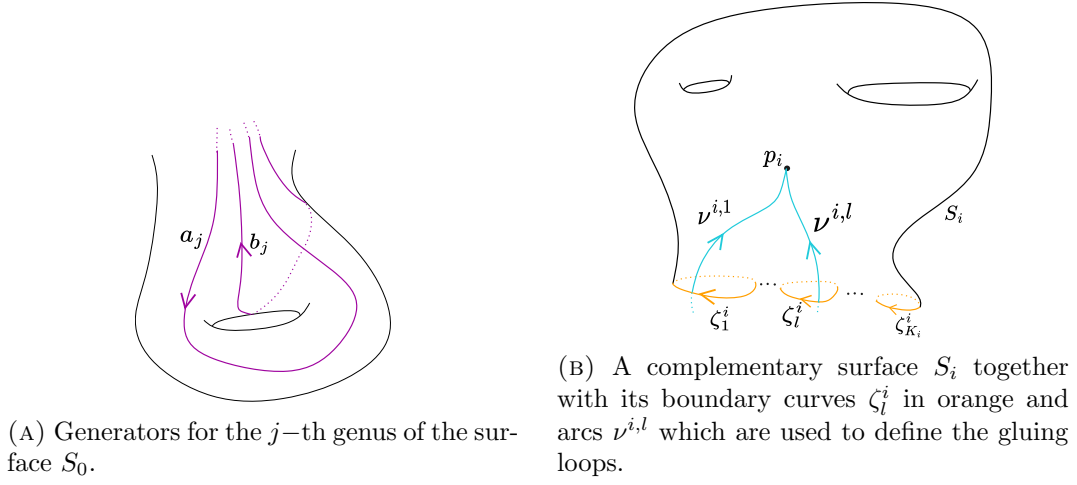

    \centering
    \begin{subfigure}[b]{0.45\textwidth}
    \centering
    \includesvg[width=0.5\textwidth]{genuspart.svg}
    \caption{Generators for the $j-$th genus of the surface $S_0$.}
    \label{fig : generators for genus}
    \end{subfigure}
    \quad
    \begin{subfigure}[b]{0.45\textwidth}
    \centering
    \includesvg[width=0.7\textwidth]{Silabels.svg}
    \caption{A complementary surface $S_i$ together with its boundary curves $\zeta^i_l$ in orange and arcs $\nu^{i,l}$ which are used to define the gluing loops.}
    \label{fig : zoom in Si}
    \end{subfigure}
    \caption{Detail on the setup for curves in the surface in Figure \ref{fig : setup for fundamental group of S'}.}
\end{figure}

Next we define paths between base points and boundary curves. For $i\in\{1,\dots,m\}$ and $l\in\{1,\dots,K_i\}$, let $\nu^{i,l}$ be an arc from $p$ to $p_i$ so that its intersection with the generators in \eqref{presentation of fundamental group of S' for subsurface deformations} is $p$, it intersects $\zeta^i_l$ exactly once, as shown in light blue in Figure \ref{fig : setup for fundamental group of S'} and in more detail in Figure \ref{fig : zoom in Si}. Arrange moreover for $\nu^{i,l}$ to be disjoint from all boundary curves except $\zeta^i_l$. The subarcs of these arcs define paths between base points and boundary curves. We denote by $\eta^i_l$ the deck transformation in $\pi$ corresponding to the curve $\nu^{i,1}*(\nu^{i,l})^{-1}$, see Figure \ref{fig : setup for fundamental group of S'}. \\ %which is mapped under the isomorphism with $\pi_1(S,p)$ to the curve $\nu^{i,1}*(\nu^{i,l})^{-1}$, see Figure \ref{fig : setup for fundamental group of S'}.\\

From the graph of groups associated to $S_0$, we obtain a graph of groups as in Figure \ref{fig : graph of groups associated to S prime}. We choose a maximal tree $T$ consisting of edges corresponding to $\zeta_1^i$ for $i = 1,\dots,m$, shown in pink in Figure \ref{fig : graph of groups associated to S prime}. We therefore get that, following Convention \ref{convention : Gammaj is image of psij}, the generators associated to the subsurface $S_0$ and the maximal tree $T$ are
\begin{gather*}
	\Gamma_0=\psi_0(\pi_1(S_{0},p)),\Gamma_1=\psi_1(\pi_1(S_1,p_1)),\dots,\Gamma_m = \psi_m(\pi_1(S_m,p_m)),\\
 \eta_2^1,\dots,\eta^1_{K_1},\dots,\eta^m_2,\dots,\eta_{K_m}^m.
\end{gather*}

\begin{figure}[ht]
	\centering
	\includesvg[scale=0.4]{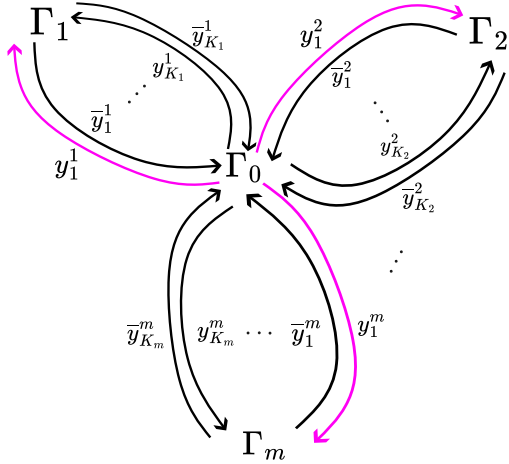}
	\caption{Graph of groups associated to the decomposition induced by $S_0$. The maximal tree we choose is in pink.}
	\label{fig : graph of groups associated to S prime}
\end{figure}

Now let $f\colon G^k\to \R$ be an invariant function. For each $i=1,\dots,k$, choose a presentation of $\alpha_i$ as a word in terms of the generators in \eqref{presentation of fundamental group of S' for subsurface deformations}. With such a representation, each $\alpha_i$ gives rise to a word map $\alpha_i\colon G^{2g_0+n_0}\to G$. As usual, define $f'\colon G^{2g_0+n_0}\to\R$  by $f'(h_1,\dots,h_{2g_0+n_0})\coloneqq f(\alpha_1(h_1,\dots,h_{2g_0+n_0}),\dots,\alpha_k(h_1,\dots,h_{2g_0+n_0}))$; this function is well adapted to our graph of groups decomposition. By construction, we have that
\begin{equation}\label{eq : equality between falpha and f' on generators}
	f_{\ub\alpha} = f'_{(a_1,b_1,\dots,a_{g_0},b_{g_0},c_1^1,\dots,c_{K_m}^m)}
\end{equation}
as functions on the character variety. 
We make use of the following notation regarding the variation functions. Namely, let $X_j = F'_{2j-1}$ and $Y_j = F'_{2j}$ for $j = 1,\dots, g_0$, which correspond to varying the $a_i$ and $b_i$ components respectively. Also let $Z^i_l = F'_{2g_0+l+K_1+\cdots K_{i-1}}$, which corresponds to varying the $c^i_l$ component. We will use $f'$ and its variation functions to compute the Hamiltonian vector field of $f_{\ub\alpha}$. Finally, given a representation $\rho\in\Reps\pi$, we write
\begin{align*}
	 X_j(\rho)&\coloneqq X_j(\rho(a_1),\dots,\rho(b_{g_0}),\rho(c^1_1),\dots\rho(c^{m}_{K_m})),\\
	 Y_j(\rho)&\coloneqq Y_j(\rho(a_1),\dots,\rho(b_{g_0}),\rho(c^1_1),\dots\rho(c^{m}_{K_m})),\\
	 Z^i_l(\rho)&\coloneqq Z^i_l(\rho(a_1),\dots,\rho(b_{g_0}),\rho(c^1_1),\dots\rho(c^{m}_{K_m})).
\end{align*}

\begin{proposition}\label{prop : Hamiltonian vector field for general subsurface}
Let $f\colon G^k\to\R$ be an invariant function, $\ub\alpha\in\pi^k$, and $S_0$ a supporting subsurface. Using the above setup and notation, let
\begin{gather*}
	\Gamma_0=\psi_0(\pi_1(S_{0},p)),\Gamma_1=\psi_1(\pi_1(S_1,p_1)),\dots,\Gamma_m = \psi_m(\pi_1(S_m,p_m)),\\
 \eta_2^1,\dots,\eta^1_{K_1},\dots,\eta^m_2,\dots,\eta_{K_m}^m.
\end{gather*}
be the generators of $\pi$ associated to the graph of groups decomposition arising from the supporting subsurface $S_0$ and the maximal tree $T$ in Figure \ref{fig : graph of groups associated to S prime}. Then the Hamiltonian vector field $\Hm f_{\ub\alpha}$ at an equivalence class $[\rho]\in\CharVar{\pi}$ is the cocycle $[u]\in H^1(\pi,\gfr_{\Ad\rho})$ with representative defined on the generators by
\begin{align}
	u(\gamma) &= \left(\identity - \Ad_{\rho(\gamma)}\right)\left( \left(\sum_{h = 1}^{i-1}\sum_{l = 1}^{K_h}\Ad_{\rho(c^h_l)^{-1}}(Z_l^h(\rho))-Z^h_l(\rho)\right)-Z^i_1(\rho)\right)\label{eq : Ham tangent cocycle for vertex groups}
	\end{align}
for $\gamma\in\Gamma_i$ with $i = 1,\dots,m$, and
\begin{align*}
    u(\eta^i_l) &= \left(\identity - \Ad_{\rho(\eta^i_l)}\right)\left(-\left(\sum_{h = 1}^{i-1}\sum_{n = 1}^{K_h}\Ad_{\rho(c^h_n)^{-1}}(Z^h_n(\rho))-Z^h_n(\rho)\right)\right)\\
    &-Z^i_1(\rho) + \Ad_{\rho(\eta^i_l)}\left(Z^i_l(\rho) -\left(\sum_{n = 1}^{l-1}\Ad_{\rho(c^h_n)^{-1}}(Z^i_n(\rho))-Z^i_n(\rho)\right) \right)
\end{align*}
where $i = 1,\dots,m$ and $l=1,\dots,K_i$. For the subgroup $\Gamma_0=\psi_0(\pi_1(S_0,p))$, the Hamiltonian vector field in terms of the standard generators in \eqref{presentation of fundamental group of S' for subsurface deformations} is given by
\begin{align*}
	u(a_j) &= \left(\sum_{h = 1}^{j-1}X_h(\rho)-\Ad_{\rho(a_h^{-1})}(X_h(\rho))+Y_h(\rho)-\Ad_{\rho(b_h^{-1})}(Y_h(\rho))\right)+Y_j(\rho)+\\
	&-\Ad_{\rho(a_j)}\left(\sum_{h = 1}^{j}X_h(\rho)-\Ad_{\rho(a_h^{-1})}(X_h(\rho))+Y_h(\rho)-\Ad_{\rho(b_h^{-1})}(Y_h(\rho))\right)\end{align*}
and
\begin{align*}
	u(b_j) &=\Ad_{\rho(b_j)}\left(\sum_{h=1}^{j-1}\Ad_{\rho(a_h^{-1})}(X_h(\rho))-X_h(\rho)+\Ad_{\rho(b_h^{-1})}(Y_h(\rho))-Y_h(\rho)\right)\\
	&+\Ad_{\rho(b_j)}(X_j(\rho))-\left(\sum_{h=1}^{j}\Ad_{\rho(a_h^{-1})}(X_h(\rho))-X_h(\rho)+\Ad_{\rho(b_h^{-1})}(Y_h(\rho))-Y_h(\rho)\right)
\end{align*}
for $j = 1,\dots,g_0$.

\end{proposition}

We can now deduce Theorem \ref{thm : hamiltonian vector field is infinitesimal subsurface deformation}.\\

\begin{proof}[of Theorem \ref{thm : hamiltonian vector field is infinitesimal subsurface deformation}]
	In Proposition \ref{prop : Hamiltonian vector field for general subsurface}, we see in \eqref{eq : Ham tangent cocycle for vertex groups} that the cocycle $\Hm f_{\ub\alpha}([\rho])$ is a coboundary when restricted to each vertex group $\Gamma_i$ for $i\neq 0$. Thus by Lemma \ref{lem: tangent space "relative representation variety" in larger variety},
	\[
	\Hm f_{\ub\alpha}([\rho])\in\T_{[\rho]}\leaf_\rho(S_0),
	\]
 completing the proof.
\end{proof}

The proof of Proposition \ref{prop : Hamiltonian vector field for general subsurface} is an application of Lemma \ref{lem : Poincare duality}. That is, it is a computation in the homology of the flat vector bundle associated to a representation $\rho$ as described in Section \ref{sec : cohomology with local coefficients}. Before providing a detailed proof of the proposition, we briefly explain how to make the computation of the intersection pairing described in Section \ref{sec : cohomology with local coefficients} more explicit.  The reader is advised to consult Figure \ref{fig : How to compute intersection pairing} for the explanation. \\

Consider the pairing of flat vector bundles $\langle\cdot,\cdot\rangle\colon\xi_\rho\times\xi_\rho^*\to \xi_0 = S\times \R$ induced by the pairing $\langle\cdot,\cdot\rangle\colon\gfr\times\gfr^*\to\R$. Note that this pairing is invariant under the action by $\pi$ induced by $\Ad$ and $\Ad^*$.
Let $X\in\gfr, Y\in\gfr^*$ and suppose that chains $\sigma\otimes X_{\widetilde{\sigma}(0)}\in C_1(S,\xi_\rho)$ and $\tau\otimes Y_{\widetilde \tau(0)}\in C_1(S,\xi_\rho^*)$ intersect transversely at double points. Let $q\in S$ be such a point of intersection.
Let $t$ be a path from $\sigma(0)$ to $\tau(0)$. Specify the lift $\tilde\tau(0)$ of $\tau(0)$ to be the endpoint in $\tilde S$ of the lift of $t$ starting at $\tilde\sigma(0)$. Let $\tau^q$ be the subarc of $\tau$ from $\tau(0)$ to $q$ and $\sigma^q$ be the subarc of $\sigma$ from $\sigma(0)$ to $q$.
Let $\eta_q\in \pi$ be the deck transformation corresponding to the loop
\[
t*\tau^q*(\sigma^q)^{-1}\in\pi_1(S,\sigma_i(0))
\]
under the isomorphism $\pi_1(S,\sigma(0))\cong \pi$ coming from the choice of lift $\tilde\sigma(0)$. Using Equation \eqref{eq : intersection in universal cover} we have that
\begin{equation}\label{eq : intersection pairing explicit}
\langle X_{\widetilde \sigma(0)}(q),Y_{\widetilde \tau(0)}(q)\rangle = \langle\Ad_{\rho(\eta_q)}(X),Y\rangle.
\end{equation}

\begin{figure}[h]
    \centering
    \includesvg[scale=0.6]{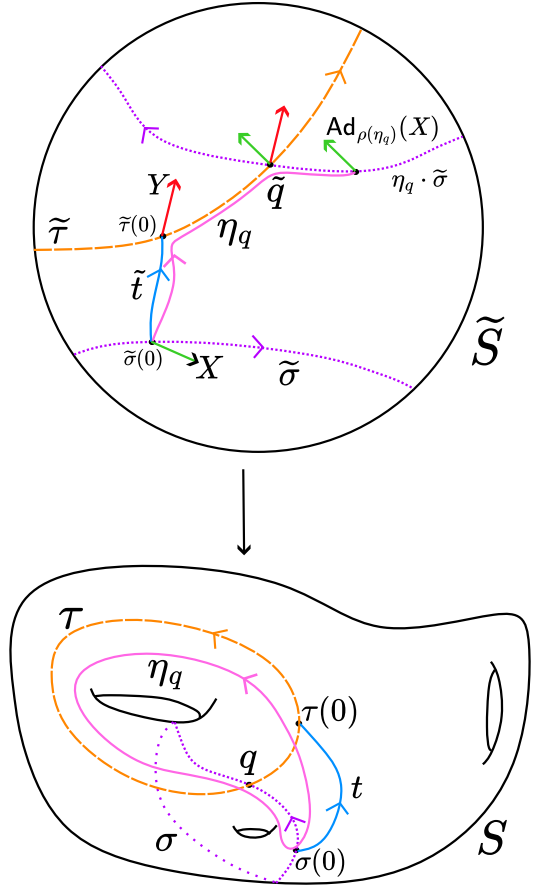}
    \caption{How to compute the intersection pairing between two flat sections. The element $Y$ is associated to the red arrow based at $\widetilde \tau(0)$, and the element $X$ is the green arrow based at $\widetilde \sigma(0)$. The curve $\tau$ is the orange dashed line, the curve $\sigma$ is dotted purple line, and the loop $\eta_q$ is the loop in pink.}
    \label{fig : How to compute intersection pairing}
\end{figure}

We will now prove Proposition \ref{prop : Hamiltonian vector field for general subsurface}.\\

\begin{proof}[of Proposition \ref{prop : Hamiltonian vector field for general subsurface}]
We compute the Hamiltonian vector field of the function $f'_{(a_1,\dots,b_{g_0},c_1^1,\dots,c^m_{K_m})}$ adapted to our choice of generators (see \eqref{eq : equality between falpha and f' on generators}). By Proposition \ref{prop: explicit cycle for Poincare dual of Hamiltonian vector field} and the isomorphism between $\pi$ and $\pi_1(S,p)$, we have that 
	\[
	M\coloneqq \Hm f_{\ub\alpha}([\rho])\frown [S] = \left[\sum_{i = 1}^m\sum_{l = 1}^{K_i}c_l^i\otimes Z_l^i(\rho) + \sum_{j = 1}^{g_0} a_j\otimes X_j(\rho) + b_j\otimes Y_j(\rho) \right],
	\]
where we are using Notation \ref{notation : flat sections} for the flat sections. To make the notation a little easier to deal with, we make the following abbreviations: $A_j\coloneqq \rho(a_j), B_j \coloneqq \rho(b_j), C_l^i\coloneqq \rho(c_l^i)$  for $j = 1,\dots,g_0$, $i = 1,\dots,m$ and $l = 1,\dots,K_i$. For $i = 1,\dots,m$ and $l = 2,\dots,K_i$, let $D^i_l \coloneqq \rho(\eta^i_l)$. %Since the above is a cycle, we have the equality
Applying the boundary operator, we see that
\begin{gather}   
    \sum_{j = 1}^{g_0} \Ad_{A_j^{-1}}(X_j(\rho)) - X_j(\rho) + \Ad_{B_j^{-1}}(Y_j(\rho)) - Y_j(\rho)=\nonumber\\
    =\sum_{i = 1}^m\sum_{l = 1}^{K_i}Z_l^i(\rho)-\Ad_{(C^i_l)^{-1}}(Z^i_l(\rho))\label{eq: cycle condition for M}
\end{gather}
We now use Lemma \ref{lem : Poincare duality}, and show that for any cycle $W\in Z_1(S,\xi_\rho^*)$, $\langle [u]\frown [W]\rangle_*=[M]\bullet_{\langle\cdot,\cdot\rangle}[W]$, where $u$ is the cocycle in the statement of Proposition \ref{prop : Hamiltonian vector field for general subsurface}. To use the intersection pairing, we need that $W$ be composed of curves based at a new base point $r\in S_0$ so that we have transverse intersections at double points. Let $t$ be a path from $p$ to $r$.  Let $\tilde r$ be the lift of $r$ corresponding to the endpoint of the lift of $t$ starting at $\tilde p$.\\

Consider the generators 
\[
	\Gamma_0^r,\Gamma_1^r,\dots,\Gamma_m^r,(\eta_2^1)^r,\dots,(\eta^1_{K_1})^r,\dots,(\eta^m_2)^r,\dots,(\eta_{K_m}^m)^r
\]
coming from the isomorphism between $\pi_1(S,p)$ and $\pi_1(S,r)$ induced by $t$ (see Remark \ref{rmk : notation on base points and deck transformations}). There is an isomorphism $H_1(S,\xi_\rho^*)\cong H_1(\pi_1(S,r),\gfr_{\Ad\rho^r})$, with corresponding representation $\rho^r \colon\pi_1(S,r)\to\pi_1(S,p)\to\pi\xrightarrow{\rho} G$. 
%arises from the isomorphisms between the fundamental groups and the group of deck transformations.
We may assume that
\begin{gather*}
	W = \sum_{i=1}^mW_i + \sum_{i = 1}^m\sum_{l = 2}^{K_i}(\eta^i_l)^r\otimes w((\eta^i_l)^r)+\sum_{j = 1}^{g_0}\left(a_j^r\otimes w(a_j^r)+b_j^r\otimes w(b_j^r)\right),
\end{gather*}
where
\[
	W_i=\sum_{s = 1}^{n_i}(\gamma_s^i)^r\otimes w((\gamma_s^i)^r)
\]
for some function $w\colon\pi_1(S,r)\to\gfr$, curves $(\gamma_s^i)^r\in\Gamma_i^r$ and some $n_i\in\N$. Let also $V_s^i\coloneqq \rho^r((\gamma_s^i)^r)$. We will abuse notation throughout this proof, and denote by the same symbols elements in $Z_1(S,\xi_\rho^*)$ and elements in $Z_1(\pi_1(S,p),\gfr_{\Ad\rho^p})\cong Z_1(\pi_1(S,r),\gfr_{\Ad\rho^r})\cong Z_1(\pi,\gfr_{\Ad\rho})$, and do similarly with the cocycles, where once again $\rho^p\colon\pi_1(S,p)\to\pi\xrightarrow{\rho}G$ comes from the isomorphism between $\pi_1(S,p)$ and $\pi$. Finally, we also fix a trivial loop $\tau$ based at $r$ which intersects all generators of $\Gamma_0$ transversely and has positive intersection number with $c_1^1$ as in shown as the dashed blue curve in Figure \ref{fig : Adapted curves}. We will use this to further adapt the curves appearing in $W$, making it easier to understand the intersection pattern with the generators of $\Gamma_0$. We can now compute the pairing of each term appearing in $W$ by using the loops appearing in Equation \eqref{eq : intersection pairing explicit}.\\

\begin{figure}[h]
    \centering
    \includesvg[scale=0.8]{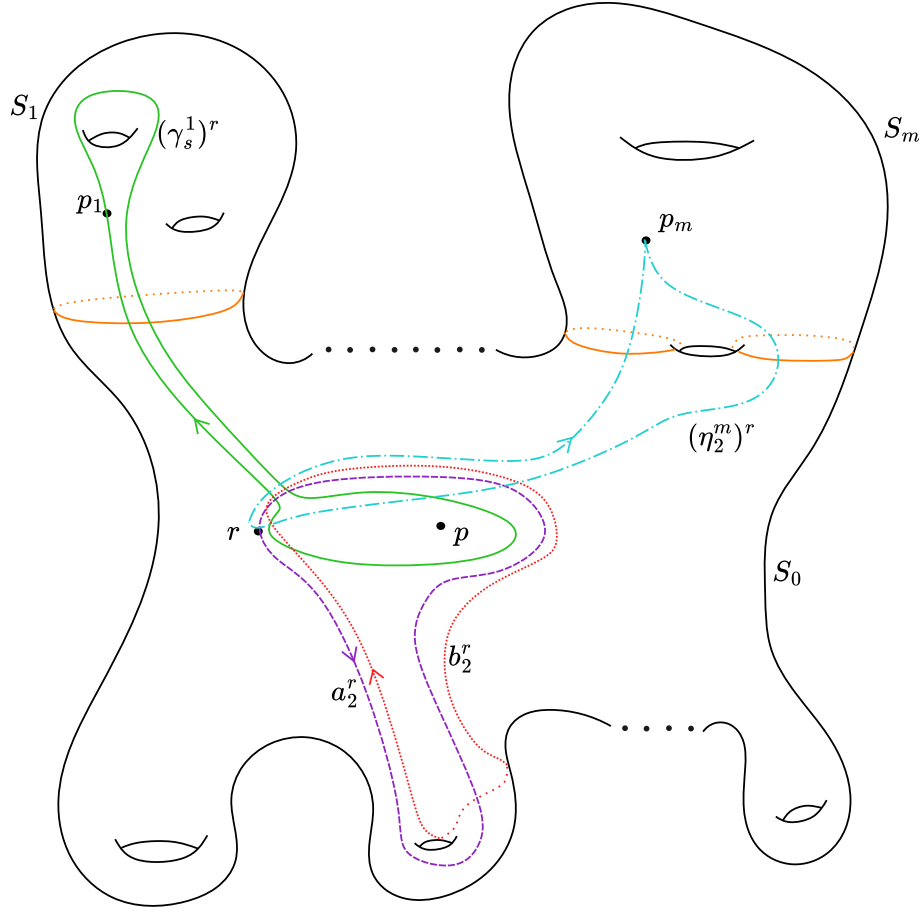}
    \caption{Here are the adapted curves for the computations in the proof of Proposition \ref{prop : Hamiltonian vector field for general subsurface}. In solid light green is the curve $(\gamma_s^1)^r$, in dashed purple is the curve $a^r_j$, in dotted red is the curve $b^r_j$, and in dotted/dashed light blue the curve $(\eta^m_2)^r$.}
    \label{fig : Adapted curves}
\end{figure}

\textbf{The pairing $M\bullet_{\langle\cdot,\cdot\rangle}W_i$.} 
Begin by further adapting the curves $(\gamma^i_s)^r$ as follows:
\begin{enumerate}
    \item Follow $\tau$ until intersecting the generator $c^i_1$.
    \item Then, by keeping the intersection with $c^i_1$ transverse, reach $p_i$ without intersecting any of the other generators.
    \item For the return, re$-$enter the surface $S_0$ and follow a path that does not intersect any of the generators until it reaches $\tau$.
    \item Finally, close the loop by following $\tau$ until reaching $r$.
\end{enumerate}
See for example the light green curve in Figure \ref{fig : Adapted curves}. Using $\tau$ allows us to use the cocycle condition in Equation \eqref{eq: cycle condition for M} to replace the terms coming from intersections with the generators $\{a_1,b_1,\dots,a_{g_0},b_{g_0}\}$ in the second equality below. We compute that
\begin{align*}
	M\bullet_{\langle\cdot,\cdot\rangle}W_i &= \sum_{s = 1}^{n_i}\Bigg\langle \left(\sum_{h=1}^{i-1}\sum_{l=1}^{K_h}\Ad_{(C^h_l)^{-1}}(Z^h_l(\rho))-Z^h_l(\rho)\right)-Z^i_1(\rho)+\Ad_{V_s^i(C_1^i)^{-1}}(Z^i_1(\rho))\\
	&+\Ad_{V_s^i}\Bigg(\left(\sum_{l=2}^{K_i}\Ad_{(C^h_l)^{-1}}(Z^h_l(\rho))-Z^h_l(\rho)\right)+\\
	&+\left(\sum_{h=i+1}^{m}\sum_{l=1}^{K_h}\Ad_{(C^h_l)^{-1}}(Z^h_l(\rho))-Z^h_l(\rho)\right)+\\&+\left(\sum_{j=1}^{g_0}\Ad_{A_j^{-1}}(X_j(\rho)) - X_j(\rho) + \Ad_{B_j^{-1}}(Y_j(\rho)) - Y_j(\rho)\right)\Bigg),w((\gamma_s^i)^r)\Bigg\rangle\\
	&=\sum_{s = 1}^{n_i}\Bigg\langle \left(\sum_{h=1}^{i-1}\sum_{l=1}^{K_h}\Ad_{(C^h_l)^{-1}}(Z^h_l(\rho))-Z^h_l(\rho)\right)-Z^i_1(\rho)\\
	&-\Ad_{V^i_s}\left(\left(\sum_{h=1}^{i-1}\sum_{l=1}^{K_h}\Ad_{(C^h_l)^{-1}}(Z^h_l(\rho))-Z^h_l(\rho)\right)-Z^i_1(\rho)\right),w((\gamma_s^i)^r)\Bigg\rangle\\
	&=\langle [u]\frown [W_i]\rangle_*,
\end{align*}
where the second equality follows from the cycle condition \eqref{eq: cycle condition for M}.

\textbf{The pairing $M\bullet_{\langle\cdot,\cdot\rangle}\left(a_j^r\otimes w(a_j^r)\right)$.} Once again, we further adapt $a_j^r$ as follows:
\begin{enumerate}
    \item Follow $\tau$ backwards until intersecting $b_{j-1}$.
    \item Loop around the genus and reach $\tau$ by intersecting only the generator $b_j$ of $\Gamma_0$.
    \item Finally, follow $\tau$ backwards until reaching $r$.
\end{enumerate}
See for example the dotted purple curve in Figure \ref{fig : Adapted curves}. We compute that
\begin{align*}
	M\bullet_{\langle\cdot,\cdot\rangle}\left(a_j^r\otimes w(a_j^r)\right)=\Bigg\langle\left(\sum_{h = 1}^{j-1}X_h(\rho)-\Ad_{A_h^{-1}}(X_h(\rho))+Y_h(\rho)-\Ad_{B_h^{-1}}(Y_h(\rho))\right)+\\
	+Y_j(\rho) + \Ad_{A_j}\Bigg(\left(\sum_{h = j+1}^{g_0}X_h(\rho)-\Ad_{A_h^{-1}}(X_h(\rho))+Y_h(\rho)-\Ad_{B_h^{-1}}(Y_h(\rho))\right)+\\
	+\sum_{i=1}^m\sum_{l=1}^{K_i}Z^i_l(\rho)-\Ad_{(C^i_l)^{-1}}(Z^i_l(\rho))\Bigg),w(a_j^r)\Bigg\rangle\\
	=\Bigg\langle\left(\sum_{h = 1}^{j-1}X_h(\rho)-\Ad_{A_h^{-1}}(X_h(\rho))+Y_h(\rho)-\Ad_{B_h^{-1}}(Y_h(\rho))\right)+Y_j(\rho)+\\
	-\Ad_{A_j}\left(\sum_{h = 1}^{j}X_h(\rho)-\Ad_{A_h^{-1}}(X_h(\rho))+Y_h(\rho)-\Ad_{B_h^{-1}}(Y_h(\rho))\right),w(a_j^r)\Bigg\rangle,\\
	=\langle[u]\frown[a_j\otimes w(a_j^r)]\rangle_*
\end{align*}
where similarly as above, we use the cycle condition \eqref{eq: cycle condition for M} to get rid of the terms coming from the intersections with $c_s^i$ in the second equality.

\textbf{The pairing $M\bullet_{\langle\cdot,\cdot\rangle}\left(b_j^r\otimes w(b_j^r)\right)$.} We adapt the curve $b_j^r$ as follows:
\begin{enumerate}
\item Follow $\tau$ until intersecting $a_j$.
\item Then loop through the genus and reach $\tau$ without intersecting any of the generators of $\Gamma_0$.
    \item Finally, follow $\tau$ until reaching $r$.
\end{enumerate}
See for example the squiggly orange curve in Figure \ref{fig : Adapted curves}. We compute that
\begin{align*}
	M\bullet_{\langle\cdot,\cdot\rangle}\left(b_j^r\otimes w(b_j^r)\right)=\Bigg\langle \left(\sum_{i = 1}^m\sum_{l=1}^{K_i}\Ad_{(C_l^i)^{-1}}(Z_l^i(\rho))-Z^i_l(\rho)\right)+\\
	+\left(\sum_{h=j+1}^{g_0}\Ad_{A_h^{-1}}(X_h(\rho))-X_h(\rho) + \Ad_{B_j^{-1}}(Y_h(\rho))-Y_h(\rho)\right)-\Ad_{B_j}(X_j(\rho))+\\
	+\Ad_{B_j}\left(\sum_{h=1}^{j-1}\Ad_{A_h^{-1}}(X_h(\rho))-X_h(\rho)+\Ad_{B_h^{-1}}(Y_h(\rho))-Y_h(\rho)\right),w(b_j^r)\Bigg\rangle\\
	=\Bigg\langle\Ad_{B_j}\left(\sum_{h=1}^{j-1}\Ad_{A_h^{-1}}(X_h(\rho))-X_h(\rho)+\Ad_{B_h^{-1}}(Y_h(\rho))-Y_h(\rho)\right)\\
	+\Ad_{B_j}(X_j(\rho))-\left(\sum_{h=1}^{j}\Ad_{A_h^{-1}}(X_h(\rho))-X_h(\rho)+\Ad_{B_h^{-1}}(Y_h(\rho))-Y_h(\rho)\right),w(b_j^r)\Bigg\rangle\\
	=\langle[u]\frown[b_j\otimes w(b_j^r)]\rangle_*
\end{align*}
where similarly as above, we use the cycle condition \eqref{eq: cycle condition for M} to get rid of the terms coming from the intersections with $c_s^i$ in the second equality.\\

\textbf{The pairing $M\bullet_{\langle\cdot,\cdot\rangle}\left((\eta^i_l)^r\otimes w((\eta_l^i)^r)\right)$.} In this case, we adapt the curve $\eta^i_l$ as follows:
\begin{enumerate}
    \item Follow $\tau$ until its intersection with $c^i_1$.
    \item Then follow $\nu^{i,1}*\nu^{i,l}$.
    \item Upon re$-$entering $S_0$, reach $\tau$ without intersecting any of the generators of $\Gamma_0$.
    \item Finally follow $\tau$ backwards until reaching $r$.
\end{enumerate}
See for example the thick dashed light blue curve in Figure \ref{fig : Adapted curves}. We compute that
\begin{align*}
	M\bullet_{\langle\cdot,\cdot\rangle}\left((\eta^i_l)^r\otimes w((\eta_l^i)^r)\right)=\Bigg\langle\left(\sum_{h = 1}^{i-1}\sum_{n = 1}^{K_h}\Ad_{(C^h_n)^{-1}}(Z^h_n(\rho))-Z^h_n(\rho)\right)\\
	-Z^i_1(\rho)+\Ad_{D^i_l}\Bigg(-\left(\sum_{h = 1}^{i-1}\sum_{n = 1}^{K_h}\Ad_{(C^h_n)^{-1}}(Z^h_n(\rho))-Z^h_n(\rho)\right)+Z^i_l(\rho)+\\
	-\left(\sum_{n = 1}^{l-1}\Ad_{(C_n^i)^{-1}}(Z^i_n(\rho))-Z^i_n(\rho)\Bigg)\right),w((\eta^i_l)^r)\Bigg\rangle\\
	=\langle[u]\frown [\eta^i_l\otimes w((\eta^i_l)^r)]\rangle_*
\end{align*}
which is the assignment $u(\eta^i_l)$.
\end{proof}

\subsection{The case of Hitchin representations}\label{sec : Hamiltonian flow is a subsurface deformation for Hitchin representations}
The hypothesis of Theorem \ref{thm : Hamiltonian flow is a subsurface deformation} are satisfied whenever we restrict to the Hitchin component for $\SL d,\Sp d,\mathsf{SO}_0(d,d+1)$, which is the content of the following lemma, where we denote by $\mathsf{Hit}(G,S)$ the Hitchin component associated to the group $G$ and the surface $S$.

\begin{lemma}\label{lem : leaves are smooth for hitchin}
Let $S_0\subseteq S$ be an essential connected subsurface inducing a decomposition $S = \bigcup_{i = 0}^mS_i$. Let $\leaf^{\mathsf{Hit}}(S_0)$ denote the decomposition into leaves of $\mathsf{Hit}(G,S)$, where $G$ is one of $\SL d,\Sp d,\mathsf{SO}_0(d,d+1)$. Then $\leaf^{ \mathsf{Hit}}(S_0)$ is a smooth foliation of the Hitchin component whose leaves all have the same dimension.
\end{lemma}
Since there are natural inclusions
\[
\mathsf{Hit}(\Sp d,S)\subset \Hit{2d},\quad\textnormal{and}\quad \mathsf{Hit}(\mathsf{SO}_0(d,d+1),S)\subset \Hit{2d+1},
\]
we will prove the lemma only in the case of $G = \SL d$. The main fact we need in the proof, is that restrictions of Hitchin representations to subsurfaces with negative Euler characteristic have the same center as the Lie group. For this reason, we do not treat the case of Hitchin representations in $\mathsf{SO}_0(d,d)$ here, where there is no natural inclusion of the Hitchin component of $\mathsf{SO}_0(d,d)$ into the Hitchin component of $\SL {2d}$.\\

For the proof, denote by $\widehat{\mathsf {Hit}}(d,S)$ the space of Hitchin representations in $\SL d$, and by $\Hit d$ the quotient under conjugation. Even though when the supporting subsurface is a cylinder Theorem \ref{thm : Hamiltonian flow when supporting subsurface is a cylinder} already gives us that the Hamiltonian flow is a subsurface deformation and we do not need any additional smoothness conditions, we still state the result for a general essential subsurface.\\

\begin{proof}
	We begin by showing that the leaves $\widehat\leaf^{\mathsf{Hit}}_{\rho}$ of $\widehat\leaf^{\mathsf{Hit}}(S_0)$ for any $\rho=(\rho_\vv,\rho_\yy)\in\widehat{\mathsf {Hit}}(d,S)$ are smooth. We do this by showing that the leaf is algebraic and then that the dimension of the Zariski tangent space is constant to use Remark \ref{rmk : smooth point when dimension of centralizer is small }. As in Section \ref{sec : "Relative character varieties"}, let $\rho_\vv=\{\rho_i\colon\Gamma_i\to G \,|\, i = 0,\dots,m\}$.  By \cite{Labourie_AnosovReps}, the holonomy of any nontrivial element of $\pi$ is hyperbolic. That is, it is conjugate to a diagonal matrix with all distinct eigenvalues.  By Theorem III.9.2 in \cite{BorelLinearAlgGroups}, the conjugacy classes $C_{i}$ of the boundary curves of $S_{0}$ are closed, that is, they are algebraic sets. From Lemma \ref{lem: "relative representation varieties" are analytic algebraic or semialgebraic}\eqref{lemma : if conjugacy classes are algebraic then algebraic}, we obtain that $\widehat\leaf^{\mathsf{Hit}}_\rho$ is an algebraic variety.\\
	
	 We now need to show that every point in the algebraic variety $\widehat\leaf^{\mathsf{Hit}}_\rho$ is smooth. As in Remark \ref{rmk : smooth point when dimension of centralizer is small }, we will show that the dimension of the Zariski tangent spaces $\T_{\varphi} \widehat\leaf^{\mathsf{Hit}}_{\rho}$ is constant. The proof will be split into two cases; first we deal with the case when $S_0$ is a cylinder, and then we deal with all the other cases.\\
	 
	When $S_0$ is a cylinder, $\Gamma_0$ is a cyclic group. Let $\varphi=(\varphi_\vv,\varphi_\yy)\in \widehat\leaf^{\mathsf{Hit}}_{\rho}$. Since $\Gamma_0$ is cyclic, $Z(\varphi_0) = Z(\varphi_0(\gamma))$, where $\gamma\in\Gamma_0$ is a generator. The element $\varphi_0(\gamma)$ is hyperbolic since $\varphi_0$ is Hitchin, and therefore $\dim Z(\varphi_0(\gamma))=\dim \afr$, where $\afr$ is a Cartan subalgebra in $\gfr$. This means that the dimension of $\T_{\varphi} \widehat\leaf^{\mathsf{Hit}}_{\rho}$ is constant as $\rho$ and $\varphi$ vary.\\
	
	Now assume that $S_0$ has negative Euler characteristic. Let $\varphi=(\varphi_\vv,\varphi_\yy)\in \widehat\leaf^{\mathsf{Hit}}_{\rho}$. Hitchin representations for surfaces with boundary and negative Euler characteristic are introduced in \cite{LabourieMcShane}, and by Theorem 1.2 in the same reference, we have that $\varphi_{0}\colon\Gamma_0=\psi_0(\pi_1(S_0,p))\to G$ is a Hitchin representation. Then by Theorem 1.4 in \cite{CanaryZhangZimmer_cuspedHitchin} and the discussion preceding it, $\varphi_{0}$ is irreducible. Proposition 15 in \cite{Sikora_CharacterVarieties} implies that the centralizer $Z(\varphi_{0})$ is a finite extension of the center $Z(G)$ of $G$. This in turn means that $\dim Z(\varphi_{0}) = \dim Z(G)$. We now turn to $Z(\rho_i)$ for $i\neq 0$. One can see that since we are cutting out $S_0$ from the closed surface $S$ and $S_0$ is essential, none of the remaining subsurfaces $S_1,\dots,S_m$ are cylinders. Hence by the above discussion,
	\[
		\dim Z(\rho_{1}) = \cdots = \dim Z(\rho_{m}) = \dim Z(\varphi_0) = \dim Z(G).
	\]
	The only terms missing in the dimension count of Lemma \ref{lem: dimension of "relative representation variety"} that depend on $\rho$ are the dimensions of the conjugacy classes $C_j$ of the boundary curves. Recall that the sets $C_j$ are conjugacy classes of hyperbolic elements. The centralizer of a hyperbolic element is isomorphic to a Cartan subalgebra $\mathfrak a$ in $\gfr$. Thus, the conjugacy class of any loxodromic element is locally isomorphic to $\gfr/\mathfrak a$. In particular, all conjugacy classes $C_j$ from Lemma \ref{lem: dimension of "relative representation variety"} have the same dimension, regardless of $\rho$ and $\varphi$. We have now shown that the dimension of the Zariski tangent spaces $\T_{\varphi} \widehat\leaf^{\mathsf{Hit}}_{\rho}$ all have the same dimension, which by Remark \ref{rmk : smooth point when dimension of centralizer is small } implies that each $\widehat\leaf^{\mathsf{Hit}}_\rho$ is smooth.\\
	
	We finish by noting that he action of $\mathsf{Inn}(G)\cong\PSL d$ by conjugation is free and proper on the Hitchin component\footnote{This is a well known fact and follows from the following. Labourie showed in \cite[Lemma 10.1]{Labourie_AnosovReps} that Hitchin representations are irreducible, which by \cite[Proposition 1.1]{JohnsonMillson} implies that the action by conjugation is proper. Freeness follows from the fact that $Z(\rho) = Z(G)$ for any Hitchin representation (see for example the proof of Theorem 1.2 in \cite{MostHitchinDense_LongReidWolff}).}. Thus $\leaf^{\mathsf{Hit}}_\rho$ is a smooth manifold and the foliation of $\Hit d$ by the leaves $\leaf_\rho^{\mathsf{Hit}}$ is smooth. 
\end{proof}

\begin{corollary}\label{cor : subsurface deformations in Hitchin}
	Let $G$ be one of $\SL d,\Sp d,\mathsf{SO}_0(d,d+1)$. For an invariant function $f\colon G^k\to\R$, let $\ub\alpha\in\pi^k$, and let $S_{0}$ be a supporting subsurface.  Then the Hamiltonian flow of $f_{\ub\alpha}$ restricted to the Hitchin component $\mathsf{Hit}(G,S)$, is a subsurface deformation along the supporting subsurface $S_{0}$.
\end{corollary}

\begin{proof}
	Since the Hitchin component is open (and closed), by Lemma \ref{lem : leaves are smooth for hitchin}, the hypothesis of Theorem \ref{thm : Hamiltonian flow is a subsurface deformation} are satisfied, from which the corollary follows.
\end{proof}

In \cite{InvFct_Goldman}, Goldman described the Hamiltonian flow of an invariant function induced by a simple closed curve as a generalized twist flow. The setup in this article allows us to say the following about Hamiltonian flows of non-simple closed curves.

\begin{corollary}\label{cor : twist flow for non simple curve is subsurface deformation}
    Let $\alpha$ be any closed curve in $S$ and let $f\colon G\to\R$ be an invariant function. Then the Hamiltonian flow of $f_\alpha$ restricted to the Hitchin component $\mathsf{Hit}(G,S)$ where $G$ is one of $\SL d,\Sp d,\mathsf{SO}_0(d,d+1)$, is a subsurface deformation along a supporting subsurface of $\alpha$.
\end{corollary}

\section{Hamiltonian flows associated to self-intersecting curves}\label{sec : self intersecting curves}
In this section we draw some consequences for the Hamiltonian flows of functions $f_\gamma$ on the character variety for the case when $\gamma$ is a self intersecting curve. We first find the Hamiltonian flow when $\gamma$ is contained in a cylinder, i.e. $\gamma$ is a power of a simple closed curve, similarly to the discussion in Section \ref{sec : Hamiltonian flow when supporting subsurface is a cylinder}. Finding the Hamiltonian flow for a general self--intersecting curve is more complicated, and we find some insights by focusing on one of the simplest cases, namely when $\gamma$ is the figure 8 curve in a pair of pants, the group $G$ is $\SL d$ and $f$ is the trace function.

\subsection{Hamiltonian flow of a power of a simple closed curve}
Let $\alpha$ be a simple closed curve. Then the for any $n\in\Z_{>0}$, the curve $\alpha^n$ has self intersection number $n-1$. The supporting subsurface $S_0$ of $\alpha^n$ is a cylinder, and we can therefore use the discussion in Section \ref{sec : Hamiltonian flow when supporting subsurface is a cylinder} and Goldman's results in \cite[Section 4]{InvFct_Goldman} to describe the Hamiltonian flow for $f_{\alpha^n}$ for any invariant function $f\colon G\to\R$. In the following, we keep the notation in \eqref{eq : amalgamated product}, \eqref{eq : HNN extensions expression} for the decomposition of $\pi$ as either an amalgamated product over $\alpha$ of $\Gamma_1\cong\pi_1(S_1,p)$ and $\Gamma_2\cong\pi_1(S_2,p)$ when $\alpha$ is separating, or an HNN extensions with $\Gamma_1\cong\pi_1(S_1,p)$ and a gluing loop $\eta$ when $\alpha$ is non-separating.

\begin{theorem}\label{thm : Hamiltonian flow of power of simple closed curve}
    Let $f\colon G\to\R$ be an invariant function and $\alpha$ a simple closed curve.
    \begin{enumerate}
        \item If $\alpha$ is separating, then the Hamiltonian flow of $f_{\alpha^n}$ is covered by the flow of representations $\Xi_\rho\colon\R\to\Reps\pi$ given by
        \[
        \Xi_\rho(t)(\gamma)\coloneqq \begin{cases}
            \rho(\gamma),&\gamma\in\Gamma_1\cong\pi_1(S_1,p),\\
            \exp\left(tnF(\rho(\alpha^n))\right)\rho(\gamma)\exp\left(-tnF(\rho(\alpha^n))\right), &\gamma\in\Gamma_2\cong\pi_1(S_2,p),
        \end{cases}
        \]
        \item If $\alpha$ is non-separating, the Hamiltonian flow of $f_{\alpha^n}$ is covered by the flow of representations $\Xi_\rho\colon\R\to\Reps\pi$ given by
        \[
        \Xi_\rho(t)(\gamma) = \begin{cases}
            \rho(\gamma),&\gamma\in\Gamma_1\cong\pi_1(S_1,p),\\
            \rho(\eta)\exp\left(t nF(\rho(\alpha^n))\right), &\gamma = \eta.
            \end{cases}
        \]
    \end{enumerate}
    In particular, the flows are defined for all time $t\in\R$.
\end{theorem}
\begin{proof}
    Instead of working with $f$, let $f'\colon G^n\to\R$ be given by $f'(h_1,\dots,h_n) = f(h_1\cdots h_n)$\footnote{In this proof we do not use the function $f'$ we used in Section \ref{sec : Hamiltonian flow when supporting subsurface is a cylinder} since there is an easier way to compute the Hamiltonian vector field in this setting.}. Then $f_{\alpha^n} = f'_{(\alpha,\dots,\alpha)}$. By Lemma \ref{lemma: variation function of reducible function}, the variation functions $F_i'\colon G^n\to\gfr$ are given by $F_i'(h_1,\dots,h_n) = F(h_1\cdots h_n)$ for all $i = 1,\dots,n$. By Proposition \ref{prop: explicit cycle for Poincare dual of Hamiltonian vector field} the Poincar\'e dual to the Hamiltonian vector field $\Hm f'_{(\alpha,\dots,\alpha)}([\rho])$ has representative $[\alpha\otimes n F(\rho(\alpha^n))]$. Then the result follows by \cite[Theorem 4.5 and Theorem 4.7]{InvFct_Goldman}.
\end{proof}

\subsection{Hamiltonian flow of the trace of the figure 8 curve in \texorpdfstring{$\SL{d}$}{SLd} }
Let $S_0$ be a fully separating pair of pants in the surface $S$. Let $a,b,c$ be the boundary curves of $S_0$ oriented so that $S_0$ lies to the left of each curve and such that (after picking representatives) $abc = 1\in\pi_1(S,p)$ as in Figure \ref{fig : Fully Separating pair of pants}. In this section, we will confuse $a,b,c$ with elements in $\pi_1(S,p)$ and with the (disjoint) boundary curves themselves, which will be clear from the context. We consider the group $G = \SL{d}$ and from now on focus on the Hitchin component $\Hit d$. We will endow the space of smooth points of the character variety with the Goldman symplectic form coming from the orthogonal form $\B\colon\gfr\times\gfr\to\R$ given by the trace $\B(X,Y)=\trace(XY)$. In the case when $d = 2$, this symplectic structure restricted to the Teichm\"uller component is a constant multiple of the Weil-Petersson symplectic form \cite{SymplecticNature_Goldman}. Let 
\begin{align*}
    f\colon G&\to\R\\
    g&\mapsto \trace(g). 
\end{align*}
Let $ \figeight= ab^{-1}$, the figure $8$ curve on the pair of pants (see Figure \ref{fig : Fully Separating pair of pants}), which is a self intersecting curve and has $S_0$ as a supporting subsurface. The goal of this section is to make remarks about the Hamiltonian flow of $f_\figeight$. This is in the setting Goldman considers in \cite{InvFct_Goldman}. However, Goldman only describes the Hamiltonian flow for functions induced by simple closed curves. \\

We use our setup of invariant functions and define
\begin{align*}
f'\colon G^2&\to\R\\
(g,h)&\mapsto \trace(gh^{-1}).
\end{align*}
With this function, we have that $f_\figeight = f'_{(a,b)}$. The function $f'$ satisfies the hypothesis of Lemma \ref{lemma: variation function of reducible function}, and hence we have that the variation functions $F_i'$ are given by
\begin{align*}
F_1'(g,h) = F(gh^{-1}) = gh^{-1} - \frac{\trace(gh^{-1})}{d}\identity\\
F_2'(g,h)=-F(gh^{-1}) = -gh^{-1}+\frac{\trace(gh^{-1})}{d}\identity,
\end{align*}
where the variation function $F\colon G\to\gfr$ is given by $F(g) = g-\frac{\trace(g)}{d}\identity$ by Corollary 1.8 in \cite{InvFct_Goldman}. Now let $S_x$ be the subsurface lying to the right of the curve $x$ for $x\in{a,b,c}$ and with that $\Gamma_x$ the vertex group associated to $S_x$. By Proposition \ref{prop : Hamiltonian vector field for general subsurface}, letting
\[
F_1’(\rho) = F_1'(\rho(a),\rho(b)) = F(\rho(ab^{-1})),\textrm{ and } F_2’(\rho) = F_2'(\rho(a),\rho(b)) = -F(\rho(ab^{-1})),
\]
we have that the Hamiltonian vector field $\Hm f_{\figeight}([\rho])$ of $f_\figeight$ is has representative $u$ given by
\begin{align}
u(\gamma) &= \begin{cases}
\left(\identity - \Ad_{\rho(\gamma)}\right)(-F_1(\rho))&\gamma\in \Gamma_a,
    \\
    \left(\identity - \Ad_{\rho(\gamma)}\right)\left(\Ad_{\rho(a^{-1})}(F_1’(\rho)) - F_1’(\rho) - F_2’(\rho)\right),&\gamma\in \Gamma_b,\\
    \left(\identity - \Ad_{\rho(\gamma)}\right)\left(\Ad_{\rho(a^{-1})}(F_1’(\rho))-F_1’(\rho) + \Ad_{\rho(b^{-1})}(F_2’(\rho))-F_2’(\rho)\right),&\gamma\in \Gamma_c
\end{cases}\nonumber\\
&=\begin{cases}
    \left(\identity - \Ad_{\rho(\gamma)}\right)(-F_1(\rho)),&\gamma\in \Gamma_a,\\
    \left(\identity - \Ad_{\rho(\gamma)}\right)\left(-\Ad_{\rho(b^{-1})}(F_2’(\rho))\right),&\gamma\in \Gamma_b,\\
    0,&\gamma\in \Gamma_c,
\end{cases}\nonumber\\
&=\begin{cases}
    \left(\identity - \Ad_{\rho(\gamma)}\right)\left(-F(\rho(ab^{-1}))\right), &\gamma\in\Gamma_a\\
    \left(\identity - \Ad_{\rho(\gamma)}\right)\left(F(\rho(b^{-1}a))\right),&\gamma\in\Gamma_b\\
    0,&\gamma\in\Gamma_c,
\end{cases}\label{eq : cocycle fully separating pair of pants figure 8 curve}
\end{align}

where the first equality for $u$ follows from the fact that by Proposition \ref{prop: explicit cycle for Poincare dual of Hamiltonian vector field}
\[
a\otimes F_1’(\rho) + b\otimes F_2’(\rho)
\]
is a cycle, i.e. 
\[
\Ad_{\rho(a^{-1})}(F_1’(\rho))-F_1’(\rho) + \Ad_{\rho(b^{-1})}(F_2’(\rho)) - F_2’(\rho) = 0.
\]
The second equality follows from the equivariance of $F$.\\

The supporting subsurface of $(a,b)$ is the pair of pants $S_0$ in Figure \ref{fig : Fully Separating pair of pants} and hence by Corollary \ref{cor : subsurface deformations in Hitchin}, the Hamiltonian flow of $f_{\figeight}$ is a subsurface deformation along $S_0$. Moreover, by Corollary \ref{cor : hamiltonian flow for fully separating subsurface}, the Hamiltonian flow is covered by conjugations of $\Gamma_x$ for $x\in\{a,b,c\}$.\\

\begin{remark}\label{rmk : X(rho) not in centralizer}
    We expect that the conjugating elements $g_x^t$ (following the notation of Corollary \ref{cor : hamiltonian flow for fully separating subsurface}) are usually not in the centralizer of $\rho(x)$ for $x\in\{a,b,c\}$. From the representative $u$ in Equation \eqref{eq : cocycle fully separating pair of pants figure 8 curve} for the Hamiltonian vector field, we see that $F_1’(\rho)$ is not in general in the centralizer of $\rho(a)$. Indeed, the centralizer $Z(\rho(a))$ contains all the matrices that are diagonal in the same basis in which $\rho(a)$ is diagonal. The matrix $\rho(ab^{-1})$ is diagonal, but not in the same basis as $\rho(a)$ (this follows from properties of the limit map of Hitchin representations). In particular, this means that in general, $u(a)\neq 0$. Similarly $u(b)\neq 0$ in general.
\end{remark}

\subsection{The case of $\SL{2}$}
In the case of $\SL{2}$, we can find the explicit flow by using trace identities special to $\SL 2$. Denote by $\LiftTeich$ the quotient by conjugation of the set of representations $\rho\colon\pi\to \SL{2}$ that are lifts to $\SL{2}$ of holonomy representations $\pi\to\PSL{2}$ of complete hyperbolic structures on $S$. The function in this case simplifies significantly since we can use the two trace identities
\begin{align}
\trace(gh^{-1}) &= \trace(g)\trace(h)-\trace(gh),\label{eq : trace identity product SL2}\\
\trace(g) &= \trace(g^{-1})\quad\textnormal{ for all }g,h\in\SL{2}\label{eq : trace is trace inverse SL2}.
\end{align}
Hence, the induced function on $\LiftTeich$ reduces to
\begin{align*}
f_\figeight([\rho]) &= f_{a}([\rho])f_{b}([\rho])-f_{c^{-1}}([\rho])\\
&= f_{a}([\rho])f_{b}([\rho])-f_{c}([\rho])
\end{align*}
where $[\rho]\in\LiftTeich$, and by the relation on the curves $a,b,c$, we have that $ab = c^{-1}$. Taking the differential, we have that for any $[\rho]\in\LiftTeich$ and a tangent cocycle $[u]\in H^1(\pi,\mathfrak{sl}(2,\R)_{\Ad\rho})$,
\[
d_{[\rho]}f_\figeight([u]) = f_b([\rho])d_{[\rho]}f_a([u]) + f_a([\rho])d_{[\rho]}f_b([u]) - d_{[\rho]}f_{c}([u]).
\]
By linearity of the symplectic form, we have that
\[
\Hm f_\figeight([\rho]) = f_b([\rho])\Hm f_a([\rho]) + f_a([\rho])\Hm f_b([\rho]) - \Hm f_{c}([\rho]).
\]
By \cite[proof of Theorem 4.3]{InvFct_Goldman}, since the curves $a,b,c$ are separating, we have that a representative cocycle for the Hamiltonian vector field is\footnote{It is a computation with $2\times 2$ matrices to see that this cocycle is cohomologous to the cocycle in \eqref{eq : cocycle fully separating pair of pants figure 8 curve} by adding to it the coboundary 
\begin{gather*}\gamma\mapsto (\identity - \Ad_{\rho(\gamma)})\left(f(\rho(a))F(\rho(b))+f(\rho(b))F(\rho(a))\right)\end{gather*} for every $\gamma\in\pi$.}
\[
u(\gamma) = \begin{cases}
    \left(\identity - \Ad_{\rho(\gamma)}\right)\left(-F(\rho(c))+f(\rho(a))F(\rho(b))\right),&\gamma\in \Gamma_a\\
    \left(\identity - \Ad_{\rho(\gamma)}\right)\left(-F(\rho(c))+f(\rho(b))F(\rho(a))\right),&\gamma\in \Gamma_b\\
    \left(\identity - \Ad_{\rho(\gamma)}\right)\left(f(\rho(a))F(\rho(b))+f(\rho(b))F(\rho(a))\right),&\gamma\in \Gamma_c.
\end{cases}
\]
By adding the coboundary $w(\gamma) = \left(\identity - \Ad_{\rho(\gamma)}\right)(F(\rho(c))-f(\rho(a))F(\rho(b))-f(\rho(b))F(\rho(a)))$ to $u$, we get that $u$ is cohomologous to the cocycle
\[
v(\gamma) = \begin{cases}
    \left(\identity - \Ad_{\rho(\gamma)}\right)\left(-f(\rho(b))F(\rho(a))\right),&\gamma\in \Gamma_a\\
    \left(\identity - \Ad_{\rho(\gamma)}\right)\left(-f(\rho(a))F(\rho(b))\right),&\gamma\in \Gamma_b\\
    \left(\identity - \Ad_{\rho(\gamma)}\right)\left(F(\rho(c))\right),&\gamma\in \Gamma_c.
    \end{cases}
\]
By Goldman's Poisson formula \cite[Theorem 3.5]{InvFct_Goldman}, we have that since the curves $a,b,c$ are simple and pairwise disjoint, the Hamiltonian flows of $f_a, f_b$ and $f_{c}$ pairwise commute and are twist flows (see Section 4 in \cite{InvFct_Goldman}). The Hamiltonian flow of $f_\figeight$ is therefore a product of pairwise commuting twist flows. By \cite[Theorem 4.3]{InvFct_Goldman}, we get the following.

\begin{corollary}\label{cor : Hamiltonian flow of figure 8 in Teich}
    Let $\delta$ be the figure $8$ curve inside a fully separating pair of pants in a surface $S$ as in Figure \ref{fig : Fully Separating pair of pants}. Then the Hamiltonian flow of the function $f_\delta = \trace_\delta$ on $\LiftTeich$ is covered by the flow of representation given by
    \[
\Xi_\rho(t)(\gamma) = \begin{cases}
    \exp(-tf(\rho(b))F(\rho(a)))\rho(\gamma)\exp(tf(\rho(b))F(\rho(a))),&\gamma\in \Gamma_a\\
    \exp(-tf(\rho(a))F(\rho(b)))\rho(\gamma)\exp(tf(\rho(a))F(\rho(b))),&\gamma\in \Gamma_b\\
    \exp(tF(\rho(c)))\rho(\gamma)\exp(-tF(\rho(c))),&\gamma\in \Gamma_c.
\end{cases}
\]
In particular, the flow is defined for all time $t\in\R$.
\end{corollary}

As opposed to Remark \ref{rmk : X(rho) not in centralizer}, we see that in this special case, the flow is a combination of twists along $\Gamma_a,\Gamma_b$ and $\Gamma_c$, i.e. the subgroups $\Gamma_x$ are conjugated by an element in the centralizer of $\rho(x)$. This can also be seen by the fact that the tangent cocycle $v$ evaluated at $x$ for each $x\in\{a,b,c\}$, is $0$. It is important to note here however, that the fact that $v$ vanishes on the boundary components can only be seen by picking a specific representative in the cohomology class (compare with Remark \ref{rmk : X(rho) not in centralizer}). This particular example is special to the $\SL 2$ case, since the trace identities \eqref{eq : trace identity product SL2} and \eqref{eq : trace is trace inverse SL2} do not hold in general in $\SL d$.

	\begin{figure}[ht]
		\centering
		\includesvg[scale=0.5]{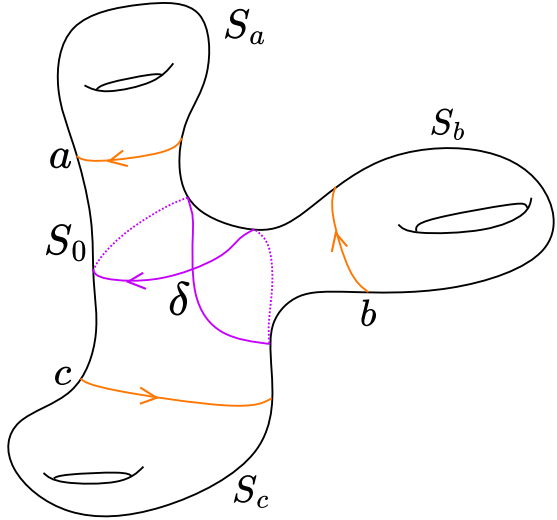}
		\caption{Fully separating pair of pants with figure $8$ curve.}
		\label{fig : Fully Separating pair of pants}
	\end{figure}

\subsection{Limits in Thurston's compactification of Teichm\"uller space}\label{sec : limits in thurstons compactification}
Corollary \ref{cor : Hamiltonian flow of figure 8 in Teich} gives us an explicit flow in $\LiftTeich$. By further quotienting the equivalence classes of representations to pass from $\SL 2$ to $\PSL 2$, we obtain a flow in the Teichm\"uller space $\Teich$, a connected component of $\mathscr X(\pi,\PSL 2)$. 
Note that since $a$ and $b$ lie in the commutator subgroup of $\pi_1(S)$ (since $S_0$ was assumed to be fully separating) $f_\delta$, $f_a$, $f_b$, and $f_c$ all descend to well-defined functions on $\Teich$ (as opposed to being defined up to sign).
With a flow on Teichm\"uller space, we may ask about the asymptotics of the flow.\\

Thurston in \cite{ThurstonCompactification} compactified Teichm\"uller space using the space of (positive) projective measured laminations, which we write as $\pml$. The space $\pml$ contains the (positive) projectivization of the space of weighted multicurves as a dense open set \cite[Proposition 4.2]{Bonahon86}.\\

If $\gamma$ is a simple closed curve, then the twist flow $\tw_t^\gamma\colon\Teich\to\Teich$ at time $t$ has the property that for every $[\rho]\in\Teich$,
\[
\tw_t^\gamma([\rho])\xrightarrow[]{t\to\pm\infty}[\gamma]\in\pml,
\]
where we use the notation $[\gamma]$ to be the projectivization of the multicurve $\gamma$ with weight one (see for example \cite[Proposition 8.2.38]{GeomTop_Martelli}). This is essentially because the lengths of curves intersecting $\gamma$ tend to infinity at a linear rate depending on the intersection number as $t$ tends to infinity and otherwise remain bounded.\\

Since the Hamiltonian flow of $f_\delta$ is a composition of pairwise commuting twist flows, we can deduce what the limit of the Hamiltonian flow in $\pml$ is. For the statement, let 
\[
\Delta(a,b,c)\coloneqq \left\{[x_1 a + x_2 b + x_3 c]\in\pml\st x_i\in\R_{\geq 0}\right\},
\]
and with interior given by
\[
\interior{\Delta(a,b,c)}\coloneqq \left\{[x_1 a + x_2 b + x_3 c]\in\Delta(a,b,c)\st x_i\in\R_{>0}\right\},
\]
where the square brackets denote projectivization.\\

By Corollary \ref{cor : Hamiltonian flow of figure 8 in Teich}, we see that the length of curves intersecting $a,b$ and $c$ grow at rates depending on which of $a,b$ or $c$ they intersect, and that the twist flow is defined for all time $t\in\R$. We can therefore deduce the following.

\begin{corollary}\label{cor : limits in thurstons compactification}
    Let $\delta$ be the figure $8$ curve inside a fully separating pair of pants in a surface $S$ as in Figure \ref{fig : Fully Separating pair of pants}. For each $[\rho]\in\Teich$, denote by $[\rho_t]$ the Hamiltonian flow of $f_\delta$ at $[\rho]$ and at time $t$.
    \begin{enumerate}
        \item The flow $[\rho_t]$ converges in $\pml$ as $t\to\pm\infty$ and has limit \[
        [\rho_t]\xrightarrow[]{t\to\pm\infty} [x_1 a+ x_2b + x_3c]\in \Delta(a,b,c),
        \]
        where
        \begin{gather*}
        x_1 = 2\cosh\left(\frac{\ell_{[\rho]}(b)}{2}\right)\sinh \left(\frac{\ell_{[\rho]}(a)}{2}\right), \quad x_2 = 2\cosh\left(\frac{\ell_{[\rho]}(a)}{2}\right)\sinh \left(\frac{\ell_{[\rho]}(b)}{2}\right),\\
        x_3 = \sinh \left(\frac{\ell_{[\rho]}(c)}{2}\right).
        \end{gather*}
        \item Conversely, for every $[x_1 a + x_2 b + x_3 c]\in\interior{\Delta(a,b,c)}$, there exists $[\rho]\in\Teich$ such that
        \[
         [\rho_t]\xrightarrow[]{t\to\pm\infty} [x_1 a + x_2 b + x_3c].
        \]
    \end{enumerate}
\end{corollary}
To see this, we use the following relation:
\[
\abs{\trace{(\rho(\gamma)})} = 2\cosh\left(\frac{\ell_{[\rho]}(\gamma)}{2}\right),
\]
where $\ell_{[\rho]}(\gamma)$ is the length of the curve $\gamma$ in the hyperbolic metric defined by $[\rho]$. The first part then follows by Corollary \ref{cor : Hamiltonian flow of figure 8 in Teich} noting that $\sinh \left(\frac{\ell}{2}\right) = v'(\ell)$, where $v(\ell) = 2\cosh\left(\frac{\ell}{2}\right)$. The second fact follows since in $\Teich$, the values $\ell_{[\rho]}(a),\ell_{[\rho]}(b),\ell_{[\rho]}(c)>0$ can be chosen freely in $\Teich$ (this follows for example by using Fenchel-Nielsen coordinates on $\Teich$ adapted to the pants decomposition containing the fully separating pair of pants containing $\delta$).

\begin{remark}
    In the case of the twist flow along a single simple closed curve, every $[\rho]\in\Teich$ converges to the same point in $\pml$ for both $t\to+\infty$ and $t\to-\infty$. In the case of the Hamiltonian flow of $f_\delta$, the flow also converges to the same point for both $t\to+\infty$ and $t\to-\infty$. However, we can see a different behaviour in terms of the limits, since the limit point depends on the representation one starts with.
\end{remark}

\bibliographystyle{hamsalpha}
\bibliography{main}
\end{document}